\documentclass[11pt,twoside]{amsart}

\usepackage{latexsym}
\usepackage{amsmath}
\usepackage{amsthm}
\usepackage{amssymb}
\usepackage{vmargin}
\usepackage{amscd}
\usepackage{stmaryrd}
\usepackage{euscript}
\usepackage{mathrsfs}
\usepackage{amscd}
\usepackage[all]{xy}
\usepackage{xr}

\DeclareMathAlphabet{\mathpzc}{OT1}{pzc}{m}{it}

\newtheorem{theorem}{Theorem}[section]
\newtheorem{proposition}[theorem]{Proposition}
\newtheorem{corollary}[theorem]{Corollary}

\newtheorem{lemma}[theorem]{Lemma}
\newtheorem*{theorem*}{Theorem}
\newtheorem*{proposition*}{Proposition}
\newtheorem*{corollary*}{Corollary}
\newtheorem*{lemma*}{Lemma}
\newtheorem*{conjecture*}{Conjecture}

\theoremstyle{definition}
\newtheorem{definition}[theorem]{Definition}

\newtheorem*{definition*}{Definition}

\theoremstyle{remark}
\newtheorem{example}[theorem]{Example}
\newtheorem{examples}[theorem]{Examples}
\newtheorem{remark}[theorem]{Remark}
\newtheorem{remarks}[theorem]{Remarks}

\newtheorem{properties}[theorem]{Properties}
\newtheorem{assumption}[theorem]{Assumption}

\newtheorem*{example*}{Example}
\newtheorem*{examples*}{Examples}
\newtheorem*{remark*}{Remark}
\newtheorem*{remarks*}{Remarks}
\newtheorem*{exercise*}{Exercise}
\newtheorem*{property*}{Property}
\newtheorem*{properties*}{Properties}

\newcommand\da{\!\downarrow\!}

\newcommand\la{\leftarrow}
\newcommand\lra{\longrightarrow}
\newcommand\lla{\longleftarrow}
\newcommand\id{\mathrm{id}}

\newcommand\ten{\otimes}

\newcommand\vareps{\varepsilon}
\newcommand\eps{\epsilon}

\newcommand\Th{\mathrm{Th}\,}

\renewcommand\H{\mathrm{H}}
\newcommand\z{\mathrm{Z}}

\newcommand\N{\mathbb{N}}
\newcommand\Z{\mathbb{Z}}
\newcommand\Q{\mathbb{Q}}

\newcommand\bA{\mathbb{A}}

\newcommand\bI{\mathbb{I}}

\newcommand\bL{\mathbb{L}}

\newcommand\bS{\mathbb{S}}

\newcommand\bU{\mathbb{U}}
\newcommand\bV{\mathbb{V}}

\newcommand\C{\mathcal{C}}

\newcommand\cA{\mathcal{A}}

\newcommand\cG{\mathcal{G}}

\newcommand\cI{\mathcal{I}}

\newcommand\cM{\mathcal{M}}

\newcommand\cS{\mathcal{S}}
\newcommand\cT{\mathcal{T}}
\newcommand\cU{\mathcal{U}}

\newcommand\cW{\mathcal{W}}

\renewcommand\O{\mathscr{O}}

\newcommand\sA{\mathscr{A}}
\newcommand\sB{\mathscr{B}}

\newcommand\sE{\mathscr{E}}
\newcommand\sF{\mathscr{F}}
\newcommand\sG{\mathscr{G}}
\newcommand\sH{\mathscr{H}}
\newcommand\sI{\mathscr{I}}

\newcommand\fS{\mathfrak{S}}

\newcommand\fX{\mathfrak{X}}
\newcommand\fY{\mathfrak{Y}}
\newcommand\fZ{\mathfrak{Z}}

\renewcommand\L{\Lambda}

\newcommand\n{\mathfrak{n}}

\newcommand\ext{\mathscr{E}\!\mathit{xt}}

\newcommand\Ho{\mathrm{Ho}}
\newcommand\Ring{\mathrm{Ring}}
\newcommand\Alg{\mathrm{Alg}}

\newcommand\Mod{\mathrm{Mod}}

\newcommand\Hom{\mathrm{Hom}}
\newcommand\Map{\mathrm{Map}}
\newcommand\map{\mathrm{map}}
\newcommand\HHom{\underline{\mathrm{Hom}}}

\newcommand\DDer{\underline{\mathrm{Der}}}

\newcommand\Ext{\mathrm{Ext}}
\newcommand\EExt{\mathbb{E}\mathrm{xt}}

\newcommand\Aut{\mathrm{Aut}}

\newcommand\Ch{\mathrm{Ch}}

\newcommand\Ab{\mathrm{Ab}}

\newcommand\hen{\mathrm{hen}}
\newcommand\loc{\mathrm{loc}}

\newcommand\Spec{\mathrm{Spec}\,}
\newcommand\Dec{\mathrm{Dec}\,}

\newcommand\Set{\mathrm{Set}}

\newcommand\Cat{\mathrm{Cat}}

\newcommand\Aff{\mathrm{Aff}}

\newcommand\Sp{\mathrm{Sp}}

\newcommand\Lim{\varprojlim}
\newcommand\LLim{\varinjlim}
\DeclareMathOperator*{\holim}{holim}
\newcommand\ho{\mathrm{ho}\!}
\newcommand\into{\hookrightarrow}
\newcommand\onto{\twoheadrightarrow}
\newcommand\abuts{\implies}
\newcommand\xra{\xrightarrow}

\newcommand\pr{\mathrm{pr}}

\newcommand\inj{\mathrm{inj}}

\newcommand\bt{\bullet}
\newcommand\by{\times}

\newcommand\et{\acute{\mathrm{e}}\mathrm{t}}

\newcommand\cart{\mathrm{cart}}

\newcommand\Tot{\mathrm{Tot}\,}
\newcommand\diag{\mathrm{diag}\,}

\newcommand\pro{\mathrm{pro}}

\newcommand\pd{\partial}

\newcommand\gpd{\mathrm{Gpd}}

\newcommand\Zar{\mathrm{Zar}}

\newcommand\sk{\mathrm{sk}}
\newcommand\cosk{\mathrm{cosk}}

\newcommand\op{\mathrm{opp}}

\newcommand\co{\colon\thinspace}

\newcommand\oC{\mathbf{C}}
\newcommand\oE{\mathbf{E}}
\newcommand\oR{\mathbf{R}}
\newcommand\oP{\mathbf{P}}
\newcommand\oL{\mathbf{L}}
\newcommand\oN{\mathbf{N}}

\newcommand\oSpec{\mathbf{Spec}\,}

\newcommand\uleft\underleftarrow
\newcommand\uline\underline
\newcommand\uright\underrightarrow

\newcommand\open{\mathrm{open}}
\newcommand\HOM{\mathrm{HOM}}

\begin{document}

\begin{abstract}
We show that an $n$-geometric stack may be regarded as a special kind of simplicial scheme, namely  a Duskin $n$-hypergroupoid in affine schemes, where surjectivity is defined in terms of covering maps,  yielding Artin $n$-stacks, Deligne--Mumford $n$-stacks and $n$-schemes as the notion of covering varies. This formulation  adapts to all  HAG contexts, so in particular works for derived $n$-stacks (replacing rings with simplicial rings). We exploit this to describe quasi-coherent sheaves and complexes on these stacks, and to draw comparisons with Kontsevich's dg-schemes. As an application, we show how the cotangent complex controls infinitesimal deformations of higher and derived stacks.
\end{abstract}

\title{Presenting  higher stacks as simplicial schemes}
\author{J.P.Pridham}
\thanks{This work was supported by the Engineering and Physical Sciences Research Council [grant numbers  EP/F043570/1 and EP/I004130/1].}
\maketitle

\tableofcontents

\section*{Introduction}
Although the usual  approach to defining $n$-stacks (\cite{hag2} and \cite{lurie}) undoubtedly  yields the correct geometric objects, it has numerous drawbacks. The inductive construction does not lend itself easily to calculations, while the level of abstract homotopy
theory involved can make $n$-stacks seem inaccessible to many.  In this paper, we introduce a far more elementary concept, namely a Duskin--Glenn $n$-hypergroupoid in affine schemes, and show how it is equivalent to the concept of $n$-geometric stacks introduced in \cite{hag2}. According to \cite{toenemail}, this is essentially the formulation of higher stacks originally envisaged by Grothendieck in \cite{pursuingstacks}. It is also closely related to Zhu's Lie $n$-groupoids. 

In \cite{glenn}, Glenn defined an    $n$-dimensional Kan hypergroupoid to be a simplicial set $X \in \bS$ for which the horn fillers all exist, and are moreover unique in levels greater than $n$. Explicitly, the $k$th horn $\L^m_k \subset \Delta^m$ is defined by deleting the interior and the $k$th face, and he required that the map
$$
\Hom_{\bS}(\Delta^m, X) \to\Hom_{\bS}(\L^m_k, X)   
$$
should be surjective for all $k,m$, and an isomorphism for $m>n$. 
When $n=0$, this gives sets (with constant simplicial structure), and for $n=1$, the simplicial sets arising in this way are precisely nerves of groupoids. This description is not as complicated as it first seems, since it suffices to truncate $X$ at the $(n+2)$th
%\textsuperscript{nd} 
level (Lemma \ref{truncate}). The combinatorial properties of these hypergroupoids were extensively studied by Glenn in \cite{glenn}.

It is well-known that  Artin $1$-stacks can be resolved by  simplicial schemes, and this idea was exploited in \cite{olssartin},  \cite{olssonstack}, \cite{aoki} and \cite{aokihom} to study  their quasi-coherent sheaves and deformations. It is therefore natural to expect that simplicial resolutions should exist for $n$-geometric Artin stacks.  In Theorem \ref{relstrict}, we show that every (quasi-compact, quasi-separated, \dots) $n$-geometric Artin stack $\fX$ (as defined in \cite{hag2}) can be resolved by a simplicial affine scheme $X$, which is an Artin $n$-hypergroupoid in the sense that the morphism
$$
\Hom_{\bS}(\Delta^m, X) \to\Hom_{\bS}(\L^m_k, X)   
$$
of affine schemes is a smooth surjection for all $k,m$, and an isomorphism for $m>n$. For $n$-geometric Deligne--Mumford stacks or $n$-geometric schemes there are similar statements, replacing smooth morphisms with \'etale morphisms or local isomorphisms.  In particular, this means that any functor satisfying the conditions of Lurie's Representability Theorem (\cite{lurie} Theorem 7.1.6) gives rise to such a simplicial scheme. 

If $\fX$ is a quasi-compact semi-separated scheme, this resolution just corresponds to constructing the \v Cech nerve of an affine cover. We may therefore think of  Artin $n$-hypergroupoids as being  analogues of atlases on manifolds. While they share with atlases an amenability to calculation, they also share the disadvantage of not being canonical. However, there is a notion of trivial relative Artin   
$n$-hypergroupoids $X' \to X$, which is analogous to refinement of an atlas, and we may regard two Artin $n$-hypergroupoids as being equivalent if they admit a common refinement of this type (Theorem \ref{duskinmor}). In this way, we can recover the entire $\infty$-category of $n$-geometric Artin stacks (Theorem \ref{bigthm}).

In fact, Theorem \ref{bigthm} gives a much more general statement, relating geometric stacks in any HAG context to the corresponding hypergroupoids. In particular, 
replacing rings with simplicial rings or dg-rings allows us to  define  derived Artin $n$-hypergroupoids, which are then  equivalent to  the $D^-$ geometric $n$-stacks of \cite{hag2}. \v Cech nerves allow us to make close comparisons between derived Artin stacks and Kontsevich's dg-schemes (Remark \ref{cechrks}.\ref{cechrk}).

Quasi-coherent sheaves on an $n$-geometric stack then just correspond to Cartesian quasi-coherent sheaves on the associated  $n$-hypergroupoid for any HAG context (Corollary \ref{qcohequiv}). This facilitates a relatively simple description of the cotangent complex, and our main application of this theory is to describe infinitesimal deformations of   $n$-geometric stacks in terms of the cotangent complex (Theorem \ref{deformstack}).

The structure of the paper is as follows. Sections \ref{hag} and \ref{nhyp} are mostly a recapitulation of background material, with a few results proved in \S \ref{nhyp} for which the author does not know of any other reference. Section \ref{nhypaff} introduces the main objects to be used in the paper, $n$-hypergroupoids in model categories (with the Artin $n$-hypergroupoids described above as a special case), and establishes their basic properties.

The technical heart of the paper is in Section \ref{ressn}, where we establish an equivalence between the simplicial categories of $n$-geometric stacks and of affine $n$-hypergroupoids. The crucial result is Theorem \ref{relstrict}, which uses an intricate induction taking $2^n-1$ steps to construct an $n$-hypergroupoid resolving a given $n$-geometric stack. 
Readers unfamiliar with the $n$-geometric stacks of \cite{hag2} can skip most of  this section and \S \ref{hag}, instead just  using Theorem \ref{duskinmor}  to define the stack $|Y|$ associated to an  $n$-hypergroupoid $Y$ as the functor
$$
\oR\HHom(-,|Y|): \Aff^{\op} \to \bS,
$$
 and  also to define the simplicial category of geometric stacks. Anyone wishing to make comparisons with \cite{lurie} should also read Remark \ref{cflurie} for  differences in terminology.

Section \ref{qucohsn} is dedicated to studying quasi-coherent sheaves. The equivalence between Cartesian quasi-coherent sheaves on an $n$-hypergroupoid and quasi-coherent sheaves on the associated  $n$-geometric  stack amounts to little more than cohomological descent.  Inverse images of sheaves are easily understood in terms of this comparison. However, (derived) direct images prove far more complicated, and an explicit description is given in \S \ref{directsn}.

In Section \ref{alternatives}, various alternative formulations of derived Artin hypergroupoids are developed. One of these builds on the Quillen equivalence   in characteristic zero between  simplicial algebras and  dg algebras, permitting comparisons with the dg-schemes of \cite{Quot}. The other formulations are based on the observation that in the homotopy category of simplicial rings, every object $R$ can be expressed as a filtered homotopy limit of homotopy nilpotent extensions of the discrete   ring $\pi_0R$. This allows us to replace any  cosimplicial affine scheme $X$ with Zariski or \'etale neighbourhoods $X^l$ or $X^h$ (or sometimes even the formal neighbourhood $\hat{X}$) of $\pi^0X$ in $X$.

Section \ref{dsheaves} looks at the results of Section \ref{qucohsn} in the context of derived Artin stacks. It then applies them to construct the relative cotangent complex of a morphism in \S \ref{cotsn}, compatibly with existing characterisations (Corollary \ref{cotgood} and Proposition \ref{cfolsson}).

In \S  \ref{defsn}, we then show how the cotangent complex governs deformations.
We first consider deformations of morphisms of derived Artin stacks in Theorem \ref{defmorstack}, generalising the deformations of $1$-morphisms of Artin $1$-stacks considered in \cite{aokihom}. We then consider deformations of derived Artin stacks in Theorem \ref{deformstack},  generalising the deformations of Artin $1$-stacks considered in \cite{aoki}.

I would like to thank Barbara Fantechi for raising many important and  interesting questions about the relation between  $n$-stacks and  simplicial schemes. I would also like to thank Bertrand To\"en for his 
helpful 
comments, Mathieu Anel for explaining properties of unique factorisation systems, and the referee for identifying many inaccuracies and imprecisions.

\subsection*{Notation and conventions}

We will denote the category of simplicial sets by $\bS$. A simplicial category is a category enriched in $\bS$. In other words, for all objects $X, Y \in \C$, there is a simplicial set $\HHom_{\C}(X,Y)$, and these $\Hom$-spaces are equipped with the usual composition laws for morphisms. The category underlying $\C$ has morphisms $\Hom_{\C}(X,Y):= \HHom_{\C}(X,Y)_0$, and the homotopy category $\Ho(\C)$ has morphisms $\Hom_{\Ho(\C)}(X,Y):= \pi_0\HHom_{\C}(X,Y)$.

A simplicial structure on a category $\C$ is sometimes given by defining $\ten: \bS \by \C \to \C$ or $(-)^{-}: \bS^{\op} \by \C \to \C$. In these cases, $\HHom_{\C}$ is determined by the formulae
$$
\Hom_{\C}(X\ten K, Y) =\Hom_{\bS}(K, \HHom_{\C}(X,Y))= \Hom_{\C}(X, Y^K).
$$

\begin{definition}
Let $\Delta^n \in \bS$ be the standard $n$-simplex, and $\pd\Delta^n\in \bS$ its boundary.  Given $0 \le k \le n$, define the $k$th horn $\L^n_k$ of $\Delta^n$ to be the simplicial set obtained from $\Delta^n$ by removing the interior and the $k$th face. See \cite{sht} \S I.1 for explicit descriptions. 
\end{definition}

\begin{definition}\label{N^s}
Given a simplicial abelian group $A_{\bt}$, we denote the associated normalised chain complex  by $N^sA$. Recall that this is given by  $N^s(A)_n:=\bigcap_{i>0}\ker (\pd_i: A_n \to A_{n-1})$, with differential $\pd_0$. Then $\H_*(N^sA)\cong \pi_*(A)$.
\end{definition}

\begin{definition}\label{N_c}
Given a cosimplicial abelian group $A^{\bt}$, we denote the associated conormalised cochain complex  by $N_cA$. Recall that this is given by  $N_c(A)^n:=\bigcap_{i}\ker (\sigma^i: A^n \to A^{n-1})$, with differential  $\sum (-1)^i\pd^i$.
\end{definition}

\begin{definition}\label{ox}
Given an affine scheme $X$, we will write $O(X):= \Gamma(X, \O_X)$.
\end{definition}

\begin{definition}
 Given a category $\C$, write $s\C$ for the category $\C^{\Delta^{\op}}$ of simplicial objects in $\C$. 
\end{definition}

\section{Background on $n$-geometric stacks}\label{background}\label{hag}

We now adapt various definitions from \cite{hag2} to slightly more general settings, and recall several results.

\subsection{A general setup}\label{setupsn}

Fix a model category $\cS$ in the sense of \cite{hovey} Definition 1.1.4. In particular, we assume that $\cS$ contains all small limits and colimits.

\subsubsection{Simplicial diagrams}

\begin{definition}\label{rKan}
Given a simplicial object  $X_{\bt}$ in a complete category $\C$, write $\Hom_{\bS}(-,X)\co \bS^{\op} \to \C$ for the right Kan extension of $X$ with respect to the Yoneda embedding $\Delta^{\op} \to \bS$. Explicitly, $\Hom_{\bS}(-,X)$ is the unique  limit-preserving functor determined by $\Hom_{\bS}(\Delta^n,X)=X_n$ together with functoriality for face and boundary maps.
\end{definition}

\begin{definition}\label{mn}
Given a simplicial set $K$ and a simplicial object  $X_{\bt}$ in a complete category $\C$, we follow \cite[Proposition VII.1.21]{sht} in defining the $K$-matching object in $\C$ by 
$$
M_KX:= \Hom_{\bS}(K, X).
$$
Thus $X_n= M_{\Delta^n}X$. Note that for finite simplicial sets $K$, the  matching object $M_KX$ still exists even if $\C$ only contains finite limits.
\end{definition}

\begin{remark}
The matching object $M_{\pd \Delta^n}X$ is usually denoted $M_nX$. In \cite{glenn}, it is called the $n$th simplicial kernel, and denoted $\Delta^{\bt}(n)(X)$, while $M_{\L^n_k}X$ is there denoted $\L^k(n)(X)$. 
\end{remark}

\begin{definition}\label{reedydef}
 Recall from \cite{hovey} Theorem 5.2.5 that for a model category $\C$, the Reedy model structure on $s\C$ is defined as follows. A morphism $f\co X_{\bt}\to Y_{\bt}$ in $s\C$ is said to be a Reedy fibration if the matching maps
$$
X_m \to M_{\pd \Delta^m} (X)\by_{M_{\pd \Delta^m} (Y)}Y_m 
$$
are fibrations in $s\C$, for all $m$, and a weak equivalence if each $f_n$ is a weak equivalence in $\C$.
\end{definition}

\begin{definition}\label{rKanh}\label{mnh}
For a model category  $\C$, write  $\oR\Hom_{\bS}(-,X)\co \bS^{\op} \to \C$ for the  homotopy right Kan extension of $X$. This is analogous to Definition \ref{rKan}, but preserves homotopy limits instead of limits, and can be realised as $\Hom_{\bS}(-, \oR X)$, where $\oR X$ is a Reedy fibrant replacement for $X$ in $s\C$. 

Defining the homotopy $K$-matching object in $\C$ by 
$$
M_K^hX:= \oR\Hom_{\bS}(K, X).
$$
\end{definition}

\subsubsection{$\vareps$-morphisms}\label{coversn}

We now consider $s\cS$  equipped with its Reedy model structure. 

\begin{definition}
 Define $|-|\co s\cS \to \cS$ to be realisation functor $\ho\LLim_{\Delta^{\op}}$.
\end{definition}

Fix a class $\vareps$  of morphisms in $\Ho(\cS)$, containing all weak equivalences and  stable under composition and homotopy  pullback. Refer to a morphism in $\cS$ as an $\vareps$-morphism if its image in $\Ho(\cS)$ is so (the corresponding notion is called  epimorphism in \cite{hag2} and covering in \cite{hag1}). We require $\vareps$-morphisms to satisfy the following properties:

\begin{properties}\label{coverprops}
\begin{enumerate}

\item\label{coverpsurj} If $X_{\bt} \in s\cS$, then the map $X_0 \to |X_{\bt}|$ is an $\vareps$-morphism.

 \item\label{coverptriv} Given a morphism $f\co X_{\bt}\to Y_{\bt}$ in $s\cS$ for which the homotopy matching maps $X_n \to M_n^hX\by_{M_n^hY}^h Y_n$ are $\vareps$-morphisms for all $n \ge 0$, the map 
\[
 | f|\co | X_{\bt}| \to | Y_{\bt}|
\]
is a weak equivalence in $\cS$.

\item\label{coverphprod} Given  morphisms $X_{\bt}\to Y_{\bt} \la Z_{\bt}$  in $s\cS$ for which the  homotopy partial matching maps $X_n \to M_{\L^n_k}^hX\by_{M_{\L^n_k}^hY}^h Y_n$ are $\vareps$-morphisms for all $n \ge 1$ and all $k$, the map 
\[
|(X_{\bt}\by^h_{Y_{\bt}}Z_{\bt})| \to | X_{\bt}|\by^h_{| Y_{\bt}|}|Z_{\bt}|
\]
is a weak equivalence in $\cS$.
\end{enumerate}
\end{properties}

Fix a full subcategory $\cA$ of $\cS$, closed under weak equivalences and finite homotopy limits. In the terminology of \cite{hag1} \S 4.1, $\cA$ is a pseudo-model category.

Fix a class $\oC$ of morphisms in $\Ho(\cA)$, containing all weak equivalences and  stable under composition and homotopy  pullback. Say that a morphism in  $\cA$ lies in $\oC$ if its image in $\Ho(\cA)$ does so. We require that $\cA$ and $\oC$ also satisfy the following properties

 \begin{properties}\label{Cprops}
 \begin{enumerate}

\item\label{coverpaffine} If $X \to Y$ is an $\vareps$-morphism in $\Ho(\cS)$ and we have a map $U \to Y$ for $U \in \Ho(\cA)$, then there is an $\vareps$-morphism $U' \to U$ in $\Ho(\cA)$ such that the composite map $U' \to Y$ in $\Ho(\cS)$ factors through $X$.

 \item Every morphism in $\oC$ is an $\vareps$-morphism.

\item Morphisms in $\oC$ are local with respect to $\vareps$-morphisms. In other words, if $f\co U \to V$ in $\cA$ has the property that the homotopy pullback $f' \co U\by_V^hV'\to V'$ is a $\oC$-morphism for some $\vareps$-morphism $V' \to V$, then $f$ is in $\oC$.
\end{enumerate}
\end{properties}

\subsection{Geometric stacks}\label{geomsn}

\begin{definition}\label{geomdef} 
Refer to the objects of $\cS$ as \emph{stacks}. Then we follow \cite{hag2} Definition 1.3.3.1 by defining $n$-geometric stacks (with respect to $(\cA, \cS, \oC)$)  as follows:
\begin{enumerate}

\item
 A stack is $0$-geometric if it lies in $\cA$.

\item For $n\ge 0$, a morphism of stacks $F \to G$ is $n$-representable if for any  $X \in \cA$ and any morphism $X \to G$ in $\Ho(\cS)$, the homotopy pull-back $F \by_G^h X$ is
$n$-geometric.

\item A morphism of stacks $f : F \to G$ is in $0-\oC$ if it is $0$-representable, and
if for any $X \in \cA$ and any morphism $X \to G$ in $\Ho(\cS)$, the induced
morphism
$
F \by^h_G X \to  X$
is in $\oC$.

\item Now let $n>0$ and let $F$ be any stack. An $n$-atlas for $F$ is a  morphism
$U \to F$ in $\Ho(\cS)$ such that $U$ is in $\cA$, and $ U \to F$ is in $(n -1)-\oC$.

\item For $n>0$, a stack $F$ is $n$-geometric if the diagonal morphism $F \to F\by F$ is $(n - 1)$-representable and the stack $F$ admits an $n$-atlas.

\item For $n > 0$,  a morphism of stacks $F \to G$ in $\Ho(\cS)$ is in $n-\oC$  if it is $n$-representable and if for any 
$X \in \cA$ and  any morphism $X \to G$ in $\Ho(\cS)$, there exists an $n$-atlas  $U \to F\by^h_G  X$, such
that $ U\to X$ is in $\oC$.
\end{enumerate}

Moreover, we follow \cite{hag2} Definition 1.3.3.7 in saying that a morphism  of stacks  is in $\oC$ (or a $\oC$-morphism) if it is in $n-\oC$ for some
integer $n$.
\end{definition}

\begin{lemma}\label{geomlocal}
Take a morphism $f\co X \to Y$ in $\Ho(\cS)$, together with an $\vareps$-morphism $Y' \to Y$. Then $f$ is $n$-representable (resp. in $n -\oC$) if and only if the homotopy pullback  $f' \co X\by^h_{Y}Y'$  is so.
\end{lemma}
\begin{proof}
 The ``only if'' part is standard (see for instance \cite{hag2} Proposition 1.3.3.3). The ``if'' part just uses the fact that $\oC$-morphisms are local with respect to $\vareps$-morphisms --- the argument of  \cite{hag2} Proposition 1.3.3.4 adapts to give the required result.
 \end{proof}

\begin{proposition}\label{hagcover}
   Let $f : F \to G$ be a $\oC$-morphism and $n\ge  1$. If $F$ is
$n$-geometric and $f$ is in $(n- 1)-\oC$, then $G$ is $n$-geometric.
\end{proposition}
\begin{proof}
 The proof of \cite{hag2} Corollary 1.3.4.5 carries over.
\end{proof}

\subsection{HAG contexts}\label{HAGcontext}

\subsubsection{Hyperdescent and hypersheaves}\label{inftysheaves}

\begin{definition}
Given a cosimplicial object $X^{\bt}$ in a model category $\cM$, define $\Tot_{\cM}X \in \Ho(\cM)$ by 
$$
\Tot_{\cM}X:= \holim_{\substack{ \lla \\n \in \Delta}} X^n.
$$
\end{definition}

\begin{definition}
Given a model category $\cM$ and a site $X$, an  $\cM$-valued presheaf $\sF$ on $X$ is said to be a hypersheaf if for all hypercovers $\tilde{U}_{\bt} \to U$, the natural map
$$
\sF(U) \to \Tot_{\cM} \sF(\tilde{U}_{\bt})
$$
is a weak equivalence. 
\end{definition}

\begin{remarks}
If $\cM$ has trivial model structure, note that hypersheaves in $\cM$ are precisely $\cM$-sheaves. A  sheaf $\sF$ of modules is a hypersheaf when regarded as a presheaf of non-negatively graded chain complexes, but not a hypersheaf  when regarded as a presheaf of unbounded chain complexes unless $\H^i(U, \sF)=0$ for all $i>0$ and all $U$. Beware that the sheafification of a hypersheaf will not, in general, be a hypersheaf.

Hypersheaves are  known as such in  \cite{lurie} \S 4.1, but are more usually called  $\infty$-stacks or hypercomplete $\infty$-sheaves, and are referred to as stacks in \cite{hag1} Definition 4.6.5. They are also sometimes known as homotopy sheaves, but we avoid this terminology for fear of  confusion  with homotopy groups of a simplicial sheaf. 
\end{remarks}

\begin{examples}\label{categs}
If $\cM$ is a category with trivial model structure (all morphisms are fibrations and cofibrations, isomorphisms are the only weak equivalences), then $\Tot_{\cM} X$ is the equaliser in $\cM$ of the maps $\pd^0, \pd^1 : X^0 \to X^1$.

$\Tot_{\bS}$ is the functor $\Tot$ of a cosimplicial space defined in \cite{sht} \S VIII.1.  Explicitly, 
$$
\Tot_{\bS} X^{\bt} =\{ x \in \prod_n (X^n)^{\Delta^n}\,:\, \pd^i_Xx_n = (\pd^i_{\Delta})^*x_{n+1},\,\sigma^i_Xx_n = (\sigma^i_{\Delta})^*x_{n-1}\}, 
$$ 
for $X$ Reedy fibrant. Homotopy groups of the total space are related to a spectral sequence given in  \cite{sht} Proposition VIII.1.15.

In the category $s\Ab$ of simplicial abelian groups, all cosimplicial objects are Reedy fibrant, and  $\Tot_{s\Ab}A= \Tot_{\bS}A$. Likewise for simplicial modules and simplicial rings. In all  these cases, we may compute homotopy groups by $\pi_n(\Tot_{s\Ab}A) = \H_n (\Tot_{\Ch_{\ge 0}} N^sA)$, for $N^sA$ the normalised cochain complex associated to $A$, and $\Tot_{\Ch_{\ge 0}}$ defined below. The unnormalised associated chain complex also gives the same result.

For the category $\Ch$ of unbounded chain complexes, $\Tot_{\Ch}\simeq \Tot^{\Pi}$, or the quasi-isomorphic functor $\Tot^{\Pi}N_c$, where $N_c$ is conormalisation (see Definition \ref{N_c}), and $(\Tot^{\Pi} V)_n= \prod_i V^i_{n+i}$. Similarly, $\Tot_{\mathrm{CoCh}^{\ge 0}}\simeq \Tot^{\Pi}$  for the category $\mathrm{CoCh}^{\ge 0}$ of non-negatively graded cochain complexes. However, for non-negatively graded chain complexes $\Ch_{\ge 0}$, we have $\Tot_{\Ch_{\ge 0}} \simeq \tau_{\ge 0} \Tot^{\Pi}$, where $\tau$ is good truncation.

For the categories $dg\Alg$ and $dg_{\le 0}\Alg$ of unbounded and non-positively graded  differential graded (chain) algebras, $\Tot_{dg\Alg}$ and $\Tot_{dg_{\le 0}\Alg}$ are  modelled by  the functor $\Th$ of Thom-Sullivan (or Thom-Whitney) cochains. In the category $dg_{\ge 0}\Alg$ of differential non-negatively graded (chain)   algebras, $\Tot_{dg_{\ge 0}\Alg} \simeq \tau_{\ge 0} \Th$.
\end{examples}

\subsubsection{HA contexts and HAG contexts}\label{HAHAG}

By \cite{hag2} Definition 1.1.0.11, a homotopical algebra (or HA) context is a symmetric monoidal model category $\C$, endowed with some additional structures which will not affect our ability to define $n$-geometric stacks. We write $\Aff_{\C}$ for the opposite category to the category of commutative monoids in $\C$.

By \cite{hag2} Definition 1.3.2, a homotopical algebraic geometry (HAG) context is a HA context $\C$, together with a model pre-topology $\tau$ having various good properties summarised in \cite{hag2} Assumption 1.3.2.2, and a class $\oP$ of morphisms in $\Aff_{\C}$ satisfying the conditions of \cite{hag2} Assumption 1.3.2.11.

The setting for $n$-geometric stacks is the model category  $\Aff_{\C}^{\sim, \tau}$ of stacks (defined as in \cite{hag2} \S 1.3.2). The underlying category of $\Aff_{\C}^{\sim, \tau}$ is the category of simplicial presheaves on $\Aff_{\C}$, and every fibrant object is a hypersheaf  with respect for the topology  $\tau$. Morphisms between hypersheaves are weak equivalences if and only if they induce isomorphisms on the sheafifications of the homotopy groups.

\begin{definition}
 As in \cite{hag2} \S 1.3.1, we have a Yoneda embedding 
\[
 \oR \uline{h}\co \Aff_{\C} \to \Aff_{\C}^{\sim, \tau}
\]
given by
\[
  (\oR \uline{h}X)(U) = \map_{\Aff_{\C}}(U,X),    
\]
where $\map$ is a functorial choice of homotopy function complex (as in \cite{hirschhorn} \S 17.4). Specifically, \cite{hag2} \S 1.3.1 takes a left homotopy function complex (denoted by $\Map_l$ in \cite{hovey} \S 5.2).
In a simplicial model category, we can just take $\map = \oR \HHom$, the right-derived functor of the simplicial $\Hom$-functor. 
\end{definition}

\begin{lemma}\label{Rhff}
For $X,Y \in \Aff_{\C}$, the canonical map
\[
 \map_{ \Aff_{\C}}(X,Y) \to   \oR\HHom_{ \Aff_{\C}^{\sim, \tau}}  (\oR \uline{h}X , \oR \uline{h}Y)       
\]
 is a weak equivalence.
 \end{lemma}
\begin{proof}
By \cite{hag2} Corollary 1.3.2.5, the presheaf $ \oR \uline{h}X$ is a hypersheaf, so the result is just \cite{hag1} Proposition 4.6.7.       
\end{proof}

\subsubsection{Geometric stacks}

\begin{definition}\label{hagndef}
The simplicial category of strongly  quasi-compact $n$-geometric stacks for the HAG context $(\C, \tau, \oP)$ is given by applying the construction of  \S \ref{geomsn} to the data $(\cA, \cS, \vareps, \oC)$ for $\cS:=\Aff_{\C}^{\sim, \tau}$, with $\cA$ consisting of objects in the essential image of $\oR \uline{h}$, $\vareps$ the class of local surjections (epimorphisms in the terminology of \cite{hag2} Remark 1.3.2.6),   and $\oC$ consisting of those $\oP$-morphisms in $\cA$ which are also $\vareps$-morphisms. 

For the simplicial category of all  $n$-geometric stacks, we leave $\cS$ unchanged, but take  $\cA$ to consist of objects of the form $\coprod_i^{\oL} \oR \uline{h} X_i$, for sets $\{X_i\}_i$ of objects of $\Aff_{\C}$. We then let $\oC$  consist of $\vareps$-morphisms $\coprod_i^{\oL}  U_i \to X$  for  families $\{U_i\to X\}_i $ of  representable
$\oP$-morphisms.
\end{definition}

We now verify the conditions of \S \ref{setupsn}. 

\begin{proposition}\label{haghprod}
The setting $(\cA,\cS, \oP, \vareps)$ above  satisfies Properties \ref{coverprops} for the class of $\vareps$-morphisms. The class $\oC$ satisfies Properties \ref{Cprops}.
\end{proposition}
\begin{proof}
Properties \ref{Cprops} are immediate, because the topology $\tau$ is defined using coverings in $\Aff_{\C}$ and the property $\oP$ is, by assumption,  local with respect to $\tau$. 

Now, let $\cT$ be the simplicial localisation of $\Aff_{\C}$ with respect to the weak equivalences (as constructed in  \cite{simploc2}). This has the properties that $\pi_0\cT= \Ho(\Aff_{\C})$, and that $\HHom_{\cT}(X,Y) \simeq \map_{\Aff_{\C}}(X,Y)$. 
By \cite{hag1} Theorem 4.7.1, the  model category $\Aff_{\C}^{\sim, \tau} $ is Quillen equivalent to the model category $s\Pr_{\tau}(\cT)$ of \cite{hag1} \S 3.4, and hence to the model category $s\Pr_{\tau, \inj}(\cT)$ of \cite{hag1} \S 3.6. The equivalence is given by sending a fibrant object  in  $\Aff_{\C}^{\sim, \tau} $ to its associated simplicial functor. Objects of the model category  $\cI:= s\Pr_{\tau, \inj}( \cT)$ are simplicial functors $\cT \to \bS$, cofibrations are the objectwise monomorphisms, and weak equivalences are the local weak equivalences. 

For Property \ref{coverprops}.\ref{coverpsurj}, take $X_{\bt} \in s\cS$. Replacing $X_{\bt}$ with a weakly  equivalent diagram if necessary, we may assume that $X_{\bt} \in s\cI$, with each $X_n$ fibrant (hence a stack). 
Now, $\ho\LLim_{\Delta^{\op}}\co s\bS \to \bS$ is given by the diagonal functor $(\diag X)_n =X_{nn}$,  so the same is true for the injective model structure on simplicial presheaves. Therefore
\[
 |X_{\bt}| \simeq a(\diag X_{\bt}),
\]
where $a$ is the fibrant replacement functor in $\cI$.

Now, by  \cite{sht}  Proposition IV.1.7 the functor $\diag$ takes levelwise weak equivalences to weak equivalences. Since each $X_n$ is a stack,  this means that $ \diag X_{\bt}$ sends weak equivalences in $\Aff_{\C}$ to weak equivalences in $\bS$ --- in the terminology of \cite{hag2}, it is a prestack. Also note that the map $X_0 \to \diag X_{\bt}$ induces surjections $\pi_0(X_0(U)) \onto \pi_0(\diag X_{\bt}(U)) $ for all $U\in \Aff_{\C}$.  This necessarily induces surjections on the associated presheaves of sets, so   by \cite{hag1} Definition 4.3.7, $X_0 \to a(\diag X_{\bt})$ is an $\vareps$-morphism, as required.

For Property \ref{coverprops}.\ref{coverptriv}, we may take a Reedy fibration $f\co X_{\bt} \to Y_{\bt}$ between Reedy fibrant objects in  $s\cI$. Now, a fibration in $\cI$ is a fibration objectwise, so a fibration is an $\vareps$-morphism if and only if it is locally a surjective map of simplicial sets.
Since $f$ is a Reedy fibration, the homotopy matching object $M_n^hX\by^h_{M_n^hY}Y_n$ is represented by the ordinary matching object $M_nX\by_{M_nY}Y_n$, and the condition that the map from $X_n$ to this be an $\vareps$-morphism says precisely that $f$ is locally a  horizontal trivial Kan fibration. 
 In particular, it is a horizontal local weak equivalence, so $\diag X_{\bt} \to \diag Y_{\bt}$ is a local weak equivalence. Applying $a$ gives $|X_{\bt}| \simeq |Y_{\bt}|$, as required.

Finally, for Property \ref{coverprops}.\ref{coverphprod}, we may once again take Reedy fibrations $X_{\bt} \xra{f} Y_{\bt}\la Z_{\bt}$ between Reedy fibrant objects in  $s\cI$. The homotopy partial matching objects are now ordinary matching objects, so the condition that the homotopy partial matching maps of $f$ be  $\vareps$-morphisms says precisely that $f$ is a locally a horizontal Kan fibration. Then  \cite{sht} Lemma IV.4.8 implies that $\diag f$ is a local fibration, and now \cite{jardineLocal} Lemma 5.16 (adapted to simplicial sites) shows that
\[
 (\diag X)\by^h_{\diag Y}(\diag Z) \simeq (\diag X)\by_{\diag Y}(\diag Z).
\]
The right-hand side is just $\diag(X\by_YZ)$,  and since 
we started with Reedy fibrations between fibrant objects, this is a model for $\diag( X\by^h_YZ)$. 
Thus we have shown 
\[
 |X|\by^h_{|Y|}|Z| \simeq |X\by^h_YZ|,
\]
 as required.

\end{proof}

\subsubsection{Higher Artin and Deligne--Mumford stacks}\label{hagn}

The first main example of a HA context is to take $\Aff_{\C}$ to be the category of affine schemes with its trivial model structure. In the trivial  model structure, all morphisms are both fibrations and cofibrations, and the weak equivalences are just isomorphisms. 
Then $\cA$ is Quillen equivalent to  $\Aff_{\C}$  and 
$\Aff_{\C}^{\sim, \tau}$ is the category of simplicial presheaves on $\Aff_{\C}$ with its local projective model structure. 
Taking $\tau$ to be the \'etale topology,
the fibrant objects of $\Aff_{\C}^{\sim, \tau}$ are hypersheaves for the \'etale site. The functor $\oR\uline{h}\co \Aff_{\C} \to \cS$ just sends $X$ to its associated sheaf
\[
 U \mapsto \Hom_{\Aff_{\C}}(U,X).       
\]

There are two common choices of HAG context associated to this HA context, given by taking $\oP$ to be the class either of smooth morphisms or of \'etale  morphisms. In these contexts, the $(n, \oP)$-geometric stacks of \S \ref{geomsn} are known as $n$-geometric Artin and   Deligne--Mumford stacks, respectively. A rarer context is that of $n$-geometric schemes, which takes $\oP$ to be the class of local isomorphisms.

\subsubsection{Derived Artin stacks}\label{hagd}

We now recall the geometric $D^-$-stacks of \cite{hag2} Chapter 2.2. Set $\Aff_{\C}$ to be $c\Aff$, the category of cosimplicial affine schemes, so $\C$ is the category of simplicial algebras. Given a simplicial algebra $A$, write $\Spec A$ for the corresponding object of $c\Aff$. Let $s\Pr(c\Aff)$ denote the category of simplicial presheaves on $c\Aff$.

We give $c\Aff$ its usual simplicial model structure, in which a map $\Spec B \to \Spec A$ is a weak equivalence if $\pi_n(A) \cong \pi_n(B)$ for all $n$.  The map is cofibration if $A \to B$ is a Kan fibration, or equivalently if $N_i^sA \to N_i^sB$ is surjective for all $i>0$, for $N^s$ the simplicial normalisation of Definition \ref{N^s}. 

\begin{definition}\label{sstr}
The Reedy model category $sc\Aff$ has a natural simplicial structure. If we regard an object $X \in sc\Aff$ as a contravariant functor from $c\Aff$ to $\bS$, then for $K \in \bS$, the object $X^K \in sc\Aff$ is given by
$$
X^K(U):= X(U)^K,
$$ 
for $U \in c\Aff$. Note that this simplicial structure is independent of the natural simplicial structure on $c\Aff$. 

We then define $\HHom_{sc\Aff}(X,Y) \in \bS$ by 
$$
\HHom_{sc\Aff}(X,Y)_n:= \Hom_{sc\Aff}(X,Y^{\Delta^n}).
$$
\end{definition}

\begin{definition}\label{dsmooth}
A map  $\Spec B \to \Spec A$ in $c\Aff$ is said to be flat (resp. smooth, resp. \'etale, resp. a local isomorphism) if $\pi_0(A) \to \pi_0(B)$ is flat (resp. smooth, resp. \'etale, resp. a local isomorphism), and $\pi_n(B)\cong \pi_n(A)\ten_{\pi_0(A)}\pi_0(B)$ for all $n$. 
\end{definition}

\begin{definition}\label{h0def}
Define the functor $\pi^0:c\Aff \to \Aff$ by $\Spec A \mapsto \Spec \pi_0A$. A morphism $f$ in $c\Aff$  is said to be surjective if $\pi^0(f)$ is surjective.
\end{definition}

\begin{definition}\label{uline}
Given a fibrant object $X \in c\Aff$, define $\underline{X} \in s\Pr(c\Aff)$ by
$$
\underline{X}(U):= \HHom_{c\Aff}(U, X)\in \bS,
$$
for $U \in c\Aff$.

\end{definition}

In the notation of \S \ref{HAHAG}, $\uline{X}$ is the stack $\oR \uline{h} X$.
 Thus an object of $s\Pr(c\Aff)$ is a $0$-geometric stack if and only if it is weakly equivalent to a disjoint union of such stacks $\underline{X}$. 

For the  HAG context given by taking $\oP$ to be the class of smooth morphisms (resp. \'etale morphisms, resp.  local isomorphisms) and $\tau$ the \'etale topology,  we will refer to the associated  geometric stacks   as derived Artin (resp. Deligne--Mumford, resp. Zariski) $n$-geometric stacks. In the terminology of \cite{hag2}, the  derived Artin $n$-geometric stacks are called $n$-geometric $D^-$-stacks.

\subsection{Other examples}

Another important example of a HAG context given in \cite{hag2} is that of symmetric spectra, the setting of spectral (or brave new) algebraic geometry. The precise formulation given depends on the choices $\tau, \oP$ of topology and morphisms. 

If our sole aim is to define $n$-geometric stacks, then HAG contexts might be regarded as unnecessarily restrictive. Rather than taking $\cA$ to be a category opposite to commutative monoids, Definition \ref{geomdef} will still work for non-commutative rings or for simplicial $\C^{\infty}$-rings, giving non-commutative and differential stacks respectively. In both cases, the main difficulty lies in finding a suitable class $\oP$ of morphisms. 

It is also worth noting that the simplicial category $\cA$ in Definition \ref{geomdef} need not be the localisation of a model category. This makes it possible to consider localisations of small categories, such as of finitely generated algebras. A better-known example of this type is given by Lurie's $n$-stacks in \cite{lurie}, where $\cA$ is taken to be the category of algebraic spaces and $\oC$ the category of smooth surjections. 

%If sufficient care is taken, it is possible to see that the hypothesis that $\cA$ contain all finite homotopy limits is slightly stronger than needed for the definitions. It suffices for $\cA$ to contain pullbacks by $\oC$-morphisms. By taking $\cA$ to be the category of manifolds and $\oC$ the class of surjective submersions, this should yield an alternative theory of non-singular differential stacks. 

We now collect some useful lemmas and a consequence for comparing the stacks of \cite{hag2, lurie}.

\begin{definition}
Given $\fX \in \cS$ and $K \in\bS$, denote by $\fX^{\oR K}$ the object of $\cS$ defined as the homotopy limit $M^h_K(c\fX)$ of Definition \ref{mnh}, where $c\fX \in s\cS$ is given by $(c\fX)_n= \fX$ for all $n$.
\end{definition}

\begin{lemma}\label{higherdiagchar}
Given a morphism $f:\fX \to \fY$ of stacks,  the diagonal of the matching morphism  
$$
\fX \simeq \fX^{\oR \Delta^{r-1}} \to \fX^{\oR \pd \Delta^{r-1}}\by_{\fY^{\oR \pd \Delta^{r-1}}}^h\fY.
$$
is equivalent  to 
$$
\fX \to \fX^{\oR \pd \Delta^r}\by_{\fY^{\oR \pd \Delta^r}}^h\fY
$$
\end{lemma}
\begin{proof}
The key observation is that we may express $\pd \Delta^r$ as the pushout
$$
\pd \Delta^r = \L^r_0 \cup_{\pd^0, \pd \Delta^{r-1}}\Delta^{r-1}.
$$
Since this is a pushout of cofibrations, we have
$$
\fX^{\oR \pd \Delta^r} \simeq \fX^{\oR \L^r_0}\by^h_{\fX^{\oR \pd \Delta^{r-1}}}\fX^{\oR \Delta^{r-1}} \simeq \fX\by^h_{\oR \fX^{\pd \Delta^{r-1}}}\fX,
$$
the second equivalence following because $\L^r_0$ and $\Delta^{r-1}$ are contractible.

This implies that the diagonal
$$
\fX \to (\fX\by^h_{\fX^{\oR \pd \Delta^{r-1}} }\fX)\by^h_{(\fY\by^h_{\fY^{\oR \pd \Delta^{r-1}} }\fY)}\fY
$$
of  $\fX \to \fX^{\oR \pd \Delta^{r-1}}\by_{\fY^{\oR \pd \Delta^{r-1}}}^h\fY$ is just
$$
\fX \to \fX^{\oR \pd \Delta^r}\by_{\fY^{\oR \pd \Delta^r}}^h\fY.
$$
\end{proof}

\begin{lemma}\label{higherdiag}
Given an $n$-representable morphism $\fX \to \fY$ of stacks, the matching  map
$$
\fX \to \fX^{\oR \pd \Delta^r}\by_{\fY^{\oR \pd \Delta^r}}^h\fY
$$
is $(n-r)$-representable for all $0 \le r \le n$, a $0$-representable closed immersion for $r=n+1$, and an equivalence for $r > n+1$. 
\end{lemma}
\begin{proof}  
We prove this by induction on $r$. When $r=0$, $\pd \Delta^0=\emptyset$, so the statement is trivially true. 
Now assume the statement for $r-1$, with $r\le n$. Since  $\fX \to \fX^{\oR \pd \Delta^{r-1}}\by_{\fY^{\oR \pd \Delta^{r-1}}}^h\fY$ is $(n-r+1)$-representable,  its diagonal must be  $(n-r)$-representable, giving the required result by Lemma \ref{higherdiagchar}. The statement for $r=n+1$ follows because the diagonal of an affine morphism is a closed immersion. For $r>n+1$, observe that the diagonal of a closed immersion is an isomorphism.
\end{proof}

\begin{remark}\label{cflurie}
When we take $\oP$ to be the class of smooth morphisms, there are slight differences in terminology between \cite{hag2} and \cite{lurie} in relation to higher stacks. In the former, only disjoint unions of affine schemes are $0$-representable, so arbitrary schemes are $2$-geometric stacks, and Artin stacks are $1$-geometric stacks if and only if they have affine diagonal. In the latter, algebraic spaces are $0$-stacks.  An $n$-stack in the sense of \cite{lurie} is called $n$-truncated in \cite{hag2}, and it follows easily that every $n$-geometric stack in \cite{hag2} is $n$-truncated.

Conversely, for any  $n$-truncated stack $X$, the map $X \to X^{\oR \pd \Delta^{n+2}}$ is an isomorphism (by Lemma \ref{higherdiag}), hence $0$-representable. Thus $X$   must be  $(n+2)$-geometric.

We can summarise this by saying that for a geometric stack $\fX$ to be $n$-truncated means that $\fX \to \fX^{\oR \pd \Delta^{n+2}}$ is an equivalence, or equivalently that $\fX \to \fX^{\oR \pd \Delta^{n}}$ is representable by algebraic spaces. For $\fX$ to be $n$-geometric means that $\fX \to \fX^{\oR \pd \Delta^{n}}$ is representable by disjoint unions of affine schemes.

As an example of this difference, consider the scheme $X$ obtained by gluing two copies of $\bA^2$, with intersection  $\bA^2-\{0\}$. This is certainly an algebraic space, so is $0$-truncated. However, the diagonal $X \to X\by X$ is not an affine morphism, so $X$ is not even $1$-geometric, but merely $2$-geometric.

Nevertheless, any Artin stack with affine diagonal (in particular any separated algebraic space) is $1$-geometric.
\end{remark}

\section{$n$-hypergroupoids}\label{nhyp}

\begin{definition}\label{duskindef}
An $n$-hypergroupoid  is an object $X \in \bS$ for which the partial matching maps
$$
X_m \to M_{\L^m_k} X 
$$
are surjective for all $k,m$ (i.e. $X$ is a Kan complex), and are isomorphisms for all $m> n$.
\end{definition}

\begin{remark}
 Terminology for $n$-hypergroupoids varies considerably. Other names include  ``$n$-dimensional Kan hypergroupoids'' (\cite[\S 3]{glenn}), ``weak $n$-groupoids'' (\cite[Definition 2.3]{getzler}) and ``exact $n$-types'' (\cite[Definition 3.3]{bekecech}).  The relative notion of Definition \ref{relhyp} below is called a ``fibration exact in $\dim \ge n$'' in \cite[Definition 1.5]{duskinsheaf}).
\end{remark}

\begin{remarks}
\begin{enumerate}
\item
A $0$-hypergroupoid is just a set $X=X_0$. A $1$-hypergroupoid is the nerve of a groupoid $G$, which can be recovered  by taking objects $X_0$, morphisms $X_1$, source and target $\pd_0, \pd_1: X_1 \to X_0$, identity $\sigma_0 : X_0 \to X_1$ and multiplication
$$
X_1 \by_{\pd_0, X_0, \pd_1} X_1 \xra{(\pd_2,\pd_0)^{-1}} X_2 \xra{\pd_1} X_1.
$$
Equivalently, $G$ is the fundamental groupoid $\pi_fX$ of $X$.

\item For an $n$-hypergroupoid $X$, the factorisation $\L^m_k \to \pd \Delta^m \to \Delta^m$ ensures that for all $m > n$, $X_m \to M_{\pd\Delta^m}$ has a section, so is surjective. This implies that $\pi_mX=0$ for all $m>n$.

\item Giving an $n$-hypergroupoid $X$ is equivalent to giving the sets $X_i$ for $i \le n$, together with the operations between them, and an operation on certain $(n+1)$-tuples in  $X_n$, satisfying the hyper-associativity and hyper-unit laws of \cite{glenn} \S3.2. A related result will be given in Lemma \ref{truncate}.

\item Under the Dold--Kan correspondence between $\N_0$-graded chain complexes and abelian groups, $n$-hypergroupoids in abelian groups correspond to chain complexes concentrated in degrees $\le n$. This is essentially because the normalisation functor is given by $N_nA = \ker(A_n \to M_{\L^n_0}A)$. 
\end{enumerate}
\end{remarks}

We also have a relative version of this definition.
\begin{definition}\label{relhyp}
Given   $Y\in \bS$, define a (relative) $n$-hypergroupoid  over $Y$ to be a morphism $f:X\to Y$ in $\bS$, such that  the   relative partial matching maps
$$
X_m \to M_{\L^m_k} (X)\by_{M_{\L^m_k} Y}Y_m 
$$
are surjective for all $k,m$ (i.e. $f$ is a Kan fibration), and isomorphisms for all $m> n$. In the terminology of \cite{glenn}, this says that $f$ is a Kan fibration which is an exact fibration in all dimensions $>n$.
\end{definition}

\begin{examples}
\begin{enumerate}
\item A  morphism $f:X \to Y$ between  $1$-hypergroupoids $X,Y$ makes $X$  a    $1$-hypergroupoid  over $Y$ if and only if it satisfies the path-lifting property that for all objects $x \in X_0$, and all morphisms $m$ in $\pi_fY$ with source $fx$, there exists a unique morphism $\tilde{m}$ in $\pi_fX$ with source $x$ and $f(\tilde{m})=m$.

\item A  $0$-hypergroupoid $f:X\to Y$ is just a space which is Cartesian over $Y$, in the sense that the maps
$$
X_n\xra{(\pd_i, f)} X_{n-1}\by_{Y_{n-1}, \pd_i}Y_n
$$
are all isomorphisms.
\end{enumerate}
\end{examples}

\begin{lemma}
Given a Kan fibration $f:X \to Y$ and a morphism $g:Y \to Z$ such that  $gf: X\to Z$ is  an $n$-hypergroupoid over $Z$,  the morphism $f$ must be an $n$-hypergroupoid. 
\end{lemma}
\begin{proof}
Existence of the lifts follows from $f$ being a fibration, and uniqueness follows from uniqueness for $gf$.
\end{proof}

\begin{definition}\label{skdef}
Let $\Delta_n \subset \Delta$ be the full subcategory on objects $\mathbf{i}$, for $i \le n$. The category of $n$-truncated simplicial sets is $\Set^{\Delta_n^{\op}}$. The forgetful functor $i_{n*}$ from $\bS$ to $\Set^{\Delta_n^{\op}}$ has a left adjoint $i_n^*$, and a right adjoint $i_n^!$.  The composites are denoted $\sk_n:= i_n^*i_{n*}$ and $\cosk_n:= i_n^!i_{n*}$, giving $\Hom(K, \cosk_nX)= \Hom(\sk_nK, X)$. We sometimes denote $i_n^*X$ by $X_{\le n}$.
\end{definition}

\begin{definition}\label{trelhyp}
Given   $Y\in \bS$, define a trivial    $n$-hypergroupoid  over $Y$ to be a morphism $f:X\to Y$ in $\bS$, such that  the   relative  matching maps
$$
X_m \to M_{\pd \Delta^m} (X)\by_{M_{\pd \Delta^m} Y}Y_m 
$$
are surjective for all $m$ (i.e. $f$ is a trivial Kan fibration), and isomorphisms for all $m\ge n$. 
\end{definition}

Note that this gives the following lemma immediately.
\begin{lemma}\label{ttruncate}
A morphism $f:X\to Y$ is a trivial    $n$-hypergroupoid over $Y$ if and only if $X= Y\by_{\cosk_{n-1}Y}\cosk_{n-1}X$, and the $(n-1)$-truncated morphism $X_{\le n-1}\to Y_{\le n-1}$ satisfies the conditions of Definition \ref{trelhyp} (up to level $n-1$).
\end{lemma} 

\begin{lemma} 
A morphism $f$ is  a trivial   $n$-hypergroupoid if and only if it is an  $n$-hypergroupoid and a weak equivalence. 
\end{lemma}
\begin{proof}
The ``only if'' part is immediate. For the converse, note that $f$ is a trivial fibration, so has the right lifting property (RLP) with respect to all cofibrations, and consider the maps $\L^i_0 \to \pd\Delta^i \to \Delta^i$. Since $f$ has exact RLP with respect to $\L^i_0 \to \Delta^i$ for $i>n$, and RLP with respect to $\pd\Delta^i \to \Delta^i$ for all $i\ge n$, it follows that $f$ has exact RLP with respect to $\pd\Delta^i \to \Delta^i$ for all $i>n$. It therefore remains only to prove uniqueness of the lift for $i=n$. 

Consider the pushout $\pd \Delta^{n+1}=\L^{n+1}_0\cup_{\pd \Delta^n}\Delta^n$. Since $\pd\Delta^n \to \L^{n+1}_0$ is a cofibration, exact RLP with respect to $\pd\Delta^n \to \Delta^n$ will follow from exact RLP with respect to $\L^{n+1}_0\to \pd\Delta^{n+1}$. This now follows from exact RLP with respect to $\L^{n+1}_0\to \Delta^{n+1}$, since both $\L^{n+1}_0\to \pd\Delta^{n+1}$ and $\pd\Delta^{n+1}\to \Delta^{n+1}$ are cofibrations.
\end{proof}

\begin{definition}
Say that a morphism of $l$-truncated simplicial sets is an $n$-hypergroupoid if it satisfies the conditions of Definition \ref{relhyp} (up to level $l$).
\end{definition}

\begin{lemma}\label{truncate}
A morphism $f:X\to Y$ is an  $n$-hypergroupoid if and only if $X= Y\by_{\cosk_{n+1}Y}\cosk_{n+1}X$, and the $(n+2)$-truncated morphism $X_{\le n+2}\to Y_{\le n+2}$  is an $n$-hypergroupoid.
\end{lemma}
\begin{proof}
For the first part, note that for $i \ge n+1$, the factorisation $\L^i_0\to \pd\Delta^i \to \Delta^i$ implies that the map
$$
X_i \to Y_i \by_{M_iY}M_iX
$$
has a retraction $\rho$, so is injective.
Since $\L^i_0 \to \pd \Delta^i$ is the pushout of $\pd\Delta^{i-1} \to \Delta^i$ along $\pd^0: \pd\Delta^{i-1} \to \L^i_0$, for $i \ge n+2$ the map
$$
M_iX \to M_iY \by_{M_{\L^i_0}Y}M_{\L^i_0}X 
$$
also has a retraction, so
$$ 
\rho:Y_i \by_{M_iY}M_iX \to Y_i \by_{M_{\L^i_0}Y}M_{\L^i_0}X \cong Y_i\by_{Y_i}X_i= X_i
$$
has a retraction, and so must be an isomorphism.

Thus the matching maps 
$
X_i \to Y_i \by_{M_iY}M_iX
$
are isomorphisms for all $i \ge n+1$, so
$$X= Y\by_{\cosk_{n+1}Y}\cosk_{n+1}X.$$

Finally, note that  $\sk_{n+1}\L^i_k \to \sk_{n+1}\Delta^i$ is an isomorphism for all $i \ge n+3$, so automatically has the exact LLP with respect to 
$X= Y\by_{\cosk_{n+1}Y}\cosk_{n+1}X$. This only leaves the lifting conditions from Definition \ref{relhyp} up to level $n+2$.
\end{proof}

\begin{remarks}
\begin{enumerate}
\item When $n=1$ and $Y$ is a point, we can compare this with the data required to define a groupoid. Levels $0$ and $1$ determine the objects, morphisms and identities, but the face maps from level $2$ are needed to define the multiplication, and hence to construct the groupoid. Only once the face maps from level $3$ have been defined can we ensure that the multiplication is associative.

\item When $n=0$, a  $0$-hypergroupoid $X \to Y$ is a Cartesian morphism. Thus level $0$ gives the fibres, level $1$ the descent datum (thereby determining the fibration), and level $2$ gives the cocycle condition, ensuring that the descent datum is effective.

\item In  Lemma \ref{ttruncate}, only the $(n-1)$-truncation was required  to describe a trivial  $n$-hypergroupoid over a base. In the $1$-dimensional case, this says that a contractible groupoid is just a set, with a unique morphism between any pair of points.
\end{enumerate}
\end{remarks}

Observe that an $n$-hypergroupoid $X$ is determined by $X_m$ for $m\le n+1$, while a trivial $n$-hypergroupoid $X$ is determined by $X_m$ for $m< n$. The next Lemma accounts for this apparent discrepancy.

\begin{lemma}\label{skeleta}
If $X \to Y$ is a relative  $n$-hypergroupoid (resp. a trivial $n$-hypergroupoid), and $K \to L$ is a trivial cofibration (resp. a cofibration) giving an isomorphism $K_{\le n-1} \to L_{\le n-1}$ on $(n-1)$-truncations, then the maps
$$
M_LX \to M_LY\by_{M_KY}M_KX
$$ 
are isomorphisms. 
\end{lemma}
\begin{proof}
This is essentially the same as the proof that the maps $\L^m_k \to \Delta^m$ (resp. $\pd\Delta^m \to \Delta^m$) generate the trivial cofibrations (resp. all cofibrations) in $\bS$. The same proof adapts to show that every map $K \to L$ satisfying the conditions above is a retract of a transfinite composition of  pushouts of maps $\L^m_k \to \Delta^m$ for $m>n$ (resp. maps $\pd\Delta^m \to \Delta^m$ for $m\ge n$).
 \end{proof}

Now, for a $1$-hypergroupoid $X$, the characterisation of $X$ as the nerve of a groupoid means that we can describe each set $X_m$ in terms of $X_0, X_1$ by
\[
 X_m\cong \overbrace{X_1\by_{\pd_0,X_0,\pd_1}X_1\by_{\pd_0,X_0,\pd_1}  \ldots \by_{\pd_0,X_0,\pd_1}X_1}^m.
\]
The following lemma gives an analogue  of this result for $n$-hypergroupoids. 

\begin{definition}\label{delta*}
Define $\Delta_*$ to be the subcategory of the ordinal number category $\Delta$ containing only  those morphisms $f$ with $f(0)=0$.  Given a category $\C$, define the category $s_+\C$ of almost simplicial diagrams   over $\C$ to consist of  functors $\Delta_*^{\op} \to \C$. Thus an almost simplicial object $X_*$ consists of objects $X_n \in \C$, with all of the operations $\pd_i, \sigma_i$ of a simplicial diagram except $\pd_0$, and satisfying the usual relations. 
\end{definition}

%%If $V\co \bS \to s_+\Set$ is forgetful, then $\cL \dashv V \dashv G$, with $(\cL X)_n = X_{n+1}$ and $(G_{\pd} Y_*)_n := Y_n\by Y_{n-1} \by \ldots \by Y_0$. Thus $\Dec_+= \cL V \dashv GV$. We also must have $\cL(\Xi^n) = \Delta^n$, so $\Xi^n_i = \Delta^n_{i-1}$, $\Xi^n_0 = \bt$. Decomp $V\Delta^n = \coprod_{0 \le i\le n} \Xi^i$.

\begin{lemma}\label{truncate2}
If $X$ is a relative  $n$-hypergroupoid over $Y$, then the underlying almost simplicial set over $Y$ is determined by the truncation $X_{\le n}$.
\end{lemma}
\begin{proof}
We can write $\Delta^m$ as the left cone $(\Delta^{m-1})^{\lhd}$, i.e. as the join $\Delta^0 \star \Delta^{m-1}$. Consider the subspaces $K^m := (\sk_{n-1}\Delta^{m-1})^{\lhd} \subset (\Delta^{m-1})^{\lhd}= \Delta^m$; the map $K^m \to \Delta^m$ is a trivial cofibration inducing an isomorphism on $(n-1)$-skeleta. Thus Lemma \ref{skeleta} gives $X_n \cong M_{K^n}X\by_{M_{K^nY}}Y_n$. 

%%The maps $\pd^i\co \Delta^{m-1} \to \Delta^m$ satisfy $(\pd^i)^{\lhd}= \pd^{i+1}$, and we have $(\sigma^i)^{\lhd}= \sigma^{i+1}$ similarly. where's $\sigma^0$ coming from? Non-trivial, but exists.
%
Now, $K^m= \sk_n(K^m)$, so $M_{K^m}X$ is determined by $X_{\le n}$. For the  degeneracy maps $\sigma^i\co \Delta^m \to \Delta^{m-1}$ we have $\sigma^i(K^m) \subset K^{m+1}$, and for the face maps  $\pd^i\co \Delta^m \to \Delta^{m+1}$ we have $\pd^i(K^m) \subset K^{m+1}$ provided $i>0$. Thus the operations $\sigma_j$ and $\pd_i$ for $i>0$ on $X$ can be recovered from the corresponding maps $\sigma_j\co M_{K^m}X \to M_{K^{m+1}}X$, $\pd_i\co M_{K^m}X \to M_{K^{m-1}}X$. 
\end{proof}

\begin{definition}
Given a simplicial set $K$ and a simplicial object $X_{\bt}$ in a category $\C$, define the  object $X^K \in s\C$ by 
$$
(X^K)_n:= \Hom_{\bS}(K\by \Delta^n, X),
$$
for $\Hom_{\bS}$ defined as in Definition \ref{rKan}.
Thus $M_KX= (X^K)_0$.
\end{definition}

\begin{lemma}\label{powers}
If $f:X \to Y$ is an   $n$-hypergroupoid, and $g:K \to L$ a cofibration  in $\bS$, then the map
$$
F:X^L\to Y^L\by_{Y^K}X^L 
$$
is an  $n$-hypergroupoid, which is trivial if either $f$ or $g$ is trivial. Moreover, if $g_{\le n-1}:K_{\le n-1} \to L_{\le n-1}$ is an isomorphism, then $F$ is a  $0$-hypergroupoid.
\end{lemma}
\begin{proof}
 This amounts to describing when a diagram
$$
\xymatrix{ (A\by L)\cup_{(A \by K)}(B\by K) \ar[r] \ar[d]& X \ar[d]^f \\
 B\by L \ar@{-->}[ur] \ar[r] &Y}
 $$
admits the lifting shown, and determining when the lift is unique. For all cofibrations $h:A \to B$, the map $h':(A\by L)\cup_{(A \by K)}(B\by K) \to B\by L$ is  a  cofibration, so the lift must exist if $f$ is trivial. If either $g$ or $h$ is trivial, then $h'$ is trivial, so the lift exists.

Now, 
$$
((A\by L)\cup_{(A \by K)}(B\by K))_{\le n-1} \to (B \by L)_{\le n-1}
$$
is an isomorphism whenever either of the maps $A_{\le n-1} \to B_{\le n-1}$ or $K_{\le n-1} \to L_{\le n-1}$ is an isomorphism, so  Lemma \ref{skeleta} then implies that the lift is unique in this case. 
\end{proof}

\begin{corollary}\label{loopsgpd}
If $f:X \to Y$ is an   $n$-hypergroupoid, then 
$$
X^{\Delta^n}\to Y^{\Delta^n}\by_{Y^{\pd \Delta^n}}X^{\Delta^n} 
$$
is a  $0$-hypergroupoid.
\end{corollary}

The following definition is taken from \cite[Remarks VI.1.4.1]{Ill2}:
\begin{definition}\label{decdef}
Define the d\'ecalage functor $\Dec_+ : \bS \to \bS$ by $\Dec_+ (X)_n=X_{n+1}$, with $\pd_i^{\Dec_+  X}=\pd_{i}^X$ and $\sigma_i^{\Dec_+  X}=\sigma_{i}^X$. This is a comonad, with counit $\pd_{n+1}^X: \Dec_+ (X)_n \to X_n$, and comultiplication $\sigma_{n+1}^X : \Dec_+ (X)_n \to \Dec_+ ^2(X)_n$.  
\end{definition}

\begin{remark}
$\Dec_+$ is the functor denoted by $\mathrm{DEC}$ in \cite{glenn}, and known to homotopy theorists as the simplicial path space $P$.
\end{remark}

\begin{lemma}\label{decgood}
The maps $ X_0 \xra{\sigma_0^X} \Dec_+ (X)\xra{\pd_0^X} X_0 $ form a deformation retract, for all $X \in \bS$. 
\end{lemma}
\begin{proof}
See \cite{duskinmatrices} \S 2.6.
\end{proof}

The following is immediate:
\begin{lemma}
A left adjoint to $\Dec_+ $ is given by the right cone $(-)^{\rhd}$, defined on simplices by $(\Delta^n)^{\rhd}= \Delta^{n+1}$, with $(\Delta^{n-1}\xra{\pd^i} \Delta^n)^{\rhd}=(\Delta^n\xra{\pd^i} \Delta^{n+1})$ and $(\Delta^{n+1}\xra{\sigma^i}\Delta^n)^{\rhd} = (\Delta^{n+2}\xra{\sigma^i}\Delta^{n+1})$. 
\end{lemma}

\begin{corollary}\label{fibres}
If $f: X\to Y$ is an   $n$-hypergroupoid, then  $\Dec_+  (X) \to \Dec_+ (Y)\by_YX$ is an $(n-1)$-hypergroupoid. 
\end{corollary}
\begin{proof}
It suffices to describe the maps $\alpha^m_k:(\L^m_k)^{\rhd}\cup_{\L^m_k}\Delta^m \to (\Delta^m)^{\rhd}$, where the natural transformation $\id \to (-)^{\rhd}$ is adjoint to the co-unit $\Dec_+  \to \id$ of the adjunction. Now, $(\L^m_k)^{\rhd}$ is the union of all faces of $\Delta^{m+1}$ except the $k$th and the $(m+1)$th, so $(\L^m_k)^{\rhd}\cup_{\L^m_k}\Delta^m = \L^{m+1}_{k}$, and $\alpha^m_k$ is just the natural inclusion $\L^{m+1}_{k} \into \Delta^{m+1}$.     

The result now follows immediately since $\alpha^m_k$ lifts with respect to $(-)^{\rhd}$ for all $m$, and lifts uniquely for $m > n-1$. 
\end{proof}

\begin{remark}
Given an $n$-hypergroupoid $X$, this gives us a way to regard the homotopy fibre product $X_0 \by^h_X X_0$ as an $(n-1)$-hypergroupoid, since a model for it is given by the $(n-1)$-hypergroupoid
$$
\Dec_+ (X)\by_X X_0,
$$
as $\Dec_+ (X) \to X$ is a fibration. This object is usually called the path-homotopy complex (\cite{duskinmatrices} \S 2.6).
\end{remark}

\begin{lemma}\label{decvgood}
If $f: X\to Y$ is a relative  $n$-hypergroupoid, then $\Dec_+  (X) \to \Dec_+ (Y)\by_{Y_0}X_0$ is a trivial relative $n$-hypergroupoid.
\end{lemma}
\begin{proof}
The retraction $\Dec_+X \to X_0$ corresponds to a section  $\pi_0K \to K^{\rhd}$, functorial in $K$. It thus suffices to describe the maps $(\pd\Delta^m)^{\rhd}\cup_{\pi_0\pd\Delta^m}\pi_0(\Delta^m) \to (\Delta^m)^{\rhd}$. Calculation  gives this map as $ \L^{m+1}_{m+1} \to \Delta^{m+1}$. Therefore the $m$th relative matching map of $\Dec_+  (X) \to \Dec_+ (Y)\by_{Y_0}X_0$ is surjective for $m<n$, and an isomorphism for $m \ge n$, as required.
\end{proof}

\section{Homotopy $n$-hypergroupoids in model categories}\label{nhypaff}

Fix a pseudo-model category $\cA$ (as in \S \ref{coversn}), and a class $\oC$ of morphisms in $\Ho(\cA)$, containing all  isomorphisms and  stable under composition  and homotopy pullback. Say that a morphism in $\cA$ is a $\oC$-morphism if its image in $\Ho(\cA)$ is so. Refer to fibrations in $\oC$ as $\oC$-fibrations. 
%%%I think we want to take a class $\oC$ which is $\oC$-local, but only so that definitions invariant under hypercovers. Prob don't bother. Desirable but unnecessary?

\subsection{Definitions and basic properties}\label{nhypaffdefsn}

\begin{definition}\label{npreldef}
Given  $Y\in s\cA$, define a (relative) homotopy  $(n,\oC)$-hypergroupoid  over $Y$ to be a morphism $X_{\bt}\to Y_{\bt}$ in $s\cA$, for which  the homotopy  partial matching maps
$$
X_m \to M_{\L^m_k}^h (X)\by_{M^h_{\L^m_k} (Y)}^hY_m 
$$
of Definition \ref{mnh}
are   $\oC$-morphisms for all  $k,m$, and are weak equivalences for all $m>n$ and all $k$. When $Y$ is the final object, we will simply refer to $X$ as a homotopy  $(n,\oC)$-hypergroupoid.

Define a (relative)   $(n,\oC)$-hypergroupoid  over $Y$ to be a homotopy  $(n,\oC)$-hypergroupoid  over $Y$ which is also a Reedy fibration (as in Definition \ref{reedydef}).
\end{definition}

Note that for any homotopy $(n,\oC)$-hypergroupoid $X \to Y$, the Reedy model structure on $s\cA$ gives a factorisation $X \to \hat{X} \to Y$, with $X \to \hat{X}$ a levelwise weak equivalence, and $\hat{X} \to Y$ an $(n,\oC)$-hypergroupoid.

\begin{definition}\label{nptreldef}
Given  $Y\in s\cA$, define a trivial homotopy  $(n,\oC)$-hypergroupoid  over $Y$ to be a morphism $X_{\bt}\to Y_{\bt}$ in $s\cA$, for which  the homotopy  matching maps
$$
X_m \to M^h_{\pd \Delta^m} (X)\by_{M^h_{\pd \Delta^m} (Y)}^hY_m 
$$
are $\oC$-morphisms for all  $m$, and are weak equivalences for all $m\ge n$.

Define a trivial   $(n,\oC)$-hypergroupoid  over $Y$ to be a homotopy  $(n,\oC)$-hypergroupoid  over $Y$ which is also a Reedy fibration.
\end{definition}

\begin{remark}
Note that  trivial relative  $(n,\oC)$-hypergroupoids of affine schemes correspond to the truncated hypercovers considered in \cite{bekecech}.
\end{remark}

The following is a straightforward uncoiling of the definitions:
\begin{lemma}
A morphism $X\to Y$ in $s\cA$ is a  relative $(n,\oC)$-hypergroupoid if and only if
\begin{enumerate}
 \item the matching maps
$$
X_m \to M_{\pd \Delta^m} (X)\by_{M_{\pd \Delta^m} (Y)}Y_m 
$$
of Definition \ref{mn}
are fibrations for all $m\ge 0$;

\item  the   partial matching maps
$$
X_m \to M_{\L^m_k} (X)\by_{M^h_{\L^m_k} (Y)}^hY_m 
$$
are  $\oC$-fibrations for all  $k,m$, and are trivial fibrations for all $m>n$ and all $k$.
\end{enumerate}

A morphism $X\to Y$ in $s\cA$ is a  trivial $(n,\oC)$-hypergroupoid if and only if the matching maps
$$
X_m \to M_{\pd \Delta^m} (X)\by_{M_{\pd \Delta^m} (Y)}Y_m 
$$
of Definition \ref{mn}
are $\oC$-fibrations  for all  $m$, and are trivial fibrations for all $m\ge n$.
\end{lemma}

\begin{examples}\label{affhgpddef}\label{cfddt}
The simplest examples of homotopy hypergroupoids arise when the category $\cA$ has trivial model structure (so all morphisms are fibrations, and weak equivalences are isomorphisms). In these cases, the Reedy fibration conditions are vacuous, and $M^h_K= M_K$. 

If $\cA= \Set$ and $\oC$ is the  class of surjective morphisms, then $(n, \oC)$-hypergroupoids are just the $n$-hypergroupoids  of \S \ref{nhyp}. Other important  examples with trivial model structure arise when we take $\cA$ to be the category of affine schemes. Taking  $\oC$ to be the class of smooth surjections (resp. \'etale surjections, resp. surjective local isomorphisms), we refer to $(n, \oC)$-hypergroupoids as Artin  (resp. Deligne--Mumford, resp. Zariski) $n$-hypergroupoids. We will see in Theorem \ref{bigthm} that these model strongly quasi-compact Artin (resp. Deligne--Mumford, resp. Zariski) $n$-geometric stacks.

Our main motivating example of a non-trivial model category is to take $\cA$ to be the category of cosimplicial affine schemes with the model structure of \S \ref{hagd}. Taking   $\oC$ to be the class of smooth surjections (resp. \'etale surjections, resp. surjective local isomorphisms), we refer to  $(n, \oC)$-hypergroupoids as derived Artin  (resp. derived Deligne--Mumford, resp. derived Zariski) $n$-hypergroupoids. Again, Theorem \ref{bigthm} will show that these model the corresponding $n$-geometric stacks.

For an infinitesimal version, we could take $\cA$ to be the category of left-exact set-valued functors on Artinian simplicial rings (so $\cA=c\Sp$ in the notation of \cite{ddt1}), and $\oC$ the class of formally smooth morphisms. Then the key notions in \cite{ddt1} of quasi-smooth (resp. trivially smooth) morphisms correspond to $(\infty, \oC)$-hypergroupoids (resp. trivial $(\infty, \oC)$-hypergroupoids), for the obvious generalisation of the definitions above to $n= \infty$.

Finally, note that the definitions make sense if we only require $\cA$ to contain pullbacks by morphisms in $\oC$, rather than all pullbacks. Taking $\cA$ to be the category of differential manifolds, and taking $\oC$ the class of surjective submersions, 
 $(n,\oC)$-hypergroupoids are just  Zhu's Lie $n$-groupoids from \cite{zhu}.
\end{examples}

\begin{definition}\label{pskdef}
The functor $i_n^!\co\cA^{\Delta_n^{\op}}\to s\cA $ of Definition \ref{skdef} is a right Quillen functor of Reedy model categories. Write  $\oR i_n^!$ for the associated derived functor, and set $\cosk_n^h:= \oR i_n^!i_{n*}$, giving $\map_{s\cA}(K, \cosk_n^hX)\simeq \map_{\cA^{\Delta_n^{\op}}}(i_{n*}K, X)$.
\end{definition}

The following is immediate:
\begin{lemma}\label{pttruncate}
A morphism $f:X\to Y$ in $s\cA$ is a trivial  homotopy   $(n,\oC)$-hypergroupoid over $Y$ if and only if $X\simeq Y\by_{\cosk_{n-1}^hY}^h\cosk_{n-1}^hX$, and the $(n-1)$-truncated morphism $i_{n-1,*}X\to i_{n-1,*}Y$ satisfies the conditions of Definition \ref{nptreldef} (up to level $n-1$).
\end{lemma} 

\begin{lemma}
Any composition of (trivial) homotopy $(n,\oC)$-hypergroupoids is a (trivial) homotopy   $(n,\oC)$-hypergroupoid. 
\end{lemma}
\begin{proof}
This follows immediately by verifying the axioms.
\end{proof}

Now observe that the functor $\Dec_+$ of Definition \ref{decdef} extends to any simplicial category; in particular, it gives $\Dec_+\co s\cA \to s\cA$.

\begin{lemma}\label{pfibres}
If $f: X\to Y$ is a homotopy $(n,\oC)$-hypergroupoid, then  $\Dec_+  (X) \to \Dec_+ (Y)\by_Y^hX$ is a homotopy $(n-1, \oC)$-hypergroupoid, and $\Dec_+  (X) \to \Dec_+ (Y)\by_{Y_0}^hX_0$ is a trivial homotopy $(n,\oC)$-hypergroupoid.
\end{lemma}
\begin{proof}
The proofs of Corollary \ref{fibres} and Lemma \ref{decvgood} carry over to this context.
\end{proof}

\begin{lemma}\label{ptruncate2}
If $X$ is a homotopy   $(n,\oC)$-hypergroupoid over $Y$, then the underlying almost simplicial object over $Y$ in $s_+\cA$ is determined by up to equivalence by the truncation $X_{\le n}$.
\end{lemma}
\begin{proof}
 The proof of Lemma \ref{truncate2} carries over immediately to this generality.
\end{proof}

\begin{remark}
Beware that the natural analogue of our other truncation result, Lemma \ref{truncate}, does not hold for homotopy   $(n,\oC)$-hypergroupoids in general. Say that a map in $\Ho(\cA)$ is a section if it has a left inverse. Then the  proof only carries over  contexts where the homotopy pullback of a section  in $\Ho(\cA)$ must be a  weak equivalence if it has a right inverse.  It will thus apply when $\cA$ is the category of affine schemes with trivial model structure, since the pullback of a section is a closed immersion.
\end{remark}

\begin{lemma}\label{ppowers}
If $f:X \to Y$ is a homotopy $(n,\oC)$-hypergroupoid, and $g:K \to L$ a cofibration  of finite simplicial sets, then the map
$$
F:X^{\oR L}\to Y^{\oR L}\by_{Y^{\oR K}}^hX^{\oR L} 
$$
is a homotopy $(n,\oC)$-hypergroupoid, which is trivial if either $f$ or $g$ is trivial. Moreover, if $\sk_{n-1}g:\sk_{n-1}K \to \sk_{n-1}L$ is an isomorphism, then $F$ is a homotopy $(0,\oC)$-hypergroupoid.
\end{lemma}
\begin{proof}
The proof of Lemma \ref{powers} carries over.
\end{proof}

\begin{lemma}\label{pdecvgood}
If $f: X\to Y$ is a homotopy $(n, \oC)$-hypergroupoid, then $\Dec_+  (X) \to \Dec_+ (Y)\by_{Y_0}X_0$ is a trivial homotopy $(n, \oC)$-hypergroupoid.
\end{lemma}

\subsection{$\Hom$-spaces}

\begin{definition}\label{HHomhgd}
Given objects $X,Y$ over $S \in s\cA $, let $\HHom_{s\cA \da S}(X,Y)\in \bS$ be given by 
$$
\HHom_{s\cA\da S}(X,Y)_n:= \Hom_{s\cA\da S}(X, Y^{\Delta^n}\by_{S^{\Delta^n}}S)= \Hom_{s\cA\da S}(X\by \Delta^n, Y),
$$
where $(X\by \Delta^n)_i$ is given by the coproduct of $\Delta^n_i$ copies of $X_i$.
\end{definition}

However, the category $s\cA$ does not have enough morphisms to model the theory of stacks:

\begin{example}
If we take $X$ to be an affine scheme, and $Y=BG$, for $G$ a smooth affine group scheme, we have
$$
\Hom_{s\Aff}(X, BG)=\Hom_{\Aff}(X, (BG)_0)= \Hom_{\Aff}(X, \Spec \Z) = \bt, 
$$
the one-point set, whereas  in the category of Artin stacks, $\Hom(X, BG)$ is given by isomorphism classes in the groupoid of $G$-torsors on $X$. 

Now, a trivial relative Artin $1$-hypergroupoid $\tilde{X} \to X$ is just given by $\tilde{X} =\cosk_0(X'/X)$, for   some smooth surjection $X' \to X$, with $X':=\tilde{X}_0$.  
Then an element of
$
\Hom(\tilde{X}, BG)
$
is a $G$-torsor $P$ on $X$ equipped with a trivialisation $\theta: P\by_XX'\cong G\by X'$, while an element of 
$
\Hom(\tilde{X}, (BG)^{\Delta^1})
$
consists of two such pairs $(P_1,\theta_1), (P_2,\theta_2)$, and an isomorphism $g \in G(X')$ between them.

Thus isomorphism classes of $G$-torsors on $X$ will be given by the colimit over all $\tilde{X}$ of $\pi_0 \HHom(\tilde{X}, BG)$. 
\end{example}

Our solution is thus to take a form of localisation with respect to trivial $(n,\oC)$-hypergroupoids.

\subsection{Morphisms via pro-objects}\label{morphismspro}

\begin{definition}\label{procat}
Given a category $\C$, recall that the category $\pro(\C)$ 
has objects consisting of inverse systems $\{A_{\alpha}\in \C\}$, with 
$$
\Hom_{\pro(\C)}(\{A_{\alpha}\}, \{B_{\beta}\})= \lim_{\substack{\lla \\ \beta}} \lim_{\substack{\lra \\ \alpha}} \Hom_{\C}(A_{\alpha},B_{\beta}).
$$
We regard $\C$ as a full subcategory of $\pro(\C)$ by sending $A$ to the singleton inverse system $\{A\}$.
\end{definition}

\begin{definition}\label{TCndef}
 Let $T\oC_n$ be the subcategory of $s\cA$  containing all objects, with morphisms consisting of trivial $(n, \oC)$-hypergroupoids. 
\end{definition}

\begin{definition}\label{tcncof}
Say that an object $X= \{X_{\alpha}\}_{\alpha \in I}$ of $\pro(s\cA)$ is $T\oC_n$-projective if for every $\alpha \in I$ and every trivial $(n, \oC)$-hypergroupoid $Y \to X_{\alpha}$, there exists $\beta \ge \alpha$ and a morphism $X_{\beta} \to Y$ compatible with the structure map $X_{\beta} \to X_{\alpha}$.
\end{definition}

\begin{lemma}\label{hhomfibrant}
If $f\co Y\to S$ is a (trivial) $(n,\oC)$-hypergroupoid, then for all $T\oC_n$-projectives $X \in \pro(s\cA)$, the morphism
\[
f_*\co  \HHom_{\pro(s\cA)}(X,Y)\to \HHom_{\pro(s\cA)}(X,S)      
\]
is a (trivial) Kan fibration.

Moreover, if the model structure on $\cA$ is trivial then $f_*$ is a  (trivial) $n$-hypergroupoid.      
\end{lemma}
\begin{proof}
Take a cofibration $i\co K \to L$ of  finite simplicial sets, with either $i$ or $f$ being trivial. Thus the map  $Y^L\by_{S^L}S \to Y^K\by_{S^K}S  $ is a  relative $(n, \oC)$-hypergroupoid,
For finite simplicial sets $K$, matching objects are  given by
\[
M_K\HHom_{\pro(s\cA)}(X,Y)=\Hom_{\pro(s\cA)}(X, Y^K),
\]
so the $T\oC_n$-projective property of $X$ ensures that $f_*$ is a (trivial) Kan fibration.  

Moreover, if $K_{\le n-1} \cong L_{\le n-1}$, then the map $Y^L\by_{S^L}S \to Y^K\by_{S^K}S  $ is a weak equivalence, and hence an isomorphism when $\cA$ has trivial model structure. In that case, the matching maps on $\HHom(X,-)$ are also isomorphisms, so $f_*$ is an $n$-hypergroupoid.
\end{proof}

We now establish existence of  $T\oC_n$-projectives.

\begin{assumption}\label{smallcat}
From now on, assume that there exists a small full subcategory $\cA^s \subset \cA$ with the property that for every trivial $(n, \oC)$-hypergroupoid $X \to Y$  in $s\cA$, there is a commutative diagram
\[
 \begin{CD}
  X @>>> \bar{X}\\
@VVV @VVV \\
Y @>>> \bar{Y},
 \end{CD}
\]
 with $\bar{X} \to \bar{Y}$ a trivial $(n, \oC)$-hypergroupoid in $s\cA^s$, and $X \to Y\by_{\bar{Y}}\bar{X}$ a Reedy trivial fibration.
\end{assumption}

\begin{examples}\label{smallcategs}
For our main motivating examples (Artin and Deligne--Mumford hypergroupoids, together with their derived analogues), Assumption \ref{smallcat} is satisfied by taking $\cA^s$ to consist of objects of finite type. This works because every smooth (resp. \'etale)  morphism $A \to B$ of rings is the pullback of a smooth (resp. \'etale) morphism $A' \to B'$ for some finitely generated subring $A' \subset A$. 

The subring $A'$ is constructed by  choosing generators and relations for $B$ over $A$, then taking $A'$ to be generated by the coefficients of the relations together with various discriminants. 
We know by Lemma \ref{pttruncate} that a trivial $(n, \oC)$-hypergroupoid $f\co X \to Y$ is determined by the morphism $f_{<n}\co X_{<n} \to Y_{<n}$ of finite diagrams, and the discussion above shows that  $f_{<n}$ must be the pullback of a  trivial $(n, \oC)$-hypergroupoid in $s_{n-1}\cA^s$, from which it follows that Assumption \ref{smallcat} is satisfied.   

 If we wish to work in the full generality of \cite{hag2} or of \S \ref{setupsn}, then we can follow \cite{hag2} in picking two universes $\bU \subset \bV$, with $\Aff_{\C}$ being a $\bV$-small category with $\bU$-small $\Hom$-sets, and containing only $\bU$-small limits. We can then take $\cA^s$ to be the image of $\Aff_{\C}$ under $\oR \uline{h}$, and Assumption \ref{smallcat} applies provided we interpret ``small'' as ``$\bV$-small''. We must then permit $\bV$-small inverse systems in Definition \ref{procat}. The corresponding application of this assumption in \cite{hag2} lies in the construction of $\Aff_{\C}^{\sim, \tau}$ via Bousfield localisation of a $\bV$-small set. 
 In practice, the only interesting example  known to the author of a HAG context requiring the introduction of a larger universe is when we take $\oC$ to be the class of fpqc morphisms of affine schemes.
\end{examples}

We now give two definitions adapted from \cite{hovey} \S 2.1.

\begin{definition}
Given an ordinal $\lambda$, define a $\lambda$-cosequence in a category $\C$ to be a limit-preserving functor $X \co \lambda^{\op} \to \C$. We say that the \emph{composition} of $X$ is the map $\Lim_{\beta < \lambda} X_{\beta} \to X_0$.  
\end{definition}

\begin{definition}
Given a class $P$ of morphisms in a category $\C$, we say that a morphism $f$ in $\C$ is a relative $P$-cocell if it can be written as a transfinite composition of pullbacks of elements of $P$. In other words, 
for some ordinal $\lambda$,   $f$ can be written as the composition of a $\lambda$-cosequence $X \co \lambda^{\op} \to \C$ with the property that for all $\beta$ with $\beta+1< \lambda$, there is a pullback diagram
\[
\begin{CD}
 X_{\beta+1} @>>> X_{\beta} \\
@VVV @VVV \\
D_{\beta} @>{g_{\beta}}>> C_{\beta}
\end{CD}
\]
with $g_{\beta}$ in $P$.
\end{definition}

\begin{proposition}\label{procofibrant}
There is a functor $Q\co \pro(s\cA) \to \pro(s\cA)$ equipped with a natural transformation $Q \to \id$  such that for all $X \in \cA$,
the morphism $QX \to X$ is a relative $T\oC_n$-cocell, with  $QX$  a $T\oC_n$-projective object of $\pro(s\cA)$.
\end{proposition}
\begin{proof}
The factorial factorisation provided by the Reedy model structure on $s\cA$ gives us a functor $Q_R \co s\cA \to s\cA$   for which each $Q_RX$ is   Reedy cofibrant, together with a natural transformation $Q_R \to \id$ with $Q_RX \to X$ a Reedy trivial fibration. 

Applying the small object argument (\cite{hovey} Theorem 2.1.14) to the set $T\oC_n^s:= \cA^s\cap T\oC_n$ yields a functor $Q^s\co \pro(\cA) \to \pro(\cA)$ and a natural transformation $Q^s \to \id$ with each  $Q^sX \to X$ a relative $T\oC_n$-cocell with $Q^sX$ being $T\oC_n$-projective. 

Now set $QX:= \Lim_n (Q^sQ_R)^n$, the limit taken in $\pro(s\cA)$. This is automatically both $T\oC_n^s$-projective and Reedy cofibrant. By assumption \ref{smallcat}, every element of $T\oC_n$ can be written as the composition of a Reedy trivial fibration with the pullback of a morphism in $T\oC_n^s$, so it follows that $QX$ is $T\oC_n$-projective. 
\end{proof}

\begin{definition}\label{hgdscatdef}
Given $S \in s\cA$ Reedy fibrant, define $\cG_n(\cA, \oC,S)$, the \emph{localised simplicial category of  $(n,\oC)$-hypergroupoids over $S$} as follows. Objects of $\cG_n(\cA, \oC,S)$ are  $(n,\oC)$-hypergroupoids $X \to S$, while morphisms are given by
\[
 \HHom_{ \cG_n(\cA, \oC,S)}(X,Y):= \HHom_{\pro(s\cA)\da S}(QX,QY),       
\]
for the $T\oC_n$-projective replacement functor $Q\co s\cA \to \pro(s\cA)$ of Proposition \ref{procofibrant}.
\end{definition}

The choice of $T\oC_n$-projective replacement does not matter:
\begin{lemma}\label{morsimpler}
Take $\tilde{X} \to X$  a relative  $T\oC_n$-cocell in $\pro(s\cA)$ with $\tilde{X}$ a $T\oC_n$-projective. Then for all 
$(n,\oC)$-hypergroupoids $Y \to S$, the spaces $\HHom_{ \cG_n(\cA, \oC,S)}(X,Y) $ and $ \HHom_{\pro(s\cA)\da S}(\tilde{X},Y) $ are weakly equivalent.
\end{lemma}
\begin{proof}
By Lemma \ref{hhomfibrant}, the map
\[
 \HHom_{\pro(s\cA)\da S}(QX,QY) \to \HHom_{\pro(s\cA)\da S}(QX,Y)       
\]
is a trivial fibration. It therefore suffices to show that for any two choices $\tilde{X}_1, \tilde{X}_2$ of $\tilde{X}$ above, we have $ \HHom_{\pro(s\cA)\da S}(\tilde{X}_1,Y) \simeq  \HHom_{\pro(s\cA)\da S}(\tilde{X}_2,Y) $

Now, let $\tilde{X}'= Q(\tilde{X}_1\by_X \tilde{X}_2)$; this is a $T\oC_n$-projective, and the canonical maps $p \co \tilde{X}' \to \tilde{X}_i$,  are both relative  $T\oC_n$-cocells. We will show that $p^*\co \HHom_{\pro(s\cA)\da S}(\tilde{X}_i,Y)\to \HHom_{\pro(s\cA)\da S}(\tilde{X}',Y) $ is a homotopy equivalence.  

Since $\tilde{X}_i$ is $T\oC_n$-projective, the map $p$ has a section $s$. Then we have $(\id, sp) \co \tilde{X}' \to\tilde{X}'\by_{\tilde{X}_i}\tilde{X}'$, but since $\tilde{X}'$ is $T\oC_n$-projective, this lifts to a map $h\co \tilde{X}' \to(\tilde{X}')^{\Delta^1}\by_{\tilde{X}_i^{\Delta^1}}\tilde{X}_i$. For any $Z \in \pro(s\cA)$, this induces a map
\[
 h^*\co \Hom_{\pro(s\cA)}(\tilde{X}',Z) \to \Hom_{\pro(s\cA)}(\tilde{X}',Z^{\Delta^1}),
\]
given by $f \mapsto (f^{\Delta^1}) \circ h$. Considering $Z$ of the form $Y^{\Delta^n}, S^{\Delta^n}$, this gives us a homotopy
\[
 h^*\co \HHom_{\pro(s\cA)\da S}(\tilde{X}',Y)\to \HHom_{\pro(s\cA)\da S}(\tilde{X}',Y)^{\Delta^1}
\]
between $p^*s^*$ and $\id$, so $p^*$ is a deformation retract.
\end{proof}

\begin{remark}\label{holimrk}
Considering \cite{bousfieldkan} XI.9.2, it is natural to ask whether the diagram $\tilde{X}$ is cofinal in some more natural category $I$ over $X$, in which case $ \HHom_{\pro(s\cA)\da S}(\tilde{X},Y)\simeq \ho\LLim_{X' \in I} \HHom_{s\cA\da S}(X',Y)$. Unfortunately,  $\tilde{X}$ is not cofinal in the category $\cT_n(X)$ of trivial $(n, \oC)$-hypergroupoids over $X$, in general. However, if we regard $\cT_n(X)$ as a full simplicial subcategory of $s\cA\da X$, then $\tilde{X}$  is cofinal in the homotopy category $\pi_0\cT_n(X)$. This means that
\begin{align*}
 \pi_0\HHom_{\pro(s\cA)\da S}(\tilde{X},Y)&\simeq \LLim_{X' \in \pi_0\cT_n(X) }\pi_0\HHom_{s\cA\da S}(X',Y),\\
\pi_n(\HHom_{\pro(s\cA)\da S}(\tilde{X},Y),f) &\simeq \LLim_{X' \in \pi_0\cT_n(X) }\pi_n(\HHom_{s\cA\da S}(X',Y),f).
\end{align*}

More generally, if we write  $\tilde{X}= \{\tilde{X}_i\}_i$, then for any $T \in \cT_n(X)$, any  $m \ge 0$  and any element   $f \in M_m\HHom_{s\cA\da X}(\tilde{X}_i,T)$, there is a $j\ge i$ for which $f$ lifts from $M_m\HHom_{s\cA\da X}(\tilde{X}_j,T)$ to $\HHom_{s\cA\da X}(\tilde{X}_j,T)_m$. This suggests that we should we working with a notion of homotopy colimits indexed by the \emph{simplicial} category $\cT_n(X)$, possibly by generalising the formulae of \cite{bousfieldkan} Ch. XII to involve nerves of simplicial categories.
\end{remark}

\begin{remark}\label{modelstr}
The full subcategory of $s\cA \da S$ consisting of  the $(n, \oC)$-hypergroupoids over $S$ has the structure of a category of  fibrant objects in the sense of \cite{brownAHT} \S 1.  We can say that $f:X \to Y$ is a weak equivalence if $\HHom_{\pro(s\cA)\da S}(QU, X) \to  \HHom_{\pro(s\cA)\da S}(QU, Y)$ is a weak equivalence in $\bS$ for all $U \in \cA$, and fibrations are relative $(n, \oC)$-hypergroupoids. The path object of $X$ is given by $X^{\Delta^1}$. We also obtain a  category of fibrant objects if we take the union over all $n$. 

This raises the question of whether  $(n, \oC)$-hypergroupoids can be realised as fibrant objects of some simplicial model structure on $s\cA$. Since a model structure is determined by the fibrant objects and the trivial fibrations, this is very unlikely to be true in general. On the other hand, $\cG_n(\cA, \oC,S)$ can be regarded as a full simplicial subcategory  of  the right Bousfield localisation $R_{T\oC_n}(\pro(s\cA))$ of  $\pro(s\cA)$ with respect to trivial $(n, \oC)$-hypergroupoids, with $Q$ becoming a cofibrant replacement  functor. However, this model category has many more fibrant objects than just the $(n, \oC)$-hypergroupoids.
\end{remark}

\subsection{Artin $n$-hypergroupoids and \'etale hypercoverings}\label{etcoversn}

We now specialise to the case where $\cA=\Aff$, the category of affine schemes, with trivial model structure. If $\oC$ is the class of smooth surjections, then Proposition \ref{procofibrant} seems inadequate,  since we would expect to be able to define the sheafification by considering just \'etale hypercoverings, rather than smooth hypercoverings. We will now show how to establish this in a slightly more general setting.

For a class $\oC$ as in \S \ref{nhypaffdefsn}, let $\oE \subset \oC$ be a class of morphisms  of affine schemes containing all  isomorphisms and  stable under composition  and homotopy pullback, together with the following condition:

\begin{enumerate}
\item[(E)] for all diagrams $Z \xra{f} X \xra{g} Y$, with $g$ in $\oC$ and $gf$ a closed immersion, there exists a factorisation $Z \xra{\alpha} W \xra{\beta} X$ of $f$, for which $g\beta$ is  in $\oE$. 
\end{enumerate}
 
\begin{lemma}\label{smoothlift}
If $\oC$ is the class of smooth surjections, then the class of  \'etale surjections satisfies condition (E).
\end{lemma} 
\begin{proof}
If $Z=\emptyset$, then this follows from \cite{EGA4.4} 17.16.3, giving $W_0 \to X$ for which the composition $W_0 \to Y$ is an \'etale surjection. 

In general, since $Z \to Y$ is a closed immersion, we may form the henselisation $Z \to Y^h \xra{f'} Y$, as in \cite{greco}; then $(Z,Y^h)$ is a henselian pair (as in \cite{lafon}), and $Y^h \to Y$ is a filtered inverse limit of \'etale morphisms.
Now, \cite{gruson}  Theorem I.8 shows that the map $Z \to X$ extends to $Y^h \to X$ over $Y$, since $g$ is smooth. Finally, since $g$ is finitely presented, and $f'$ is pro-\'etale, the map $Y^h \to X$ factors through some quotient $W'$ with $W' \to Y$ \'etale. Setting $W= W_0 \coprod W'$ completes the proof. 
\end{proof}

\begin{definition}
Given a simplicial diagram $X_{\bt}$ in a cocomplete category $\C$, recall from \cite{sht} \S VII.1 that the $n$th latching object $L_nX$ is defined to be $(\sk_{n-1}X)_n$. Explicitly, this is  the coequaliser
$$
\xymatrix@1{ \coprod_{i=0}^{n-1} \coprod_{j=0}^{i-1} X_{n-2}\ar@<.5ex>[r]^-{\alpha} \ar@<-.5ex>[r]_-{\beta} & \coprod_{i=0}^{n-1} X_{n-1} \ar[r]& L_nX,}
$$
where for $x \in X_{n-2}^{(i,j)}$, we define $\alpha(x)= \sigma_jx \in X_{n-1}^{(i)}$, and $\beta(x)= \sigma_{i-1}x\in X_{n-1}^{j}$.

Note that there is a map $L_nX \to X$ given by $\sigma_i$ on $X_{n-1}^{(i)}$ --- this comes from the counit $\sk_{n-1}X \to X$ of the adjunction $i_{n-1}^* \dashv i_{n-1,*}$ of Definition \ref{skdef}. 
\end{definition}

\begin{proposition}\label{etalecover}
Given a diagram $Z \xra{f} X \xra{g} Y$ in $s\Aff$, with $g$ a  trivial relative  $(n,\oC)$-hypergroupoid  and $gf$ a levelwise closed immersion, there exists a factorisation $Z \xra{\alpha} W \xra{\beta} X$ of $f$, for which $g\beta$ is  a  trivial relative  $(n,\oE)$-hypergroupoid. 
\end{proposition}
\begin{proof}
 The construction of $W$ is inductive. For $r< n$, assume that  we have constructed an   $(r-1)$-truncated simplicial  affine scheme   $W_{<r} $, with $Z_{<r} \to W_{<r} \to X_{<r}$ satisfying the required properties up to level $r-1$.

Since we are working with affine schemes, all colimits exist, including latching objects, and we now seek $W_r$ fitting into the diagram
$$
L_rW\cup_{L_rZ}Z_r \to W_r \to  M_rW\by_{M_rX}X_r, 
$$ 
with the relative matching map $W_r \to M_rW\by_{M_rY}Y_r$ an  $\oE$-morphism. 

Since $X_r \to  M_rX\by_{M_rY}Y_r$ is a $\oC$-morphism, pulling back along $M_rW \to M_rX$ gives a $\oC$-morphism $ M_rW\by_{M_rX}X_r\to M_rW\by_{M_rY}Y_r $. 
Since the relative matching map of a surjection of cosimplicial abelian groups is abelian, the same is true for cosimplicial rings. Thus the latching maps of a levelwise closed immersion of  simplicial affine schemes are all closed immersions. In particular, $L_rW\cup_{L_rZ}Z_r \to M_rW\by_{M_rY}Y_r$ is a closed immersion.

Applying condition (E) to 
$L_rW\cup_{L_rZ}Z_r \to M_rW\by_{M_rX}X_r \to M_rW\by_{M_rY}Y_r$
 provides a factorisation $L_rW\cup_{L_rZ}Z_r \to W_r \to M_rW\by_{M_rX}X_r$, with the composition $W_r \to M_rW\by_{M_rY}Y_r$ an  $\oE$-morphism. This completes the inductive step. Having constructed $W_{\le n-1}$, we set $W= (i_{n-1}^!W_{\le n-1})\by_{\cosk_{n-1}Y}Y$.
\end{proof}

\begin{corollary}\label{procofibrantet}
There is a functor $Q\co \pro(s\Aff) \to \pro(s\Aff)$ equipped with a natural transformation $Q \to \id$  such that for all $X \in \cA$,
the morphism $QX \to X$ is a relative $T\oE_n$-cocell, with  $QX$  a $T\oC_n$-projective object of $\pro(s\cA)$.
\end{corollary}
\begin{proof}
Proposition \ref{procofibrant} applied to the class $\oE$ provides us with a functor $Q$ for which $QX \to X$ is a relative $T\oE_n$-cocell, with  $QX$  a $T\oE_n$-projective object of $\pro(s\cA)$. Now just observe that Proposition \ref{etalecover} ensures that every $T\oE_n$-projective  is $T\oC_n$-projective. 
\end{proof}

\section{Hypergroupoids vs. $n$-stacks}\label{ressn}

\subsection{Passage to $n$-stacks}

We now assume that our pseudo-model category $\cA \subset \cS$ is equipped  with all the additional structure of \S \ref{setupsn}.

\begin{proposition}\label{easy}
If $f:X\to Y$ is a relative  $(n,\oC)$-hypergroupoid, then the associated morphism
$$
|f|: | X| \to |Y|
$$
in $\cS$ an $n$-representable morphism    in the sense of  Definition \ref{geomdef}. 

If $f_0: X_0 \to Y_0$ is also in $\oC$, then  $|f|$ is an $(n,\oC)$-morphism.
\end{proposition}
\begin{proof}
First, we  reduce this  to the case when $Y \in \cA$. By Property \ref{coverprops}.\ref{coverpsurj}, the map $ Y_0 \to | Y|$ is an $\vareps$-morphism. By Lemma \ref{geomlocal}, it therefore suffices to show that $|X| \by^h_{|Y|}Y_0 \to Y_0$ is $n$-representable, and that it lies  in $n-\oC$ when $f_0$ is in $\oC$.

Since $ f$ satisfies the conditions of Property \ref{coverprops}.\ref{coverphprod}, we have
\[
 | X| \by^h_{| Y|} Y_0 \simeq | X\by_Y^hY_0|.
\]
By homotopy pullback, $X\by_Y^hY_0$ is an   $(n,\oC)$-hypergroupoid over $Y_0$, and $ (X\by_Y^hY_0)_0 \simeq X_0$, allowing us to replace $Y$ with $Y_0$ for both statements.  

We now proceed by induction on $n$. For $n=0$, the statements are immediate.

Now assume that the inductive hypothesis (over arbitrary base) holds for $n-1$, and take a relative $(n,\oC)$-hypergroupoid $X \to Y_0$. Since $X$ is then an $(n,\oC)$-hypergroupoid, 
Lemma \ref{pfibres} implies that $\Dec_+ (X) \to X$ is a relative $(n-1, \oC)$-hypergroupoid, with $\Dec_+ (X)_0 \to X_0$ in $\oC$, so by induction $| \Dec_+(X)| \to | X|$ is an $(n-1,\oC)$-morphism.  Moreover, Lemma \ref{pdecvgood} and Property \ref{coverprops}.\ref{coverptriv} imply that $X_0 \simeq | \Dec_+(X)|$, which means that we have an $(n-1,\oC)$-morphism  $X_0 \to |X|$. Since $X_0$ is $0$-geometric, Proposition \ref{hagcover} now implies that $| X|$ is $n$-geometric. As $Y_0$ is in $\cA$, this implies that $| f|: | X| \to  Y_0$ is $n$-representable. 

 If $f_0 :X_0 \to Y_0$ is in $\oC$, then the $n$-atlas $ X_0$ is a $\oC$-morphism over $Y_0$, so  $|X| \to  Y_0$ is in $n$-$\oC$. This completes the induction.
\end{proof}

\subsection{Resolutions}

In this section, we will establish the converse to Proposition \ref{easy}. By way of motivation, we begin with the cases of low degree. Given a $1$-geometric stack $\fX$, take a $1$-atlas $X_0 \to \fX$, and observe that $\cosk_0^h(X_0/\fX)$ gives a $(1,\oC)$-hypergroupoid resolving $\fX$.

Now for the construction of a $(2,\oC)$-hypergroupoid $X$ resolving a $2$-geometric stack $\fX$. Take a $2$-atlas $Z_0 \to \fX$; since $Z_0\by^h_{\fX} Z_0$ is $1$-geometric, it admits a $1$-atlas $Z_1$. Since an atlas is a $\oC$-morphism, the diagonal map $Z_0 \to Z_0\by^h_{\fX} Z_0$ must admit a lift $X_0 \to Z_1$ 
for some $\oC$-morphism  $X_0 \to Z_0$.
Setting $X_1:=Z_1\by_{(Z_0\by Z_0)}(X_0\times X_0)$ gives  a $1$-truncated simplicial scheme $X_{\le 1}$ over $\fX$, so we can take $X$ to be the coskeleton
$$
X:= \cosk_1^h(X_{\le 1}/\fX),
$$
and $X$ will be a $(2,\oC)$-hypergroupoid.
 
Note that we cannot adapt the approach of \cite{SD} to construct $X$ as a split simplicial resolution, since that would entail taking $X_0=Z_0$ and $X_1=Z_1 \sqcup Z_0$. 
It fails because the diagonal map $(Z_0\by_{\fX}Z_0) \sqcup Z_0 \to Z_0\by_{\fX}Z_0$ will seldom be a $\oC$-morphism. 

\begin{definition}\label{gencosk}
Given an object $S\in \cS$ , define $S^{\Delta_r} \in s\cS$ to be given by $(S^{\Delta_r})_i= S^{\Delta^i_r}$. 
\end{definition}

\begin{lemma}\label{gencoskworks}
For $K \in \bS$, there is a natural equivalence
$$
M_K^h (S^{\Delta_r}) \simeq  S^{K_r},
$$
where $M_K^h$ is the homotopy matching object of Definition \ref{mnh}.

Thus, for a morphism $f:S \to T$ in $\cS$,  the $i$th homotopy matching map $(S^{\Delta_r})_i \to M_i^h(S^{\Delta_r})\by^h_{M_i^h(T^{\Delta_r})}(T^{\Delta_r})_i$ is
\begin{enumerate}
\item a weak equivalence when $i>r$;
 \item a homotopy pullback of $f$ when $i=r$;
 \item a homotopy pullback of finitely many copies of $f$ when $i<r$.
\end{enumerate}
\end{lemma}
\begin{proof}
The first statement follows because the functors $K \mapsto M_K^h(S^{\Delta_r})$ and $K \mapsto S^{K_r}$  both send  limits in $\bS^{\op}$ to homotopy limits in $\cS$, and the functors agree on the objects $\Delta^n$. 

Thus we have canonical equivalences
$$
(S^{\Delta_r})_i \simeq S^{\Delta^i_r}\simeq S^{\pd \Delta^r_i}\by S^{\Delta^r_i- \pd\Delta^r_i}  \simeq  M_i(S^{\Delta_r})\by S^{\Delta^r_i- \pd\Delta^r_i},
$$
giving the required isomorphisms and pullbacks.
\end{proof}

\begin{definition}
Given  $Z_{\bt}  \in s\cS$ (i.e. a simplicial diagram in $\cS$), let $a_r(Z) \in \N_0$ be the smallest number for which the homotopy matching  morphism 
$$
Z_r \to M^h_rZ
$$
is $a_r(Z)$-representable, whenever this is defined.

If the numbers $a_r(Z)$ are all defined and are eventually $0$, define
$$
\nu(Z):=\sum_{\substack{i \ge 0\\ a_i(Z)\ne 0}}2^i.
$$
\end{definition}

\begin{proposition}\label{res}
Given $Z\in s\cS$ for which $\nu(Z)$ is defined and $\nu(Z)<2^n$, there exists $\tilde{Z} \in s\cA$  and a morphism  $ \tilde{Z} \to Z$ in $\Ho(s\cS) $ with the property that
the relative homotopy matching maps
$$
 \tilde{Z}_r \to (M^h_r\tilde{Z})\by^h_{(M^h_rZ)}Z_r
$$
are  $\oC$-morphisms for all $r\ge 0$, and equivalences for all $r\ge n$.
\end{proposition}
\begin{proof}
Begin by replacing $Z$ with a Reedy fibrant object in $s\cS$ --- this will allow us to choose atlases in $\cS$ rather than just in $\Ho(\cS)$.
We work  by induction on $\nu(Z)$. If $\nu(Z)=0$, then the matching maps are all $0$-representable. Since $0$-representability of the objects $Z_i$ for $i<r$ implies $0$-representability of $M_r^hZ$, we deduce that each $Z_r$ is $0$-geometric, giving $Z \in s\cA$.

Now assume that the result holds for all $Y$ with $\nu(Y) < \nu(Z)$.  Let $r$ be the smallest number for which $a_r(Z)\ne 0$; since $\nu(Z)< 2^n$, we know that $r<n$. The stacks $Z_i$ are all $0$-representable for $i < r$, so $M_r^hZ$ is $0$-representable. As $Z_r \to M_r^hZ$ is $a_r(Z)$-representable, we know that $Z_r$ is an $a_r(Z)$-geometric stack.

Let  $f:S \to Z_r$ be an $a_r(Z)$-atlas for $Z_r$ and a fibration,  observe that Lemma \ref{gencoskworks} gives a canonical map $S \to Z_r^{\Delta_r}$, and let
$$
Z':= Z\by^h_{Z_r^{\Delta_r}}S^{\Delta_r}= Z\by_{Z_r^{\Delta_r}}S^{\Delta_r}.
$$
Since $f$ is  in   $(a_r(Z)-1)- \oC$, Lemma \ref{gencoskworks} now implies that the $i$th relative homotopy matching map of $Z' \to Z$ is a $\oC$-morphism   and is
\begin{enumerate}
\item an equivalence when $i > r$;
\item a pullback of $f$ when $i=r$.
\end{enumerate}

As the homotopy matching maps of $Z'$ are simply the composition of these relative homotopy matching maps with a pullback of the homotopy matching maps for $Z$, 
the $r$th homotopy matching map is given by pulling back the composition
$$
S \to Z_r \to M^h_rZ= M_rZ
$$
along $M_rZ' \to M_rZ$. Since $S$ and $M^h_rZ$ are both $0$-representable, it follows that this map is $0$-representable, so $a_r(Z')=0$. Likewise, $a_i(Z')\le a_i(Z)$ for all $i>r$. Moreover, $a_i(Z')$ is defined for all $i$, since the relative matching homotopy maps are $\oC$-morphisms, so $k$-geometric for some $k$.

Thus $\nu(Z')<\nu(Z)$ (as $2^r>\sum_{i<r} 2^i$), so satisfies the inductive hypothesis, giving $\tilde{Z}\to Z'$ satisfying the conditions of the proposition. Now observe that the composition $ \tilde{Z} \to Z' \to Z$ also satisfies these conditions, which completes the proof of the inductive step. 
 \end{proof}

\begin{remark}
Here is a further explanation of why the induction has $2^n -1$ steps. If we let $f(r)$ be the number of steps required to set $a_i$ to $0$ for all $i<r$, then one further step sets $a_r$ to $0$, at the expense of the $a_i$ becoming non-zero for $i<r$. We therefore require $f(r)$ further steps to get $a_i=0$ for all $i < r+1$. Thus we have the recurrence relation $f(r+1)= 2f(r)+1$, with initial condition $f(0)=0$, so $f(n)=2^n-1$. 
\end{remark}

\begin{theorem}\label{relstrict}
Given  $Y \in s\cA$, and an $n$-representable   morphism $f:\fX\to |Y|$ in $\Ho(\cS)$, there exists a relative  $(n,\oC)$-hypergroupoid  $X\to Y$ in $\Ho(s\cA)$ whose realisation $| X|$ is equivalent to $\fX$ over $|Y|$. Moreover, if $f$ is a  $\oC$-morphism, then $X_0 \to Y_0$ is in $\oC$. 
\end{theorem}
\begin{proof}
We first let $\fY:=|Y|$, and form the simplicial stack $Z\in s\cS$ given by $Z_r:= \fX\by_{\fY}^h Y_r$. The $r$th relative homotopy matching map of $Z\to Y$ is  just the pullback of that for $\fX\to \fY$, so by Lemma \ref{higherdiag}, it is $(n-r)$-representable for $r \le n$, and $0$-representable  for $r>n$. Since the matching maps of $ Y$ are $0$-representable (being in $\cA$), it follows that   $a_r(Z)\le n-r$ for $r \le n$, and $a_r(Z)=0$ for $r\ge n$. Thus $\nu(Z)$ is defined and $\nu(Z) \le 2^n-1$.  

Proposition \ref{res} now gives a morphism $X \to Z$ in $\Ho(s\cA)$ for which the relative homotopy matching maps
$$
X_r \to (M_rX)\by^h_{M_r^hZ}Z_r
$$
are  $\oC$-morphisms  for all $r$, and equivalences for $r \ge n$. 
Since the relative homotopy matching maps are $\vareps$-morphisms, Property \ref{coverprops}.\ref{coverptriv} ensures that $|X| \to |Z|$ is an equivalence in $\cS$.

Now, Property \ref{coverprops}.\ref{coverphprod} gives
$$
| Z| \simeq \fX\by_{\fY}^h|Y|\simeq \fX,
$$
so we have a weak equivalence $|X| \to \fX$.

For any cofibration $K \into L$ of finite simplicial sets, the map
$$
M_L^h X \to M_K^hX\by^h_{M^h_KZ}M^h_LZ=   (M_K^h X \by_{M_K^h Y}M_L^h Y)\by^h_{(\fX^{\oR  K}\by^h_{\fY^{\oR  K}}\fY^{\oR L})} \fX^{\oR L}
$$
is a  $\oC$-morphism, and moreover an equivalence if $\sk_{n-1}K \cong \sk_{n-1}L$. When $K$ and $L$ are contractible, this map just becomes
$$
M_L^h X \to M_K^h X \by_{M_K^h Y}M_L^h Y.
$$
Thus $X$ is indeed an  $(n,\oC)$-hypergroupoid over $Y$. 

 Finally, if $f$ is a $\oC$-morphism, then 
$$
| X|\by_{| Y|}^h Y_0 \to  Y_0
$$
is also a $\oC$-morphism. Since $ X_0 \to |X|\by_{| Y|}^h| Y_0|$ is an $n$-atlas, it is a $\oC$-morphism too, so $X_0 \to Y_0$ is in $\oC$. 
\end{proof}

\subsection{Morphisms}\label{morphisms}

As we saw in \S \ref{morphismspro}, the realisation functor $|-|$ from $s\cA$ to stacks is not full, even if we restrict to $(n, \oC)$-hypergroupoids. We will therefore now have to localise over trivial $(n, \oC)$-hypergroupoids.

\begin{lemma}\label{resconv}
Given an $(n,\oC)$-hypergroupoid $Y \to S$, the $i$th relative homotopy matching map of
$$
Y \to S\by_{|S|}^h|Y|
$$
is a $\oC$-morphism  for all $i$, and an equivalence for $i \ge n$. If $i \le n$, it is $(n-i)$-representable.
\end{lemma}
\begin{proof}
Let $\fS=|S|, \fY= |Y|$. The matching map is
$$
\mu_i:Y_i \to (M_i^hY\by^h_{M_i^hS}S_i)\by^h_{(\fY^{\oR \pd \Delta^i}\by^h_{\fS^{\oR \pd \Delta^i}}\fS )}\fY,
$$
which by Property \ref{coverprops}.\ref{coverphprod} is given by applying $|-|$ to the map
$$
Y_i \to (M_iY\by^h_{M_iS}S_i)\by^h_{(Y^{\pd \Delta^i}\by_{S^{\pd \Delta^i}}S^{\Delta^i } )}Y^{\Delta^i },
$$ 
where we have used the fact that $Y \to S$ is a Reedy fibration to replace homotopy matching objects with strict matching objects.

The map is given in simplicial level $0$ by
$$
Y_i \to (M_iY\by_{M_iS}S_i)\by_{(M_iY\by_{M_iS}S_i )}Y_i=Y_i;
$$
this  is certainly a $\oC$-morphism, so $\mu_i$ is also
a  $\oC$-morphism  by Proposition \ref{easy}.

If $i \ge n$, then Lemma \ref{higherdiag} implies that $\mu_i$  is $0$-representable, and if $i \le n$, then it is $(n-i)$-representable.
\end{proof}

\begin{proposition}\label{duskinfull}
Take  a morphism $X \to S$ in  $ s\cA$ with $S$ Reedy fibrant, % (for instance  a relative $(m, \oC)$-hypergroupoid), 
a homotopy $(n, \oC)$-hypergroupoid $Y \to S$, and a morphism 
$$
f:|X| \to |Y|
$$
in  the homotopy category $\Ho(\cS\da |S|)$. 

Then there exists a trivial homotopy  $(n,\oC)$-hypergroupoid $\pi:\tilde{X} \to X$ and a morphism $\tilde{f}: \tilde{X}  \to Y$ in $\Ho(s\cA)$, such that $f\circ |\pi|=|\tilde{f}|$. Moreover, the map $(\pi, \tilde{f}): \tilde{X} \to X\by_S^hY$ is a homotopy   $(n,\oC)$-hypergroupoid.
\end{proposition}
\begin{proof}
Without loss of generality, we may assume that $Y$ is a Reedy fibration over $S$. Write $\fS:= |S|$, $\fY:= |Y|$ and $\fX:= |X|$, and define the simplicial stack $Z\in s\cS$ by 
$$
Z_r:= X_r\by^h_{(\fY\by_{\fS}^hS_r)}Y_r.
$$
Observe that the relative homotopy matching maps of $Z\to X$ are obtained by pulling back the relative homotopy matching maps of $Y \to \fY\by^h_{\fS}S$ along $X \to \fY\by^h_{\fS}S$. 

Now by Lemma \ref{resconv}, the $i$th such map  is 
a $\oC$-morphism, which is an equivalence for $i \ge n$, and  $(n-i)$-representable for $i \le n$. 
Since the matching maps of $X$ are $0$-geometric, we may  apply Proposition \ref{res} to $Z$, obtaining  $\tilde{X}\in s\cA$, with the $i$th relative homotopy matching map of $\tilde{X} \to Z$ being a $\oC$-morphism  for all $i$, and an equivalence for  $i \ge n$. 

We therefore conclude that the  $i$th matching map of $\pi:\tilde{X} \to X$ is  a  $\oC$-morphism for all $i$, and an equivalence for  $i \ge n$, so $\pi$ is a trivial homotopy $(n, \oC)$-hypergroupoid.
Projection $Z \to Y$ on the second factor gives the map $\tilde{f} :\tilde{X} \to Y$. 

Finally, observe that $Z \to X\by_SY$ is a pullback of the $(n-1)$-geometric map $\fY \to \fY\by^h_{\fS}\fY$, so the relative partial matching maps of $\tilde{X} \to  X\by_SY$ are $\oC$-morphisms, and are weak equivalences in levels above $n$, as required. 
\end{proof}

\begin{theorem}\label{duskinmor}
 Take  $S\in s\cA$ Reedy fibrant, and  take a morphism  $X \to S$  in $s\cA$. Take a  relative cocell $\tilde{X} \to X$ in trivial $(n, \oC)$-hypergroupoids in $\pro(s\cA)$,  with $\tilde{X}$ a $T\oC_n$-projective object in $\pro(s\cA)$ (see Definition \ref{tcncof}).  Then for all 
 $(n, \oC)$-hypergroupoids $f\co Y \to S$ in  $s\cA$, the natural map
\[
 \HHom_{\pro(s\cA)\da S}(\tilde{X},Y)\to  \oR\HHom_{\cS\da |S|}(|X|, |Y|)
\]
 is a weak equivalence.
\end{theorem}
\begin{proof}
Writing $L:= \HHom_{\pro(s\cA)\da S}(\tilde{X},Y)$ and   $R:= \oR\HHom_{\cS\da |S|}(|X|, |Y|) $, it suffices to show that the maps
\[
 L_m \to M_mL\by^h_{R^{\oR \pd \Delta^m}}R^{\oR \Delta^m}        
\]
are surjective on $\pi_0$ for all $m \ge 0$.

Now, for $K \in \bS$, we have
\begin{align*}
M_KL &= \HHom_{\pro(s\cA)\da S}( \tilde{X},Y^K\by_{S^K}S),\\
R^{\oR K} &\simeq  \oR\HHom_{\cS\da |S|}(|X|, |Y^K\by_{S^K}S|).    
\end{align*}
Fix a morphism $h\co\tilde{X} \to Y^{\pd\Delta^m}\by_{S^{\pd\Delta^m}}S$, and note that this factors through some $\tilde{X}_{\alpha} \in s\cA\da X$, where $\tilde{X} =\{\tilde{X}_{\alpha} \}_{\alpha}$. Also fix an element $g$ of $\pi_0(\{h\}\by^h_{R^{\oR \pd \Delta^m}}R^{\oR \Delta^m})$.

We now just apply Proposition \ref{duskinfull}, 
replacing $X$ with $\tilde{X}_{\alpha}$, $Y$ with $Y^{\Delta^m}$ and $S$ with $Y^{\pd\Delta^m}\by_{S^{\pd\Delta^m}}S$. 
This gives a trivial $(n, \oC)$-hypergroupoid $X' \to \tilde{X}_{\alpha}$, with an element $f \in \Hom_{s\cA \da (Y^{\pd\Delta^m}\by_{S^{\pd\Delta^m}}S)}(X', Y^{\Delta^m}) $ mapping to $g$. Since $\tilde{X}$ is $T\oC_n$-projective, we have a section $\tilde{X} \to X'$, so $f$ lifts to an element of $L_m\by_{M_mL}\{h\}$ over $g$, which gives the required surjectivity.
\end{proof}

\begin{remark}\label{holimrk2}
Following Remark \ref{holimrk}, note that the homotopy groups thus have the more natural characterisation
\begin{align*}
 \pi_0\oR\HHom_{\cS\da |S|}(|X|, |Y|)&\simeq \LLim_{X' \in \pi_0\cT_n(X) }\pi_0\HHom_{s\cA\da S}(X',Y),\\
\pi_n(\oR\HHom_{\cS\da |S|}(|X|, |Y|),f) &\simeq \LLim_{X' \in \pi_0\cT_n(X) }\pi_n(\HHom_{s\cA\da S}(X',Y),f),
\end{align*}
where the colimits are taken over the category  of trivial $(n, \oC)$-hypergroupoids $X'$ over $X$.
In the case of Artin stacks, the results of \S \ref{etcoversn} show that  it suffices to take the colimit over trivial Deligne--Mumford hypergroupoids over $X$.
\end{remark}

\subsection{From representables to affines}

In the case of strongly quasi-compact $n$-geometric stacks, recall that Definition \ref{hagndef} takes $\cA$ to be the essential image of the functor $\oR \uline{h} \co \Aff_{\C} \to \Aff_{\C}^{\sim, \tau}$. While Theorems \ref{relstrict} and \ref{duskinmor} characterise $n$-geometric stacks as $n$-hypergroupoids in $\cA$, a description as $n$-hypergroupoids in $\Aff_{\C}$ would be far more satisfactory. We now show that these two theories are indeed equivalent.

\begin{lemma}\label{sequiv}
 When $\cA\subset \Aff_{\C}^{\sim, \tau}$ is the pseudo-model category of representable stacks, the functor $\oR \uline{h} \co s\Aff_{\C} \to s\cA$ induces  a weak equivalence on the simplicial subcategories of fibrant cofibrant objects.
\end{lemma}
\begin{proof}
In the proof of \cite{hag2} Lemma 1.3.2.9,  it is shown that the essential image of $\oR \uline{h} \co \Ho(s\Aff_{\C}) \to \Ho(s\Aff_{\C}^{\sim, \tau})$ consists of simplicial stacks $F_{\bt}$ for which each $F_n$ is representable, i.e. strongly quasi-compact and $0$-geometric. In other words, the essential image is $\Ho(s\cA)$.

Since $\oR\HHom_{s\cA}(\oR \uline{h}X,\oR \uline{h}Y)$ and $\oR\HHom_{s\Aff_{\C}}(X,Y)$ can be expressed as homotopy limits of $\map_{\cA}(\oR \uline{h}X_m,\oR \uline{h}Y_n) \simeq \map_{\Aff_{\C}}(X_m,Y_n)$, we also have weak equivalences
\[
 \oR\HHom_{s\cA}(\oR \uline{h}X,\oR \uline{h}Y) \to \oR\HHom_{s\Aff_{\C}}(X,Y),
\]
giving equivalence of simplicial $\Hom$-spaces.
\end{proof}

\begin{proposition}\label{equivAff}
For $S \in s\Aff_{\C}$ Reedy fibrant, the functor $\oR \uline{h}$ induces a equivalence
\[
 \cG_n(\Aff_{\C}, \oC,S)\to \cG_n(\cA, \oC,\oR \uline{h} S)
\]
 of simplicial categories, for $\cG_n$ the localised category of $(n, \oC)$-hypergroupoids from Definition \ref{hgdscatdef}.
\end{proposition}
\begin{proof}
This is an immediate consequence of Lemma \ref{sequiv} once we note that the functor $Q$ of Proposition \ref{procofibrant} sends fibrant objects to inverse systems of fibrant cofibrant objects.  
\end{proof}

\begin{remark}
In particular, when $Y$ is an $(n, \oC)$-hypergroupoid in $s\Aff_{\C}$,  note that the simplicial presheaf $Y\co \Aff_{\C} \to \bS$  is determined by $Y(U)\simeq  \oR\HHom_{\cS}(\oR \uline{h} U, |Y|)$. By Theorem \ref{duskinmor}, this can be expressed as a direct limit of spaces $\HHom_{s\cA}(U',Y)$ for certain families of trivial $(n, \oC)$-hypergroupoids $U' \to U$. This gives us a characterisation of the sheafification functor $\oR \uline{h}$ depending only on $\oC$ (without reference to the topology $\tau$).
\end{remark}

Summarising the main results so far  gives the following.

\begin{theorem}\label{bigthm}
The simplicial category of $(n, \oC)$-geometric  stacks  is weakly equivalent to the 
simplicial category $\cG_n(\cA, \oC)$ whose objects are  $(n, \oC)$-hypergroupoids in the category $\cA$ of $0$-geometric  stacks, with morphisms given by
$$
\HHom_{\cG_n}(X,Y)= \HHom_{\pro(s\cA)}(QX,QY),
$$
for $Q$ the $T\oC_n$-projective replacement functor of Proposition \ref{procofibrant}.

The simplicial category of strongly quasi-compact $(n, \oC)$-geometric stacks  is weakly equivalent to the simplicial category $\cG_n(\Aff_{\C}, \oC)$ whose objects are   $(n, \oC)$-hypergroupoids in $\Aff_{\C}$, with morphisms given by 
$$
\HHom_{\cG_n}(X,Y)= \HHom_{\pro(s\Aff_{\C})}(QX,QY).
$$
\end{theorem}
\begin{proof}
The functor from $\cG_n(\cA, \oC, S)$ to $(n, \oC)$-geometric  stacks is given by geometric realisation $X \mapsto |X|$.
Essential surjectivity of this functor is given by Theorem \ref{relstrict}, while full faithfulness is Theorem \ref{duskinmor} combined with Lemma \ref{morsimpler}. The $\Aff_{\C}$ statements then follow from Proposition \ref{equivAff}.
\end{proof}

\begin{remarks}
 Observe that this categorisation of  $n$-geometric stacks gives  the same category for any  topology $\tau'$ with respect to which the class of $\oC$-morphisms is local. 

Also note that in the case of $n$-geometric Artin stacks, we can apply Corollary \ref{procofibrantet} and construct $Q$ using just \'etale (rather than smooth) covers.

For a simpler description of morphisms $\pi_0 \HHom$ in the homotopy category, see Remark \ref{holimrk2}.
\end{remarks}

\section{Quasi-coherent modules}\label{qucohsn}

Theorem \ref{relstrict} allows us to  replace  $(n,\oC)$-geometric stacks   with $(n, \oC)$-hypergroupoids. This means that we can apply the results of \cite{olssartin} concerning sheaves on simplicial algebraic spaces, thus extending their consequences from  Artin $1$-stacks to higher Artin stacks, and possibly beyond. 

However, it is easier to define quasi-coherent modules  as presheaves than as sheaves. This becomes especially important when we work with quasi-coherent complexes, which involve hyperdescent rather than descent.

\subsection{Left Quillen presheaves and Cartesian sections}

The following is taken from \cite{hirschowitzsimpson} \S 17:
\begin{definition}
Given a category $\bI$, define a \emph{left Quillen presheaf} $M$ on $\bI$  to consist of model categories $M(i)$ for all objects $i \in \bI$, together with left Quillen functors $f^*\co M(j) \to M(i)$ for all morphisms $f\co i \to j$ in $\bI$, satisfying  
associativity. Denote the right adjoint to $f^*$ by $f_*$.
\end{definition}

Note that we can regard $M$ as a functor from $\bI^{\op}$ to the category of categories.
\begin{definition}
Given a left Quillen presheaf $M$ on a category $\bI$, define the category  $M^{\bI^{\op}}$ of sections to consist of natural transformations $\id_{\bI^{\op}} \to M$. In other words, an object of $M^{\bI^{\op}}$ consists of objects $m(i) \in M(i)$ for all $i \in \bI$ and morphisms $\eta(f)\co f^*m(j) \to m(i)$ for all morphisms $f\co i \to j$ in $\bI$, with $\eta$ satisfying associativity. A morphism $g\co (m, \eta) \to (n, \zeta)$ consists of morphisms $g(i) \co m(i) \to n(i)$ for all $i \in \bI$, satisfying $\zeta(f) \circ f^*g(j)=g(i) \circ \eta(f)  $ for all  $f\co i \to j$ in $\bI$

A morphism $g$ in $M^{\bI^{\op}}$ is said to be a \emph{weak equivalence} if each $g(i)$ is a weak equivalence in $M(i)$.
\end{definition}

\begin{definition}
An object  $(m, \eta)$ of  $M^{\bI^{\op}}$ is said to be \emph{homotopy-Cartesian}  if the transformations $\eta(f)\co f^*m(j) \to m(i) $
 induce weak equivalences $\oL f^*m(j) \to m(i) $ for all morphisms $f$ in $\bI$. When the model structures on the model categories are all trivial, we refer to homotopy-Cartesian sections as \emph{Cartesian}. 

Write $ M^{\bI^{\op}}_{\cart}$ for the full subcategory of $M^{\bI^{\op}}$ consisting of homotopy-Cartesian sections, and refer to a morphism in $ M^{\bI^{\op}}_{\cart}$ as a weak equivalence if it is so in $M^{\bI^{\op}}$.
\end{definition}

\begin{definition}\label{classdef}
Given a category $\C$ and a subcategory $\cW$ (the \emph{weak equivalences}) containing all objects, we follow \cite{rezk}  3.3 in defining the classification diagram $\oN(\C):=\oN(\C,\cW)$ to be the following  simplicial space. 
 $N(\C,\cW)_n $ is the nerve of the category $w(\C^{[n]})$ whose objects are strings of morphisms of length $n$, with morphisms in $w(\C^{[n]})$ given by weak equivalences. 
\end{definition}

\begin{lemma}\label{holimcart}
 Given a left Quillen presheaf $M$ on a  category $\bI$, there is a canonical weak equivalence
\[
 \oN( M^{\bI^{\op}}_{\cart})\simeq \holim_{\substack{ \lla \\ n\in \bI^{\op}}} \oN(M^n) 
\]
of bisimplicial sets (with weak equivalences defined levelwise).
\end{lemma}
\begin{proof}
By \cite{rezk} Theorem 8.3, any Reedy fibrant replacement of $N(\C,\cW)$   is a complete Segal space. Thus we may consider homotopy limits and weak equivalences in the CSS model structure, so the result reduces to \cite{bergnerHolim} Theorem 4.1.
\end{proof}

\begin{definition}\label{cartdef}
Given a  left Quillen presheaf $M$ on $\cA$ and  $X \in s\cA$, write  $cM(X):= (X^*M)^{\Delta}$ and  $cM(X)_{\cart}:= (X^*M)^{\Delta}_{\cart}$.
\end{definition}

\begin{remarks}\label{cartrks}
 Note that  in order to establish that an object $\sF \in cM(X)$ lies in $cM(X)_{\cart}$,  it suffices to verify that the maps $\pd^i\co \oL \pd_i^* \sF^n \to \sF^{n+1}$  are weak equivalences, since  $\sigma^i\co \oL \sigma_i^* \sF^{n+1} \to \sF^n$  has right inverse $\oL \sigma_{i*}\pd^i$.

When each category $M(Z)$ has trivial model structure, Lemma \ref{truncate} adapts to show that the category  $cM(X)_{\cart}$ is equivalent to the category of pairs $(\sG, \omega)$, with $\sG \in M(X_0)$ and $\omega \co  \pd_0^*\sG \to \pd_1^*\sG$ satisfying the cocycle conditions $\sigma_0^*\omega=\id$ and $(\pd_2^*\omega) \circ (\pd_0^*\omega)= \pd_1^*\omega$.

When each $M(Z)$ is some category of sheaves on $Z$,   Definition \ref{cartdef} is essentially the same definition as \cite{Hodge3} 5.1.6. Thus we may also describe $cM(X_{\bt})$ as the category of sheaves on the site of pairs $(\mathbf{n}, U)$, for $\mathbf{n} \in \Delta^{\op}$ and $U \to X_n$.   
\end{remarks}

\subsection{Comparison with  sheaves on stacks}

Fix a homotopical algebra context in the sense of \cite{hag2} Definition 1.1.0.11. Thus we have a symmetric monoidal model category $\C$, and as in \cite{hag2} \S 1.3.2 let $\Aff_{\C}$  be opposite to the category of commutative monoids in $\C$. 
Now fix a topology $\tau$ on $\Aff_{\C}$  as in \S \ref{HAGcontext}. 

\begin{definition}
Say that a left Quillen presheaf $M$ on $\Aff_{\C}$ is a left Quillen hypersheaf  if for every $X \in \cA$ and every   $\tau$-hypercover $\pi\co \tilde{X}_{\bt} \to X$ in $s\cA$, the natural transformations
\begin{align*}
\sF &\to  \holim_{\substack{ \lla \\n \in \Delta}} \oR\pi_{n*} \pi_n^*\sF\\
 \oL\pi_n^* \Lim_{m \in \Delta} \pi_{m*}\sG^m &\to \sG^n
\end{align*}
are weak equivalences for all cofibrant objects $\sF \in  M(X)$ and all Reedy fibrant objects $\sG \in cM(\tilde{X})_{\cart}$. 
\end{definition}

\begin{proposition}\label{sheafprop}
Given a left Quillen hypersheaf  $M$ on $\Aff_{\C}$ and 
 $X \in s\Aff_{\C}$, there  are compatible equivalences
\[
 \oN(cM(X)_{\cart})_i \simeq  \oR\HHom_{\Aff_{\C}^{\sim, \tau}}(|X|, \oN(M)_i)  
\]
 of simplicial sets for all $i$.
\end{proposition}
\begin{proof}
 This is based on  \cite{hag2} Theorem 1.3.7.2.  The left Quillen hypersheaf condition ensures that each $\oN(M)_i$ is a hypersheaf. Since $|X|= \ho\LLim_n X_n$, this implies
\[
 \oR\HHom_{\Aff_{\C}^{\sim, \tau}}(|X|, \oN(M)_i)  \simeq    \holim_{\substack{ \lla \\n \in \Delta}}\oN(M(X_n))_i,    
\]
and the result reduces to Lemma \ref{holimcart}.
\end{proof}

\begin{remark}\label{modulesokrk}
In particular, note that Proposition \ref{sheafprop} combines with Theorem \ref{bigthm} to show that for any trivial  $(n,\oC)$-hypergroupoid $\pi\co \tilde{X}\to X$, the functor
\[
 \oL\pi^*\co   cM(X)_{\cart}\to cM(\tilde{X})_{\cart}     
\]
 induces weak equivalences of classifying diagrams. This makes it a weak equivalence of relative categories in the sense of \cite{BarwickKanEquiv}.
\end{remark}

\begin{definition}\label{moddef}
 For any object $\Spec A \in \Aff_{\C}$, there is a model category $\Mod(A)$ of $A$-modules in $\C$, so $\Mod$ defines a left Quillen presheaf on $\Aff_{\C}^{\op}$ as in \cite{hag2} \S 1.3.7.  
\end{definition}

\begin{corollary}\label{qcohequiv}
For all
 $X \in s\Aff_{\C}$, there  are compatible equivalences
\[
 \oN(c\Mod(X)_{\cart})_i \simeq  \oR\HHom_{\Aff_{\C}^{\sim, \tau}}(|X|, \oN(\Mod)_i)  
\]
 of simplicial sets for all $i$.
\end{corollary}

\begin{remark}\label{cfhagqcoh}
In \cite{hag2}, $\oN(\Mod)_0$ is denoted by $\mathbf{QCoh}$, and  $\oR\HHom_{\Aff_{\C}^{\sim, \tau}}(|X|, \oN(\Mod)_0) $ is regarded as nerve of the $\infty$-groupoid of 
quasi-coherent sheaves on $|X|$. 
Using the equivalences of \cite{bergner3}, we should thus regard the complete Segal space $ i \mapsto \oR\HHom_{\Aff_{\C}^{\sim, \tau}}(|X|, \oN(\Mod)_i)$ as the $\infty$-category of quasi-coherent sheaves on $|X|$, and the proposition above shows us that this is just $\oN(c\Mod(X)_{\cart})$. 
\end{remark}

\subsection{Cartesian replacement and derived direct images}

By applying Remark \ref{modulesokrk} to left Quillen hypersheaves with additive structure, we can define $\Ext$ groups and hence cohomology. However, the existence of (derived) direct images is more subtle in general. Although a morphism $f\co X \to Y$ in $s\cA$ will induce functors $\oR f_* \co cM(X)\to  cM(Y)$ for left Quillen hypersheaves $M$, these functors seldom preserve homotopy-Cartesian objects. We will therefore now construct a homotopy retraction of the inclusion 
$cM(X)_{\cart} \into cM(X)$.

\begin{definition}
Given $X_{\bt}\in s\cA$, let $(\Dec_+)_{\bt}X_{\bt}$ be the bisimplicial scheme given in level $n$ by $(\Dec_+)^{n+1}X_{\bt}$,  for $\Dec_+$ as in Definition \ref{decdef}. The simplicial structure  is defined  as in \cite{W} 8.6.4,  using  the comonadic structure of $\Dec_+$. This admits an augmentation $\pd_{\top}:(\Dec_+)_0 X_{\bt} \to X_{\bt}$, given by the co-unit. 

This also has an augmentation $\alpha$ in the orthogonal direction, given by $(\pd_0)^{\bt+1}:(\Dec_+)^{n+1}X_{\bt} \to ((\Dec_+)^{n}X_{\bt})_0=  X_n$, corresponding to the retraction of Lemma \ref{decgood}. 
\end{definition}

\begin{definition}\label{Rcartdef}
Given a left Quillen presheaf $M$ on $\cA$ and $X \in s\cA$, define
\[
 \oR \cart_* \co cM(X) \to cM(X)       
\]
to be the composition  
\[
(X^*M)^{\Delta} \xra{ \oL\pd_{\top}^*} (((\Dec_+)_{\bt}X_{\bt} )^*M)^{\Delta \by \Delta}   \xra{\oR \alpha_*} (cM(X))^{\Delta} \xra{\ho\Lim_{\Delta}} cM(X).      
\]
\end{definition}
Note that when $X$ has constant simplicial structure (i.e. $X \in \cA$), then $(\oL\pd_{\top}^*\sF)^{mn}= \sF^m$ and $\oR\alpha_* $ is the identity on $ M(X)^{\Delta \by \Delta}$, so  $\oR \cart_*$ is just given by the constant diagram $(\oR\cart_*\sF)^n= \ho\Lim_{m \in \Delta}\sF^m$.

We now take a HAG context $(\C, \oP, \tau)$, and consider $\cA= \Aff_{\C}$, letting $\oC \subset \oP$ consist of $\oP$-morphisms which are local $\tau$-surjections.

\begin{definition}\label{basechange}
Say that a left Quillen hypersheaf $M$ on $\Aff_{\C}$ \emph{satisfies $(n,\oC)$-base change} if for all $X \in \Aff_{\C}$, all  trivial $(n, \oC)$-hypergroupoids $\pi\co Y \to X$ and all $\oC$-morphisms $f\co Y \to X$, the natural map
\[
 f^* (\ho\Lim_{\Delta}\oR\pi_*)\sE \to (\ho\Lim_{\Delta}\oR \pi'_*)f^*\sE       
\] 
is a weak equivalence in $M(X') $ for all $\sE \in cM(Y)$, where $\pi'\co Y'\to X'$ is the pullback along $f$ of $\pi$. 
\end{definition}

The following is immediate:
\begin{lemma}\label{basechangelemma}
Assume that a left Quillen hypersheaf $M$ on $\Aff_{\C}$ satisfies the conditions:
\begin{enumerate}
        \item for all $\oC$-morphisms $\pi\co X \to Y$ and $f\co X' \to X$ in $\Aff_{\C}$, the natural map
\[
 f^* \oR\pi_*\sE \to \oR \pi'_*f^*\sE       
\]
is a weak equivalence in $M(X') $ for all $\sE \in M(Y)$, where $\pi'\co Y'\to X'$ is the pullback along $f$ of $\pi$;

\item for all   $\oC$-morphisms  $f\co X' \to X$ in $\Aff_{\C}$ and all $\sE \in cM(X)$, the natural map
\[
 f^*(\ho\Lim_{\Delta}\sE) \to   \ho\Lim_{\Delta} f^*\sE     
\]
is a weak equivalence in $M(X') $.
\end{enumerate}
Then  $M$  satisfies $(n,\oC)$-base change for all $n$.
\end{lemma}

\begin{proposition}\label{rcartgood}
Assume  that $M$ is a left Quillen hypersheaf  satisfying $(l,\oC)$-base change. Then for every $(l,\oC)$-hypergroupoid $X$ and every $\sF \in cM(X)$, the object $\oR \cart_*\sF$ of $cM(X)$  is homotopy-Cartesian.
\end{proposition}
\begin{proof}
Consider the  maps $\pd_i(\Dec_+): \Dec_+^{n+1}X \to \Dec_+^{n}X$ associated to the comonad $\Dec_+$.
  If we write $\vareps_Y: \Dec_+Y \to Y$ for the counit of the adjunction, then $\pd_i(\Dec_+)= (\Dec_+)^i\vareps_{\Dec_+^{n-i}Y}$.  
 
Since $\pd_i: X_{n} \to X_{n-1}$ is a $\oC$-fibration, $(l, \oC)$-base change implies that
$$
\oL\pd_i^* \ho\Lim_{\Delta}  \oR \alpha_{n-1,*} (\oL\pd_{\top}^{*})^{n}\sF   \simeq \ho\Lim_{\Delta}\oR \pr_{2*}\pr_1^*(\oL\pd_{\top}^{*})^{n}\sF,
$$  
for $\pr_1\co \Dec_+^{n}X\by_{X_{n-1}}X_n \to \Dec_+^{n}X$ and $\pr_2\co \Dec_+^{n}X\by_{X_{n-1}}X_n \to X_n$.

From Lemma \ref{pfibres}, it follows  that the morphisms 
$$
(\pd_i(\Dec_+), \pd_0): \Dec_+^{n+1}X \to \Dec_+^{n}X\by_{X_{n-1}}X_n
$$ 
are trivial  $(l,\oC)$-hypergroupoids. 
In particular, they are $\tau$-hypercoverings, so Remark \ref{modulesokrk} gives that 
\[
 \oL(\pd_i(\Dec_+), \pd_0)^*\co cM(\Dec_+^{n}X\by_{X_{n-1}}X_n)_{\cart} \to cM(\Dec_+^{n+1}X)_{\cart}
\]
is a weak equivalence. 

Now, $\ho\Lim_{\Delta} \oR \alpha_{n*}\co cM(\Dec_+^{n+1}X)_{\cart} \to M(X_n)$  is homotopy right adjoint to $\oL \alpha_n^*$, so the weak equivalence above gives $\ho\Lim_{\Delta} \oR \alpha_{n*}\circ \oL(\pd_i(\Dec_+), \pd_0)^*$ as homotopy right adjoint to $\oL\pr_2^*\co \Mod(X_n) \to cM(\Dec_+^{n}X\by_{X_{n-1}}X_n)_{\cart}$. 
Another homotopy right adjoint is given by $ \ho\Lim_{\Delta}\oR \pr_{2*}$, so 
\[
 \ho\Lim_{\Delta}\oR \pr_{2*} \simeq \ho\Lim_{\Delta} \oR \alpha_{n*}\circ \oL(\pd_i(\Dec_+), \pd_0)^*.
\]

Combining the equivalences above, we have
\begin{eqnarray*}
 \oL\pd_i^* \ho\Lim_{\Delta}  \oR \alpha_{n-1,*} (\oL\pd_{\top}^{*})^{n}\sF&\simeq& \ho\Lim_{\Delta} \oR \alpha_{n*}\circ \oL(\pd_i(\Dec_+), \pd_0)^*\pr_1^*(\oL\pd_{\top}^{*})^{n}\sF\\
&=& \ho\Lim_{\Delta} \oR \alpha_{n*}(\oL\pd_{\top}^{*})^{n+1}\sF,
\end{eqnarray*}
so we have shown that
$$
\pd^i: \oL \pd_i^*(\oR \cart_*(\sF)^{n-1}) \to \oR \cart_*(\sF)^n
$$
is a quasi-isomorphism, as required.
\end{proof}

\begin{proposition}\label{rcartgood2}
For $X$ and $M$ as above and $\sF \in cM(X)$, there is a natural transformation $\vareps_{\cart}:\oR \cart_*\sF\to \sF$, which is a weak equivalence whenever $\sF$ is homotopy-Cartesian.
\end{proposition}
\begin{proof}
Given a module $\sG\in cM(Y)$ for  $Y\in s\Aff_{\C}$, we may define $\Dec^+(\sG) \in cM(\Dec_+Y)$ by $(\Dec^+\sG)^m= \sG^{m+1}$, with $\pd^i_{\Dec^+\sG}= \pd^i_{\sG}$ and $\sigma^i_{\Dec^+\sG}= \sigma^i_{\sG}$.
Setting 
$$((\Dec^+)^{\bt}\sF)^n:= (\Dec^+)^{n+1}\sF
$$ for any $\sF\in cM(X)$, this gives us 
 $(\Dec^+)^{\bt}\sF \in M((\Dec_+)_{\bt}X)^{\Delta \by \Delta}$, and there is a  natural map 
$
\pd^{\top}: \oL\pd_{\top}^*\sF \to (\Dec^+)^{\bt}\sF,
$
which is a  weak equivalence whenever $\sF$ is homotopy-Cartesian. Applying the functor $\ho\Lim_{\Delta}  \oR \alpha_*$ gives a map
$
\pd^{\top}:\oR \cart_*\sF \to \ho\Lim_{\Delta}  \oR \alpha_*(\Dec^+)^{\bt}\sF.
$

Now, 
 the deformation retraction $\pd_0:\Dec_+Y \to Y_0$ corresponds to a deformation retraction
$$
\sG^0 \xra{\pd^0} \ho\Lim_{\Delta}  \oR (\pd_{0*}^{\bt})_* \Dec^+(\sG) \xra{ \oR\pd_{0*}^{\bt}(\sigma^0)^{\bt}} \sG^0
$$
in  $cM(Y_0)$. Applying this to $Y= (\Dec_+)^nX$ and $\sG= (\Dec^+)^n\sF$ for all $n$ gives a deformation retraction 
$$
\sF \xra{\pd^0} \oR\alpha_*(\Dec^+)^{\bt}\sF   \xra{\oR \alpha_*(\sigma^0)^{\bt}} \sF.
$$

We now define $\vareps_{\cart}$ to be
$
\oR\alpha_*(\sigma^0)^{\bt}\circ\pd^{\top}: \oR \cart_*\sF \to \sF.
$
\end{proof}

\begin{definition}\label{ddirectdef}
For any morphism $f: X \to Y$ in $s\Aff_{\C}$ with $Y$ an $(n, \oC)$-hypergroupoid, and  for any  left Quillen hypersheaf $M$  satisfying $(n,\oC)$-base change, we may now  define $\oR f_*^{\cart}: cM(X)_{\cart}\to cM(Y)_{\cart}$ by 
$$
\oR f_*^{\cart}:= \oR \cart_*\circ \oR f_*,
$$ 
where $\oR f_*:cM(X)\to cM(Y)$ is given on $X_n$ by $\oR f_{n*}$. 

When $n=0$, note that we just have $\oR f_*^{\cart}= \ho\Lim_{\Delta} \oR f_*$.
\end{definition}

\subsection{Sheaves on Artin $n$-stacks}

We now specialise to the context of strongly quasi-compact Artin $n$-stacks. If we let $\oC$ be the class of smooth surjections of affine schemes, then Theorem \ref{relstrict} allows us to work instead with Artin $n$-hypergroupoids (i.e. $(n,\oC)$-hypergroupoids in affine schemes --- see Examples \ref{cfddt}). Write $\Aff$ for the category of affine schemes.

\subsubsection{Quasi-coherent modules}

As in Definition \ref{moddef}, let $\Mod\co \Aff^{\op} \to \Cat$ send $\Spec A$ to the category of $A$-modules. The following is now a special case of Remarks \ref{cartrks}:

\begin{lemma}
Given an Artin $n$-hypergroupoid $X$,  the category $c\Mod(X)_{\cart}$ is equivalent to the category of pairs $(\cG, \omega)$, where $\cG$ is a quasi-coherent sheaf on $X_0$ and  $\omega \co  \pd_0^*\sG \to \pd_1^*\sG$ is a morphism of quasi-coherent sheaves on $X_1$  satisfying the cocycle conditions $\sigma_0^*\omega=\id$ on $X_0$, and $(\pd_2^*\omega) \circ (\pd_0^*\omega)= \pd_1^*\omega$ on $X_2$.
\end{lemma}

When $X$ is an Artin $1$-hypergroupoid and $\fX=|X|$,  we thus have that  $c\Mod(X)_{\cart}$ is equivalent to the category of Cartesian $\O_{\fX}$-modules in the sense of \cite{olssartin} Lemma 4.5.

\subsubsection{Quasi-coherent complexes}

Now consider the functors $dg\Mod$ (resp. $dg_+\Mod$, resp. $dg_-\Mod$) from  $\Aff^{\op}$ to $\Cat$  which send $\Spec A$ to the model category  of chain complexes (resp. chain complexes in non-negative degrees, resp. chain complexes in non-positive degrees) of $A$-modules. Weak equivalences in all these categories are quasi-isomorphisms. Fibrations in $dg\Mod(A)$ and $dg_-\Mod(A)$ are surjections, while fibrations in $dg_+\Mod(A)$ are surjective in non-zero degrees.

The following is an immediate consequence of the definitions and flat base change:
\begin{lemma}
Given an Artin $n$-hypergroupoid $X$,  an object $\sF_{\bt} \in  cdg\Mod(X)$ is homotopy-Cartesian if and only if  the  
homology presheaves $\H_i(\sF_{\bt}) \in c\Mod(X)$ are Cartesian for all $i$.
\end{lemma}

\begin{lemma}\label{dgmodhyp}
The functors $dg\Mod, dg_+\Mod, dg_-\Mod$ are left Quillen hypersheaves on $\Aff$ with respect to the fpqc topology. 
\end{lemma}
\begin{proof}
We first consider the case of $dg\Mod$.
Taking an fpqc hypercover $p\co \tilde{X}_{\bt} \to X$, we need to show that the maps
 \begin{align*}
\sF &\xra{\eta}   \holim_{\substack{ \lla \\n \in \Delta}}\oR p_{n*}  p_n^*\sF\\
 \oL p_n^* \Lim_{m \in \Delta}  p_{m*}\sG^m &\xra{\vareps} \sG^n
\end{align*}
are weak equivalences for all cofibrant objects $\sF \in dg\Mod(X)$ and all fibrant objects $\sG \in cdg\Mod(\tilde{X})_{\cart}$.

All objects of $dg\Mod(\tilde{X}_n)$ are fibrant, so $ p_{n*}= \oR p_{n*}$. Since the morphisms $ p_n \co \tilde{X}_n \to X$ are all flat, $ p_n^*$ also preserves weak equivalences, so $\oL p_n^*=  p_n^*$.

As in Examples \ref{categs}, a model for the functor $\ho\Lim_{n \in \Delta}$ is given by the product total complex.  Writing $\H_c^*$ for cohomology of a cosimplicial complex, and taking $\sG \in cdg\Mod(\tilde{X})_{\cart}$, there is spectral sequence
$$
\H_c^i p_*\H_j\sG \abuts \H_{j-i} \Tot^{\Pi} p_*\sG, 
$$
which converges weakly, by \cite{W} \S 5.6, since $\H_c^i p^*\sF_j=0 $ for $i<0$ (and hence for $i<0, j<0$). 

Since $\H_j\sG \in \Mod(\tilde{X})_{\cart}$, by faithfully flat descent there exists $\sE_j=\H_c^0( p_*\H_j\sG)  \in dg\Mod(X)$ with $ p^*\sE_j \cong \H_j\sG$. Thus cohomological descent gives  $\H_c^i p_*\H_j\sG=0$ for all $i>0$, and the spectral sequence converges, giving
$$
\H_{n} \Tot^{\Pi}  p_*\sG\cong \H_c^0( p_*\H_n\sG).
$$
the isomorphism $ p^*\sE_j \cong \H_j\sG$ thus ensures that
\[
 \vareps_{\sG}\co  p^*\H_c^0( p^*\sG) \to \sG 
\]
 is a quasi-isomorphism. 

For $\sF \in dg\Mod(X)$, this implies that $\vareps_{ p^*\sF}$ is a quasi-isomorphism. Thus $ p^*\eta_{\sF}$ is also a quasi-isomorphism, since $\vareps_{ p^*\sF}\circ  p^*\eta_{\sF}= \id_{\sF}$. In particular, $ p_0^*\eta_{\sF}$ is a quasi-isomorphism, so $\eta_{\sF}$ must also be so,  $ p_0$ being faithfully flat. 

This completes the proof for $dg\Mod$. The same proof applies to $dg_-\Mod$, in which case the product total complex is just a direct sum. For $dg_+\Mod$, $\ho\Lim_{\Delta}$ is given by applying the good truncation $\tau_{\ge 0}$ to $\Tot^{\Pi}$, so the proof above still applies, except that we need only consider $\H_n$ for $n \ge 0$.
\end{proof}

\begin{remarks}\label{hcartrks}
Note that $cdg\Mod(X)_{\cart}$ corresponds to the category  of quasi-coherent complexes on $|X|$ given  in
\cite{toenseattle} 3.7, which anticipated the simplicial characterisation of $n$-geometric stacks given in Theorem \ref{relstrict}. 

However, this differs crucially from the construction of $D_{\cart}(\O_{X_{\bt}})$ in \cite{olssartin} 4.6. The latter takes complexes $\sF_{\bt}$ of sheaves of $\O_{X}$-modules whose homology \emph{sheaves} $\sH_n(\sF)$ are Cartesian. Thus the sheafification $\sF^{\sharp}$ of any homotopy-Cartesian $\O_X$-module $\sF$ lies in $D_{\cart}(\O_{X_{\bt}})$, but $\sF^{\sharp}$ might not be homotopy-Cartesian. 
One consequence is that whereas projective resolutions exist in the category of quasi-coherent complexes, they do not exist in $D_{\cart}(\O_{X_{\bt}})$. This means that derived pullbacks $\oL f^*$ exist automatically for quasi-coherent complexes (as will be exploited in Corollary \ref{cotgood}), but are constructed only with considerable difficulty in \cite{olssartin}.

Objects of $cdg_+\Mod(X)_{\cart}$ correspond to the definition of  homotopy-Cartesian modules in \cite{hag2} Definition 1.2.12.1. These are called quasi-coherent complexes in \cite{lurie} \S 5.2. Again, this differs slightly from the construction of $D_{\cart}(\fX)$ in \cite{olssartin} 3.10, which takes complexes  of sheaves of $\O_{\fX}$-modules whose homology \emph{sheaves}  are quasi-coherent. The differences arise because quasi-coherent complexes are hypersheaves (as explained in \S \ref{inftysheaves}) rather than sheaves.

For a quasi-compact semi-separated scheme $X$, it is observed  in the introduction of \cite{huettemann} that the  proof of  \cite{huettemann} Theorem 4.5.1 adapts to show that the homotopy category of quasi-coherent complexes $X$ is equivalent to the homotopy category of complexes of quasi-coherent sheaves on $X$. By \cite{bokstedtneeman}, this in turn is equivalent to $D_{\cart}(\O_X)$ under the same hypotheses. 
\end{remarks}

\subsubsection{Derived direct images and cohomology}\label{directsn}

Observe that  the left Quillen hypersheaves $dg_-\Mod, dg_+\Mod, dg\Mod$ on $\Aff$ all satisfy  the first condition of Lemma \ref{basechangelemma}, which is just flat base change. The second condition, however, is only satisfied by $dg_-\Mod$ in general, since $\Tot^{\Pi}$ does not commute with pullbacks. Thus, for any morphism $f \co X \to Y$ in $s\Aff$ with $Y$ an Artin $n$-hypergroupoid, Definition \ref{ddirectdef} gives
\[
 \oR f_*^{\cart} \co  cdg_-\Mod(X)_{\cart} \to cdg_-\Mod(Y)_{\cart},      
\]
homotopy right adjoint to $\oL f^*$.

When $Y$ is an affine scheme, or even the \v Cech resolution of a quasi-compact semi-separated scheme, the total product $\Tot^{\Pi}$ becomes a finite sum, giving
\[
 \oR f_*^{\cart} \co  cdg\Mod(X)_{\cart} \to cdg\Mod(Y)_{\cart},       
\]
 and similarly for $dg_+\Mod$.

Taking $Y$ to be affine, we can write $\oR \Gamma(X,-)$ for $\oR f_*^{\cart}$, giving
\[
\oR \Gamma(X, \sF)= \Tot^{\Pi} \sF.        
\]

We also have the following  generalisation of the vanishing of higher direct images of an affine morphism of schemes, allowing us to regard $f_*$ as $\oR f_*^{\cart}$.
\begin{lemma}\label{rcart0}
If $f: X \to Y$ is a relative Artin $0$-hypergroupoid over an Artin $n$-hypergroupoid $Y$, then 
$$
\oR f_* \simeq f_*: cdg\Mod(X) \to cdg\Mod(Y)
$$
sends
$cdg\Mod(X)_{\cart}$ to $cdg\Mod(Y)_{\cart}$.
\end{lemma}
\begin{proof}
Since $f$ is a relative Artin $0$-hypergroupoid, it is Cartesian, so $f_*M$ is Cartesian for all $M \in cdg\Mod(X)_{\cart}$. As all objects of $dg\Mod(A)$ are fibrant, we also have $\oR f_*= f_*$, which  completes the proof.
\end{proof}

\section{Alternative formulations of derived Artin stacks}\label{alternatives}
In this section, we will restrict to the strongly quasi-compact case in order to simplify the exposition. 

Intuitively, a derived Artin $n$-stack $X$ should be a small neighbourhood of the associated Artin $n$-stack $\pi^0X$ given in Definition \ref{h0def}. We will see how any homotopy derived Artin $n$-hypergroupoid $X$ is equivalent to the Zariski neighbourhood $X^{l}$ of $\pi^0X$ in $X$, and even the \'etale neighbourhood $X^{h}$ of $\pi^0X$ in $X$. Indeed, under fairly strong Noetherian hypotheses, $X$ is  equivalent to the formal neighbourhood $\hat{X}$.

\subsection{Approximation and completion}\label{fthick}

We use the terms \emph{formally \'etale}  and \emph{formally smooth}  in the sense of \cite{hag2}, meaning that a morphism $f:A \to B$ of simplicial rings is formally \'etale (resp. formally smooth) if the cotangent complex $\bL_{\bt}^{B/A}\simeq 0$ (resp. $\bL_{\bt}^{B/A}$ is equivalent to a retract of a direct sum of copies of $B$). Note that when $A$ and $B$ are discrete rings, these notions are stronger than those used classically (e.g. in \cite{Mi}).

\begin{definition}\label{strongdef}
Recall from \cite{hag2} Definition 2.2.2.1 that a morphism $f:A \to B$ in $s\Ring$ is said to be strong if the maps $\pi_n(A)\ten_{\pi_0(A)}\pi_0(B) \to \pi_n(B)$ are isomorphisms for all $n$.
\end{definition}

\begin{lemma}\label{feteasy}
A morphism $f:A \to B$ in $s\Ring$ is formally \'etale if and only if  the relative cotangent complex $\bL^{B/A}\ten_B \pi_0B$ is contractible as a simplicial $\pi_0(B)$-module, i.e $\pi_*(\bL^{B/A}\ten_B \pi_0B)=0$.
\end{lemma}
\begin{proof}
The argument from the proof of \cite{hag2} Lemma 2.2.2.8 shows that contractibility of $\bL^{B/A}\ten_B \pi_0B$ implies contractibility of the simplicial $B$-module $\bL^{B/A}$, which in turn means that $f$ is formally \'etale.
\end{proof}

\begin{lemma}\label{weakfet}
A morphism $f:A \to B$ in $s\Ring$ is a weak equivalence if and only if $f$ is formally \'etale and $\pi_0(f)$ is an isomorphism.
\end{lemma}
\begin{proof}
 \cite{hag2} Proposition 2.2.2.4 implies that a formally \'etale map $f$ is \'etale if and only if $\pi_0f$ is finitely presented. Since $\pi_0f$ is an isomorphism, this holds, so $f$ is \'etale. Finally, \cite{hag2} Theorem 2.2.2.6 states that \'etale morphisms $f$ in $s\Ring$ are precisely strong morphisms for which $\pi_0f$ is \'etale. This implies that $f$ is a weak equivalence.
\end{proof}

\begin{lemma}\label{levelwisefet}
If $f:A \to B$ is a morphism in $s\Ring$, for which each map $f_n:A_n \to B_n$ is formally \'etale (in the sense of \cite{hag2}), then $f$ is formally \'etale (i.e. the cotangent complex $\bL^{B/A}$ is contractible). In fact, we only need $\bL^{B_n/A_n}\ten_{B_n}\pi_0B$ to be contractible for all $n$. 
\end{lemma}
\begin{proof}
Let $\tilde{B}_{n\bt}$ be the cotriple resolution of $B_n$ as an $A_n$-algebra. Thus $\tilde{B}_{\bt\bt}$ is a bisimplicial ring, and $A \to \diag \tilde{B}_{\bt\bt}$ is a cofibrant approximation to $B$ in $A \da s\Ring$.

Therefore
$$
\bL^{B/A}\ten_B \pi_0B \simeq \Omega(\diag \tilde{B}_{\bt\bt}/A_{\bt})\ten_{\diag \tilde{B}_{\bt\bt}}\pi_0B,
$$
and this is the diagonal of the bisimplicial diagram
$$
M_{nm}:= \Omega(\tilde{B}_{nm}/A_{n})\ten_{\tilde{B}_{nm}}\pi_0B
$$

Since $f_n$ is formally \'etale, 
$$
L_n:=\pi_*M_{n\bt} =\pi_*(\bL^{B_n/A_n}\ten_{B_n}\pi_0B)= 0,
$$ 
for all $n$. Since $\pi_*(\diag M_{\bt\bt})\cong \H_*(\Tot M_{\bt\bt})$ (Eilenberg-Zilber), we have a convergent
spectral sequence
$$
0 =\pi_*(0)=\pi_*(L_{\bt}) \abuts \pi_*(\diag M_{\bt\bt}),
$$
so $f$ is formally \'etale by Lemma \ref{feteasy}.
\end{proof}

\begin{definition}
Given a cosimplicial affine scheme $X$, define $\hat{X}\in c\Aff$ by setting $\hat{X}^n$ to be the formal neighbourhood of $\pi^0X$ in $X^n$, so $O(\hat{X}^n ) = \Lim_m O(X^n)/(I_n)^m$, where $I_n= \ker(O(X^n) \to \pi_0O(X))$.    
\end{definition}

\begin{proposition}\label{fthm}
For $X \in c\Aff$ levelwise Noetherian, the morphism $f:\hat{X} \to X$ is a weak equivalence.
\end{proposition}
\begin{proof}
Write $X^n = \Spec A_n$ and $\hat{X}^n= \Spec \hat{A}_n$.
We first describe the cotangent complex of  $f_n: A_n \to \hat{A}_n$. Since $A_n$ is Noetherian, \cite{Mat} Theorem 8.8 implies that $f_n$ is flat. Thus flat base change shows that 
$$
\bL^{\hat{A}_n/A_n}\ten_{A_n} \pi_0A\simeq  \bL^{(\hat{A}_n\ten_{A_n}\pi_0A)/\pi_0A},
$$ 
and note that since $\bL^{\hat{A}_n/A_n}\ten_{A_n} \pi_0A = \bL^{\hat{A}_n/A_n}\ten_{\hat{A}_n}( \hat{A}_n\ten_{A_n} \pi_0A)$, we then have 
$$
\bL^{\hat{A}_n/A_n}\ten_{\hat{A}_n} \pi_0A =  \bL^{(\hat{A}_n\ten_{A_n}\pi_0A)/\pi_0A}\ten_{ (\hat{A}_n\ten_{A_n}\pi_0A)}\pi_0A.
$$
\cite{Mat} Theorem 8.7 implies that $\hat{A}_n\ten_{A_n}\pi_0A= \pi_0A$, so we have shown that
$$
\bL^{\hat{A}_n/A_n}\ten_{\hat{A}_n} \pi_0A \simeq 0,
$$
and hence that $f$ is formally \'etale, by Lemma \ref{levelwisefet}.

In order to satisfy the conditions of Lemma \ref{weakfet}, it only remains to prove that $\pi_0\hat{A}= \pi_0A$. This follows by letting $I= \ker(A \to \pi_0A)$, and observing that the surjections $I^{\ten n} \to I^n$ imply that $\pi_0(I^n)=0$ for all $n>0$, so $ \pi_0(\Lim A/I^n)= \pi_0A$.
\end{proof}

Note that \cite{Mat} Theorem 8.12 implies that $\hat{X}$  is then also  levelwise Noetherian. See \cite{nishimura} for more properties preserved by $I$-adic completion.

\begin{corollary}
Any levelwise Noetherian homotopy derived Artin $n$-hypergroupoid $X$ is weakly equivalent to the homotopy derived Artin $n$-hypergroupoid $\hat{X}$, defined by setting $(\hat{X})_n:= \widehat{(X_n)}$.
\end{corollary}

\subsection{Local thickenings}\label{locthick}

\begin{definition}
Given a morphism $ f:X \to Y$ of affine schemes, define $(X/Y)^{\loc}$ to be the localisation of $Y$ at $X$. By \cite{anel} Proposition 52, this is given by
$$
(X/Y)^{\loc}= \Spec \lim_{\substack{ \lra \\ f(X) \subset U \subset Y \\ U \text{ open}}} \Gamma(U, \O_Y)= \Spec \Gamma(X, f^{-1}_{\Pr}\O_{Y, \Zar}),
$$ 
where $f_{\Pr}^{-1}$ is the presheaf inverse image functor.
\end{definition}

\begin{lemma}\label{locprop}
If $X_0 \to X$ is a morphism of affine schemes, and $X^l:= (X_0/X)^{\loc}$, then $X_0\by_XX^l = X_0$.
\end{lemma}
\begin{proof}
Given any affine scheme $Y$ under $X_0$, write $Y^l$ for the localisation $(X_0/Y)^{\loc}$. Properties of localisation from \cite{anel} give 
$$
(X_0\by_XX^l)^l = X_0^l\by_{X^l}X^l=  X_0,  
$$
so the morphism $X_0 \to X_0\by_XX^l$ is conservative (i.e. an inverse limit of open immersions), and hence a monomorphism. Since $X^l \to X$ is also conservative, the projection $X_0\by_XX^l \to X_0$ is also a  monomorphism, so must in fact be an isomorphism.
\end{proof}

\begin{definition}
Given a cosimplicial affine scheme $X$, define $X^{l}\in c\Aff$ by $(X^l)^n:= (\pi^0X/X^n)^{\loc}$.
\end{definition}

\begin{proposition}\label{lthm}
For $X \in c\Aff$, the morphism $f:X^l \to X$ is a weak equivalence.
\end{proposition}
\begin{proof}
Given any affine scheme $Y$ under $\pi^0X$, write $Y^l$ for the localisation $(\pi^0X/Y)^{\loc}$.

By construction, the maps $f^n:(X^l)^n \to X^n$ have trivial cotangent complexes (being  filtered limits of open immersions), so Lemma \ref{levelwisefet} implies that $f$ is formally \'etale. By Lemma \ref{weakfet}, it only remains to prove that $\pi^0f$ is an isomorphism. 

Observe that the functor $Y \mapsto Y^l$ from affine schemes under $\pi^0X$ to localisations preserves limits, since it has a right adjoint given by inclusion. Thus $\pi^0(X^l)$, being the equaliser of the maps $\pd^0, \pd^1: (X^l)^0 \to (X^l)^1$, is just the localisation of the equaliser $\pi^0X$ of  $\pd^0, \pd^1: X^0 \to X^1$. Thus
$$
\pi^0(X^l) = (\pi^0X)^l = \pi^0X.
$$
\end{proof}

\begin{corollary}
Any homotopy derived Artin $n$-hypergroupoid $X$ is weakly equivalent to the homotopy derived Artin $n$-hypergroupoid $X^l$, defined by setting $(X^l)_n:= (X_n)^l$.
\end{corollary}

\begin{remark}\label{locmodel}
In fact, $s\Ring^l$ has a simplicial model structure, in which a morphism is a weak equivalence or fibration whenever the underlying map in $s\Ring$ is so. 
\end{remark}

\begin{theorem}\label{lshfthm}
Fix a  scheme $Z$ over a ring $R$, and fix $m \ge 2$ (or $m\ge 1$ if $Z$ is semi-separated). Then the  homotopy category of $m$-geometric derived schemes $X$ over $R$  with $\pi^0X \simeq Z$ is weakly equivalent to the   homotopy category $\cT_{\Zar}(Z)$  of  presheaves $\sA_{\bt}$ of simplicial $R$-algebras on  the category $\Aff_{\Zar}(Z)$ of affine open subschemes of $Z$,  satisfying the following:
\begin{enumerate}
\item $\pi_0(\sA_{\bt})= \O_{Z}$;
\item for all $i$, the presheaf $\pi_i(\sA_{\bt})$ is a quasi-coherent $\O_{Z}$-module.
\end{enumerate}
\end{theorem}
\begin{proof}
The key idea behind this proof is that given a local morphism $B \to A$, we can define a presheaf $(B/\O_Y)^{\loc}$ on $Y:=(\Spec A)_{\Zar}$, by $(B/\O_Y)^{\loc}(\Spec C)= (B/C)^{\loc}$.

Recall from Remark \ref{cflurie} that all schemes are $2$-geometric, and that a quasi-separated scheme is $1$-geometric if and only if it is semi-separated. Theorem \ref{relstrict} therefore gives a  Zariski $m$-hypergroupoid $\tilde{Z}_{\bt}$ in disjoint unions of affine schemes (see Examples \ref{cfddt}), with with $a \co \tilde{Z} \to Z$ such that  $|\tilde{Z}|\simeq Z$.

Now, the proof of Lemma \ref{dgmodhyp} adapts to show that $\sA_{\bt}$ is a Zariski hypersheaf. 
 Thus $a^{-1}$ gives a weak equivalence
from  $\cT_{\Zar}(Z)$ to the homotopy category $\cT_{\Zar}(\tilde{Z}_{\bt})$ of   homotopy-Cartesian presheaves $\sB$  of those simplicial algebras on the simplicial site $\Aff_{\Zar}(\tilde{Z}_{\bt})$ for which $\pi_0\sB_{\bt}\cong \O_{\tilde{Z}}$ and $\pi_n\sB_{\bt}$ is  a Cartesian $\O_{\tilde{Z}}$-module. 

We then define the simplicial cosimplicial scheme ${\tilde{X}}_{\bt}$ by letting  ${\tilde{X}}_n$ be the cosimplicial  scheme  $ \oSpec_{\tilde{Z}_n}\sB_{\bt}^n$. It follows immediately from the properties of $\sB_{\bt}$ that ${\tilde{X}}_{\bt}$  is a    homotopy derived Zariski $n$-hypergroupoid, and that $\pi^0{\tilde{X}} =\tilde{Z}$. Therefore $X(Z, \sA_{\bt}):= {\tilde{X}}^{\sharp}$ has all the required properties. 

If we replace $\tilde{Z}$ by a trivial relative Zariski $m$-hypergroupoid $\tilde{Z}' \to \tilde{Z}$ in the above construction, we get a homotopy trivial derived  Zariski $n$-hypergroupoid ${\tilde{X}}' \to {\tilde{X}}$, so $({\tilde{X}}')^{\sharp}\simeq {\tilde{X}}^{\sharp}$, and the functor $X(-)$ is  independent of the choice of $\tilde{Z}_{\bt}$. 

For  the inverse construction, take a  derived $m$-geometric scheme $X$ with $\pi^0X \simeq Z$, and apply  Theorem \ref{relstrict} to form a derived Zariski $m$-hypergroupoid ${\tilde{X}}$ with ${\tilde{X}}^{\sharp} \simeq X$. Let $\tilde{Z}:=\pi^0{\tilde{X}}$, and  for $\iota : \tilde{Z} \to {\tilde{X}}$, consider the presheaf $\sB_{\bt}:= \iota^{-1}_{\Pr}\O_{{\tilde{X}}}$ on the simplicial site $\Aff_{\Zar}(\tilde{Z}_{\bt})$ --- this is a Zariski presheaf of simplicial algebras. Since $\iota^{-1}_{\Pr}$ commutes with colimits, $\pi_0(\sB_{\bt})= \O_{\tilde{Z}}$.  Properties of localisation imply that $(\sB_n)(U)= (O({\tilde{X}}^n)/O(U))^{\loc}$, which is  \'etale (as in the proof of Proposition \ref{lthm}), so 
$$
\pi_i(\sB_{\bt}(U)) \cong \pi_0(\sB_{\bt}(U))\ten_{\pi_0(\sB_{\bt}(\tilde{Z}))}\pi_i(\sB_{\bt}(\tilde{Z}))=  \O_{\tilde{Z}}(U)\ten_{\O_{\tilde{Z}}(\tilde{Z})}\pi_i(\sB_{\bt}(\tilde{Z})),
$$
which means that $\pi_i(\sB_{\bt})$ is homotopy-Cartesian and quasi-coherent.

This allows us to  define a Zariski presheaf $\sA_{\bt}(X):= \oR a_* \sB_{\bt}$, for $a: \tilde{Z}_{\bt} \to Z$. If we replace ${\tilde{X}}$ by a homotopy trivial relative Zariski $n$-hypergroupoid ${\tilde{X}}' \to {\tilde{X}}$ in the above construction, we get a weakly equivalent presheaf, so $\sA_{\bt}(X)$ is well-defined.

Moreover, $a^{-1} \sA_{\bt}(X) \simeq \sB_{\bt}$, and $\oSpec_{\tilde{Z}_n}\sB_{\bt}^n= {\tilde{X}}_n^l$, 
so 
$$
X(Z, \sA_{\bt}(X)) \simeq ({\tilde{X}}^l)^{\sharp} \simeq X, 
$$
by Proposition \ref{lthm}. 

Conversely, to compute $\sA_{\bt}(X(\sA_{\bt}))$, we may choose the resolution ${\tilde{X}}_{\bt}$ of $X$ used in the construction of $X(\sA_{\bt})$, and then we have 
$$
\sA_{\bt}(X(\sA_{\bt}))=  \oR a_* ( \Gamma(\tilde{Z}, a^{-1}\sA_{\bt})/\O_{\tilde{Z}})^{\loc};
$$
which is weakly equivalent to $\oR a_* a^{-1}\sA_{\bt}$ by Proposition \ref{lthm}. This in turn is weakly equivalent to $\sA_{\bt}$, as $\sA_{\bt}$ is a hypersheaf. 
\end{proof}

\begin{remark}\label{lshfrk}
In fact, since $\sA_{\bt}$ is a hypersheaf, we can replace the category $\Aff_{\Zar}(Z)$ with any subcategory $\cU$ generating the Zariski topology on $Z$. This means that the objects of $\cU$ must cover $Z$, and that for any $U, V \in \cU$, the  scheme $U\cap V$ can be covered by objects of $\cU$.
\end{remark}

\subsection{Henselian thickenings}\label{henthick}

\begin{definition}
Given a morphism $ f:X \to Y$ of affine schemes, define $(X/Y)^{\hen}$ to be the Henselisation of $Y$ at $X$. It follows from  \cite{anel} Proposition 64 that this is given by
$$
(X/Y)^{\hen}= \Spec \lim_{\substack{ \lra \\ X \to U \xra{e} Y \\ e \text{ \'etale}}} \Gamma(U, \O_Y)= \Spec \Gamma(X, f^{-1}_{\Pr}\O_{Y, \et}),
$$ 
where $f_{\Pr}^{-1}$ is the presheaf inverse image functor.
\end{definition}

\begin{remark}
Note that if $f$ is also a closed immersion, then $(X,Y)$ is a Henselian pair (as in  \cite{lafon}  or \cite{raynaudhen} \S 11).
\end{remark}

\begin{lemma}\label{henprop}
If $X_0 \to X$ is a closed immersion of affine schemes, and $X^h:= (X_0/X)^{\hen}$, then $X_0\by_XX^h = X_0$.
\end{lemma}
\begin{proof}
Given any affine scheme $Y$ under $X_0$, write $Y^h$ for the localisation $(X_0/Y)^{\hen}$. Properties of henselisation from \cite{anel} give 
$$
(X_0\by_XX^h)^h = X_0^h\by_{X^h}X^h=  X_0,  
$$
so the morphism $X_0 \to X_0\by_XX^h$ is pro-\'etale. It is automatically a closed immersion, so we must have $X_0\by_XX^l= X_0 \sqcup T$, for some $T$. However, by \cite{grecononlocal} Theorem 4.1, components of 
$X^h\by_XX_0 \to X^h$ correspond to components of $X_0\by_XX_0 \to X_0$. The latter is an isomorphism, so $X_0 \to  X_0\by_XX^h$ must be an isomorphism on the set of components, hence an isomorphism.
\end{proof}

\begin{definition}
Given a cosimplicial affine scheme $X$, define $X^{h}\in c\Aff$ by $(X^h)^n:= (\pi^0X/X^n)^{\hen}$.
\end{definition}

\begin{proposition}\label{hthm}
For $X \in c\Aff$, the morphism $f:X^h \to X$ is a weak equivalence.
\end{proposition}
\begin{proof}
The proof of Proposition \ref{lthm} carries over.
\end{proof}

\begin{corollary}
For any simplicial cosimplicial affine scheme  $X$, there  is a (levelwise)  weak equivalence $X^h \to X$ in $sc\Aff$, where $X^h$ is defined by setting $(X^h)_i:= (X_i)^h$. In particular, this applies when $X$ is 
any homotopy derived Artin $n$-hypergroupoid (in which case $X^h$ is also).
\end{corollary}

\begin{remarks}\label{henmodel}
The maps $\pi^0(X) \to X^h$ are levelwise closed immersions, so $\pi^0(X)\to X_i^h$ is a Henselian pair. 
By \cite{gruson} Theorem I.8, Henselian pairs are lifted by all smooth morphisms of affine schemes. This means that many techniques involving infinitesimal extensions can be adapted for  Henselian pairs. 
Henselian pairs are also the largest class of maps with this lifting property, since lifting with respect to all standard \'etale morphisms forces a map to be Henselian, while lifting with respect to $\bA^1$ forces a map to be a closed immersion. 
There is, in fact, a simplicial  model structure on $s\Ring^h$, for which a morphism is a weak equivalence or fibration whenever the underlying map in $s\Ring$ is so. For this model structure, smooth rings are cofibrant.
\end{remarks}

Recall from Remark \ref{cflurie} that all algebraic spaces are $2$-geometric, and that a quasi-separated algebraic space is $1$-geometric if and only if it is semi-separated.

\begin{theorem}\label{hshfthm}
Fix an algebraic space $Z$ over a ring $R$, and fix $m \ge 2$ (or $m\ge 1$ if $Z$ is semi-separated). Then the  homotopy category of $m$-geometric derived Deligne--Mumford stacks $X$ over $R$  with $\pi^0X \simeq Z$ is weakly equivalent to the   homotopy category  of  presheaves $\sA_{\bt}$ of simplicial $R$-algebras on  the category  of affine schemes \'etale over $Z$,  satisfying the following:
\begin{enumerate}
\item $\pi_0(\sA_{\bt})= \O_{Z}$;
\item for all $i$, the presheaf $\pi_i(\sA_{\bt})$ is a quasi-coherent $\O_{Z}$-module.
\end{enumerate}
\end{theorem}
\begin{proof}
The proof of Theorem \ref{lshfthm} carries over to this context, replacing Proposition \ref{lthm} with Proposition \ref{hthm}.
\end{proof}

\begin{remark}\label{hshfrk}
As in Remark \ref{lshfrk}, we can replace the category $\Aff_{\et}(Z)$ with any subcategory $\cU$ generating the \'etale topology on $Z$. 
\end{remark}

\begin{remark}\label{postnikov}%%Compare Lurie? His Thm 3.4.13 only deals with affine ($0$-geometric) case
Take a homotopy derived Artin $n$-hypergroupoid $Y$ and let $\fY=Y^{\sharp} $; since $Y^h \to Y$ is a weak equivalence and $\pi^0Y \to Y^h$ is a Henselian pair on each level,  the \'etale sites of $\fY$ and $\pi^0\fY$ are weakly equivalent.

[Another way to look at this phenomenon is that taking the Postnikov tower of $O(Y)$ (as in \cite{sht} Definition VI.3.4) allows us to describe $Y$ as the colimit of a sequence $\pi^0Y \into \pi^{\le 1}Y \into \pi^{\le 2} Y \into \ldots$ of homotopy derived Artin $n$-hypergroupoids, with the $i$th closed immersion defined by an ideal  sheaf weakly equivalent to $\pi_iO(Y)[-i]$. Thus
$\fY$ can be expressed as the colimit of the sequence $\pi^0\fY \to \pi^{\le 1}\fY \to \pi^{\le 2} \fY \to \ldots$ of homotopy square-zero extensions, so the \'etale sites of $\fY$ and $\pi^0\fY$ are isomorphic.] 

Thus the corresponding $\infty$-topoi are equivalent, so giving the \'etale sheaf $\O_{\fY}$ on $\fY$ is equivalent to giving the  \'etale sheaf $\iota^{-1}\O_{\fY}$ on $\pi^0\fY$. The same is not true for smooth morphisms, where the $\infty$-topoi differ (although the homotopy categories are the same), which  suggests that we cannot expect an exact analogue for derived Artin stacks.  However, Corollary \ref{thickenart} will provide a strong substitute.
\end{remark}

\subsection{DG stacks and dg-schemes}\label{dgstacks}
Using the Eilenberg-Zilber shuffle product,  the normalisation functor $N^s$ from Definition \ref{N^s} extends to a functor
$$
N^s:s\Ring \to dg_+\Ring
$$ 
from simplicial rings to (graded-commutative) algebras in  non-negatively graded chain complexes. Where no confusion is likely, we will denote this functor by $N$, and its left adjoint by $N^*$. There is a model structure on $dg_+\Alg_{\Q}:= \Q\da dg_+\Ring$, defined by setting weak equivalences to be quasi-isomorphisms, and fibrations to be maps which are surjective in strictly positive degrees.

Now, the functor
$$
N:s\Alg_{\Q} \to dg_+\Alg_{\Q}
$$ 
is a right Quillen equivalence by \cite{QRat}.  In particular, this implies that it gives a weak equivalence of the associated $\infty$-categories, and hence that 
$$
N: \Ho(s\Alg_{A}) \to \Ho(dg_+\Alg_{A})
$$
is a weak equivalence for any $\Q$-algebra $A$.  

%Since all objects of $dg_+\Alg_A$ are fibrant,  
We may now apply Theorems \ref{relstrict} and \ref{duskinmor}  to the HAG context given by smooth morphisms and $dg_+\Alg_A$. Note that this is not the same as the HAG context of $D$-stacks in \cite{hag2}, since our chain complexes are non-negatively graded. However, since $N$ is a Quillen equivalence, this HAG context is equivalent to that of $D^-$-stacks. Explicitly:

\begin{definition}
A morphism $f:A \to B$ in $dg_+\Ring$ is said to be smooth if $\H_0(f): \H_0A \to \H_0B$ is smooth, and  the maps $\H_n(A)\ten_{\H_0(A)}\H_0(B) \to \H_n(B)$ are isomorphisms for all $n$.
\end{definition}

\begin{definition}
Let $DG^+\Aff_{\Q}$ be the opposite category to $dg_+\Alg_{\Q}$. Define the denormalisation functor $D: DG^+\Aff_{\Q}\to c\Aff_{\Q}$ to be opposite to $N^*$, so $DX(A)= X(NA)$, for $X \in DG^+\Aff_{\Q}, A \in s\Alg_{\Q}$. This has left adjoint $D^*$, given by $D^*\Spec A= \Spec NA$. 
\end{definition}

\begin{definition}
Let $\pi^0\co  DG^+\Aff_{\Q}\to \Aff_{\Q}$ be the functor $\Spec A \mapsto \Spec \H_0A$, and say that a morphism in $DG^+\Aff_{\Q}$ is \emph{surjective} if its image under $\pi^0$ is so.
       
\end{definition}

 \begin{definition}\label{dgnpreldef}
Taking $\oC$ to be the class of smooth surjections in $DG^+\Aff$, we will (by analogy with Examples \ref{cfddt}) refer to  (trivial) $(n, \oC)$-hypergroupoids as (trivial) DG Artin $n$-hypergroupoids. 
\end{definition}

The following is an immediate consequence of the Quillen equivalence between simplicial and dg algebras:
\begin{proposition}
If $f:X \to S$ is   a   
  DG Artin $n$-hypergroupoid, then $Df: DX \to DS$ is a %strongly quasi-compact  
  derived Artin $n$-hypergroupoid, which is trivial whenever $f$ is so. 
If $S= D^*Z$ for some $Z \in sc\Aff$, then the map
$$
D^*DX\by_{D^*DS}S \to X 
$$
is a weak equivalence.

If $g:Y \to T$ is a homotopy
 derived Artin $n$-hypergroupoid, then $D^*g: D^*Y \to D^*T$ is a homotopy DG Artin $n$-hypergroupoid, which is trivial whenever $g$ is so. 
If $\widehat{D^*Y}$ is a Reedy fibrant approximation to $D^*Y$ over $D^*T$, then the map 
$$
Y \to D\widehat{D^*Y}\by_{DD^*T}T 
$$
is a weak equivalence.
 \end{proposition}

Note that the functor $D$ only behaves well when applied to Reedy fibrations, so the proposition does not apply if $f$ is only a  homotopy     DG Artin $n$-hypergroupoid. However, we will now define a homotopy inverse to $N$, similar to the Thom--Sullivan functor, which does preserve weak equivalences.

\begin{definition}
Let $\Omega_n$ be the  cochain algebra 
$$
\Q[t_0, t_1, \ldots, t_n,dt_0, dt_1, \ldots, dt_n ]/(\sum t_i -1, \sum dt_i)
$$  of rational differential forms on the $n$-simplex $\Delta^n$.
 These fit together to form a simplicial diagram $\Omega_{\bt}$ of DG-algebras.
\end{definition}

\begin{definition}
Given a simplicial module $S_{\bt}$ and a cosimplicial module $C^{\bt}$, define $S\ten_{\la} C$ by
$$
S\ten_{\la} C:=\{ x \in \prod_n S_n\ten C^n \,:\, (1 \ten \pd^i)x_n= (\pd_i\ten 1)x_{n+1},\, (1 \ten \sigma^i)x= (\sigma_i\ten 1)x_{n-1}\}.
$$

Similarly, given a chain complex $S_{\bt}$ and a cochain complex  $C^{\bt}$, define $S\ten_{\la} C$ by
$$
S\ten_{\la} C:=\{ x \in \prod_n S_n\ten C^n \,:\, (1 \ten d)x_n= (d\ten 1)x_{n+1}\}.
$$
\end{definition}

\begin{definition}
Define the  functor
$
T: dg_+\Alg_{\Q} \to s\Alg_{\Q}
$
by 
$
T(B)_n:= \Omega_n \ten_{\la}B$. 

Define  $T: DG^+\Aff_{\Q}\to c\Aff_{\Q}$ by $T(\Spec A): = \Spec T A$.
\end{definition}

\begin{proposition}\label{Dequiv}
For a $\Q$-algebra $A$, the functor $T$ is quasi-inverse to normalisation $N:\Ho(s\Alg_A) \to \Ho(dg_+\Alg_A)$.  
\end{proposition}
\begin{proof}
Let $\Q^{\Delta}$ be the simplicial cosimplicial $\Q$-algebra given by $(\Q^{\Delta})^m_n= \Q^{\Delta^n_m}$. As in \cite{HinSch}, there are cochain quasi-isomorphisms $\int: \Omega_m \to N_c (\Q^{\Delta})_m$ for all $m$, where $N_c$ denotes cosimplicial normalisation. Explicitly, 
$$
(\int \omega)(f)= \int_{|\Delta^n|}f^*\omega,
$$ 
for $\omega \in \Omega^n_m, f \in \Delta^m_n$. Denormalisation gives a quasi-isomorphism $\int:N_c^{-1}\Omega_m \to (\Q^{\Delta})_m $ of cosimplicial complexes,   and analysis of the shuffle product shows that this is  a quasi-isomorphism of cosimplicial algebras.

Now, for $A \in s\Alg_A$,
$$
T(N^sA) = (N^sA) \ten_{\la, \Q} \Omega \cong A \ten_{\la, \Q} (N_c\Omega),
$$
and $\int$ gives a  quasi-isomorphism
$$
T N^sA\to A \ten_{\la, \Q}(\Q^{\Delta}) =A
$$
of simplicial algebras, so $T$ is a quasi-inverse to $N^s$ on the homotopy categories. 
\end{proof}

\begin{corollary}\label{Dequivcor}
If $f:X \to S$ is   a    
homotopy    DG Artin $n$-hypergroupoid, then $T f: T X \to T S$ is a
homotopy  derived Artin $n$-hypergroupoid, which is trivial whenever $f$ is so. 
If $S= D^*Z$ for some $Z \in sc\Aff$, then the map
$$
D^*T X\by_{D^*T S}S \to X 
$$
is a weak equivalence.
\end{corollary}

We now introduce an analogue of the results of \S\S \ref{fthick}--\ref{henthick}.
\begin{definition}
Given $A \in dg_+\Alg_{\Q}$, define localisation and henselisation of $A$ by 
$$
A^l:= (A_0/\H_0A)^{\loc}\ten_{A_0}A, \quad A^h:= (A_0/\H_0A)^{\hen}\ten_{A_0}A.
$$

Define completion by letting $I:= \ker(A_0 \to \H_0A)$ and setting
$$
\hat{A}:= \Lim_n A/I^nA.
$$
\end{definition}

\begin{lemma}\label{dgshrink}
The maps $A \to A^l$ and $A\to A^h$ are weak equivalences. If $A_0$ is Noetherian and each $A_n$ is a finite $A_0$-module, then $A \to \hat{A}$ is also a weak equivalence.
\end{lemma}
\begin{proof}
We begin by noting that the maps $A_0 \to A_0^l$ and $A_0 \to A_0^h$ are flat, so 
$$
\H_*(A^l)\cong \H_*(A)\ten_{A_0}A_0^l, \quad \H_*(A^h)\cong \H_*(A)\ten_{A_0}A_0^h.
$$
Now, Lemmas \ref{locprop} and \ref{henprop} imply that $\H_0(A)\ten_{A_0}A_0^l \cong \H_0(A)$ and $\H_0(A)\ten_{A_0}A_0^h \cong \H_0(A)$. Since $\H_*(A)$ is automatically an $\H_0(A)$-module, this completes the proof in the local and Henselian cases.

If $A_0$ is Noetherian, then \cite{Mat} Theorem 8.8 implies that $A_0 \to \hat{A}_0$ is flat. If $A_n$ is a finite $A_0$-module, then \cite{Mat} Theorem 8.7 implies that $\hat{A}_n= \hat{A}_0 \ten_{A_0}A_n$. Thus 
$$
\H_*(\hat{A})\cong \H_*(A)\ten_{A_0}\hat{A}_0,
$$
and applying  \cite{Mat} Theorem 8.7 to the $A_0$-module $\H_0A$ gives that $\H_0(A)\ten_{A_0}\hat{A}_0 \cong \H_0(A)$, so 
$\H_*(\hat{A})\cong \H_*(A)$, as required. 
\end{proof}

\begin{corollary}
For any simplicial diagram  $X$ in $DG^+\Aff$, there   are (levelwise)  weak equivalences $X^l \to X$ and $X^h \to X$ in $sDG^+\Aff$, where $X^l$ is defined by setting $(X^l)_i:= (X_i)^l$ and similarly for $X^h$. In particular, this applies when $X$ is 
any homotopy DG Artin $n$-hypergroupoid (in which case $X^l$ and $X^h$ are also). If $X^0$ is levelwise Noetherian, with each $\O(X^n)$ levelwise coherent on $X^0$, the same is true of the levelwise completion $\hat{X} \to X$.  
\end{corollary}

We are now in a position to describe the $D^-$-stack associated to a dg-scheme. 
\begin{definition}
Recall from \cite{Quot} Definition 2.2.1 that a dg-scheme $X=(X^0, \O_{X,\bt})$ consists of a scheme $X^0$, together with a sheaf $\O_{X,\bt}$ of quasi-coherent  non-negatively graded (chain) dg-algebras over $\O_{X^0}$, with $\O_{X,0}=\O_{X^0}$. 
\end{definition}

\begin{definition}
The degree $0$ truncation $\pi^0X$ of a dg-scheme $X$ is defined to be the closed subscheme
$$
\pi^0X := \oSpec \sH_0(\O_{X, \bt}) \subset X^0. 
$$
This is denoted by $\pi_0X$ in \cite{Quot}.

A morphism $f:X \to Y$ of dg-schemes is said to be a quasi-isomorphism if it induces an isomorphism  $ (\pi^0X, \sH_*(\O_{X, \bt})) \to (\pi^0Y,\sH_*(\O_{Y, \bt}))$.  
\end{definition}

Note that Lemma \ref{dgshrink} implies that for a dg-scheme $X=(X^0, \O_X)$, the maps 
$$
((\pi^0X/X)^{\hen}, a^*b^*\O_X) \xra{a} ((\pi^0X/X)^{\loc}, b^*\O_X) \xra{b} X
$$ 
are quasi-isomorphisms. If $X^0$ is locally Noetherian and each $\O_{X,n}$ coherent, then the map
$$
(\widehat{(\pi^0X/X)}, c^* a^*b^* \O_X) \xra{c} ((\pi^0X/X)^{\hen}, a^*b^*\O_X)
$$
 is also a quasi-isomorphism, where $\widehat{(\pi^0X/X)}$ is the formal completion of $X$ along $\pi^0X$.

Given a dg-scheme $X=(X^0, \O_X)$, with $X^0$ semi-separated, we may take an affine cover $U=\coprod_i U_i$ of $X^0$, and define the simplicial scheme $Y^0$ by $Y^0:= \cosk_0(U/X^0)$.  Hence
\begin{eqnarray*}
Y_n^0&=& \overbrace{U\by_{X^0}U \by_{X^0} \ldots  \by_{X^0}U}^{n+1},\\
&=& \coprod_{i_0, \ldots, i_n} U_{i_0}\cap \ldots \cap U_{i_n},
\end{eqnarray*}
so $Y_n^0$ is a disjoint union of affine schemes (as $X^0$ is semi-separated, i.e. $X^0 \to X^0 \by X^0$ is affine). Thus $Y^0$ is a Zariski $1$-hypergroupoid and $a:Y^0 \to X^0$ a resolution, and we may define a homotopy DG $(1, \open)$-hypergroupoid $Y=\gpd(X)$ by 
$$
\gpd(X):= \oSpec (a^{-1}\O_X).
$$
\emph{A fortiori}, this is a homotopy DG Artin $1$-hypergroupoid.

\begin{remarks}\label{cechrks}
\begin{enumerate}
\item \label{cechrk}%%NB no retraction from $X$ to $X'$, as order matters.
Note that $Y^0$ is a variant of the unnormalised \v Cech resolution for $X^0$. The standard normalised \v Cech resolution is given by $\coprod_{i_0< \ldots< i_n} U_{i_0}\cap \ldots \cap U_{i_n} $ in level $n$, so the standard unnormalised \v Cech resolution $(X^0)'$ is given by $(X^0)'_n= \coprod_{i_0\le \ldots\le i_n} U_{i_0}\cap \ldots \cap U_{i_n} $. There is a canonical morphism $(X^0)' \to Y$, but the partial matching map
$$
(X^0)'_2 \to M_{\L^2_2}(X^0)'
$$
is not surjective, so $(X^0)'$ is not a relative Zariski $n$-hypergroupoid for any $n$.

\item
 If $X^0$ is not semi-separated, we could instead apply Theorem \ref{relstrict} to obtain a Zariski $2$-hypergroupoid $Y^0$, and then a homotopy DG $(2, \open)$-hypergroupoid $\gpd(X)$ by the formula above.

\item
If $X$ is a dg-manifold in the sense of \cite{Quot}, then $Y^h$ will be Reedy fibrant in the model category $sDG^+\Aff^h$ (defined analogously to Remarks \ref{henmodel}). \cite{Quot} Theorem 2.6.1 shows that if $X^0$ is quasi-projective, then there exists a dg-manifold $X'$ quasi-isomorphic to $X$, which can be thought of as a form of fibrant approximation. In fact, the construction is such that $(X')^l$ is Reedy fibrant in the model category $sDG^+\Aff^l$, since $\O_{X'}$ is locally free for the Zariski topology.
\end{enumerate}
\end{remarks}

\begin{theorem}\label{dgshfthm}
Fix an algebraic space $Z$ over a $\Q$-algebra  $R$, and fix $m \ge 2$ (or $m\ge 1$ if $Z$ is semi-separated). Then the  homotopy category of $m$-geometric derived schemes (resp.   $m$-geometric derived Deligne--Mumford stacks)
$X$ over $R$  with $\pi^0X \simeq Z$ is weakly equivalent to the   homotopy category  of  presheaves
 $\sA_{\bt}$ of non-negatively graded dg $R$-algebras on  the category $\Aff_{\Zar}(Z)$ (resp.  $\Aff_{\et}(Z)$),  satisfying the following:
\begin{enumerate}
\item $\H_0(\sA_{\bt})= \O_{Z}$;
\item for all $i$, the presheaf $\H_i(\sA_{\bt})$ is a quasi-coherent $\O_{Z}$-module.
\end{enumerate}
In particular, the corresponding homotopy categories are equivalent.
\end{theorem}
\begin{proof}
The proofs of  Theorems \ref{lshfthm} and \ref{hshfthm} carry over to this context, replacing $\iota_{\Pr}^{-1}\O_{\tilde{X}}$ with the presheaves
$$
 (O({\tilde{X}}^0)/\O_{\tilde{Z}})^{\loc}\ten_{O({\tilde{X}}^0)}O({\tilde{X}})_{\bt} \,\text{ or }\, (O({\tilde{X}}^0)/\O_{\tilde{Z}})^{\hen}\ten_{O({\tilde{X}}^0)}O({\tilde{X}})_{\bt}.
$$
\end{proof}

\begin{corollary}
If $\fX$ is an $n$-geometric derived Deligne--Mumford stack with $\pi^0\fX$ a quasi-affine scheme, then there exists a dg-scheme $X$ with $\gpd(X)^{\sharp} \simeq \fX$. 
\end{corollary}
\begin{proof}
Let $Y= \Spec \Gamma(\pi^0\fX, \O_{\pi^0\fX})$; since $\pi^0\fX \to Y$ is an open immersion, the complement $Z$ is closed. Take a presheaf $\sA_{\bt}$ as in Theorem \ref{dgshfthm}, and let $W:= \Spec \Gamma(\pi^0\fX, \sA_0)$, noting that $Y \to W$ is a closed immersion. Now set $X^0$ to be the quasi-affine scheme $W- Z$, and define $\O_{X, n}$ to be the quasi-coherent sheaf associated to the module $\Gamma(\pi^0\fX, \sA_n)$ on the affine $W$. 
\end{proof}

\section{Derived Quasi-coherent sheaves and the cotangent complex}\label{dsheaves}

We now return to the derived Artin hypergroupoids of \S \ref{hagd} and Examples \ref{cfddt}, so $c\Aff$ will denote the model category of cosimplicial affine schemes. For a derived Artin $n$-hypergroupoid $X$, we write $X^{\sharp}$ for the associated  $n$-geometric derived Artin stack $|\oR \uline{h} X|$ given by combining Propositions \ref{easy} and \ref{equivAff}.

For a simplicial ring $A_{\bt}$, the simplicial normalisation functor $N^s$ induces an equivalence  between the categories $s\Mod(A)$ of simplicial $A_{\bt}$-modules, and $dg_+\Mod(N^sA)$ of $N^sA$-modules in non-negatively graded chain complexes, where $N^sA$ is given a graded-commutative multiplication by the Eilenberg--Zilber shuffle product. As observed in \cite{hag2} \S 2.2.1, this extends to give a weak equivalence between the  $\infty$-category of stable $A$-modules, and the $\infty$-category $dg\Mod(N^sA)$ of $N^sA$-modules in $\Z$-graded chain complexes, and hence an equivalence between the corresponding homotopy categories.

\begin{definition}
Define left Quillen presheaves $dg\Mod$ and $dg_+\Mod$ on $c\Aff$ by $dg\Mod(X):= dg\Mod(N^sO(X))$ and  $dg_+\Mod(X):= dg_+\Mod(N^sO(X))$. Note that under the Dold--Kan equivalence, $dg_+\Mod$ is equivalent to the  left Quillen presheaf $\Mod$ of Definition \ref{moddef}.

Morphisms are said to be weak equivalences if they induce isomorphisms on homology. Fibrations in $dg\Mod(X)$ are surjections, while fibrations in $dg_+\Mod(X)$ are surjective in non-zero degrees.
\end{definition}

\begin{lemma}\label{dgmodhyp2}
The functors $dg\Mod, dg_+\Mod$ are   left Quillen hypersheaves on $c\Aff$ with respect to flat surjections. 
\end{lemma}
\begin{proof}
The proof of  Lemma \ref{dgmodhyp} carries over, since \cite{hag2} Proposition 2.2.2.5 shows that for any flat morphism $f$ in $c\Aff$, we have $\oL f^* \simeq f^*$. This means that for a faithfully flat hypercover $ p\co \tilde{X}_{\bt}\to X$ in $c\Aff$, and $\sF \in cdg\Mod(\tilde{X}_{\bt})_{\cart}$, we have $\H_i\sF \in c\Mod(\pi^0\tilde{X}_{\bt})_{\cart}$ for all $i$, for $\pi^0$ as in Definition \ref{h0def}.
\end{proof}

Note that Remark \ref{cfhagqcoh} specialises to give that for any derived Artin $n$-hypergroupoid $X$, $cdg_+\Mod(X)_{\cart}$ is equivalent to the relative category of quasi-coherent modules on the derived $n$-geometric stack $X^{\sharp}$. Quasi-coherent complexes on $X^{\sharp}$ in the sense of \cite{lurie} \S 5.2 correspond to $cdg\Mod(X)_{\cart}$.

\subsection{Cotangent complexes}\label{cotsn} 

Fix a   derived Artin $m$-hypergroupoid $X\to S$. 

\begin{definition}\label{ulinedef}
 We make $cdg\Mod(X)$  into a simplicial category by setting (for $K \in \bS$)
$$
(M^K)^n: = (M^n)^{K_n}, 
$$
as an $N^sO(X)^n$-module in chain complexes. This has a left adjoint $M \mapsto M\ten K$. Given a cofibration $K \into L$, we write $M \ten (L/K):= (M\ten L)/(M\ten K)$.

Given $M \in cdg\Mod(X)$, define $\uline{M} \in (cdg\Mod(X))^{\Delta}$ to be the cosimplicial complex given in cosimplicial level $n$ by $M\ten \Delta^n$.
\end{definition}

\begin{lemma}\label{omegacalc}
If $N^s\Omega(X/S) \in cdg\Mod(X)$ denotes the chain complex given  on $X_n$ by $N^s\Omega_{X_n/S_n}$, then 
$$
N^s\Omega(X/S) \ten K= \eta^*N^s\Omega(X^K/S^K),
$$
for $\eta: X \to X^K$ corresponding to the constant map $K\to \bt$
\end{lemma}
\begin{proof}
This follows immediately from the definitions of $X^K$ and $M \ten K$.
\end{proof}

\begin{definition}
Given $M \in cdg_+\Mod(X)$, define the derived deformation space by
$$
\DDer(X/S, M):=  \HHom_{(X \da sc\Aff\da S)}(\oSpec (\O_X \oplus N_s^{-1}M), X) \in \bS,
$$ 
for the simplicial structure of Definition \ref{sstr}.
 Since $\oSpec (\O_X \oplus N_s^{-1}M)$ is an abelian cogroup object in $(X \da sc\Aff \da S)$, $\DDer(X/S, M)$ has the natural structure of a simplicial abelian group. 
\end{definition}

The following is immediate:
\begin{lemma}\label{derchar}
For $M\in cdg_+\Mod(X)$, the simplicial group $\DDer(X/S,M)$ is given in level $n$ by $\Hom_{cdg\Mod(X)}( \underline{\Omega(X/S)}^n, M)$.
\end{lemma}

% \begin{definition}\label{cotdef}
% Define the cotangent complex $\bL^{X/S} \in cdg\Mod(X)$ by  $\bL^{X/S}:= \Tot N_c\underline{N^s\Omega(X/S)}$, where $N_c$ denotes cosimplicial conormalisation (Definition \ref{N_c}), and $\Tot$ is the direct sum total functor from cochain chain complexes to chain complexes.
% Note that Lemma \ref{omegacalc} implies that $N_c^i\underline{N^s\Omega(X/S)} \cong \eta^* N^s\Omega(X^{\Delta^i}/X^{\L^i_0}\by_{S^{\L^i_0}}S^{\Delta^i})$.
% \end{definition}

% \begin{definition}\label{cotdef}
% Define the cotangent complex $\bL^{X/S} \in cdg\Mod(X)$ by  $\bL^{X/S}:= \Tot^{\Pi} N_c\underline{N^s\Omega(X/S)}$, where $N_c$ denotes cosimplicial conormalisation (Definition \ref{N_c}), and $\Tot^{\Pi}$ is the product total functor from cochain chain complexes to chain complexes.
% Note that Lemma \ref{omegacalc} implies that $N_c^i\underline{N^s\Omega(X/S)} \cong \eta^* N^s\Omega(X^{\Delta^i}/X^{\L^i_0}\by_{S^{\L^i_0}}S^{\Delta^i})$.
% \end{definition}

\begin{definition}\label{cotdef}
Define the cotangent complex $\bL^{X/S} \in \pro(cdg\Mod(X))$ by  $\bL^{X/S}:= \{\Tot N_c^{\le p} \underline{N^s\Omega(X/S)}\}_p$, where $N_c$ denotes cosimplicial conormalisation (Definition \ref{N_c}), and $\Tot$ is the direct sum functor from cochain chain complexes to chain complexes.
\end{definition}
Note that Lemma \ref{omegacalc} implies that $N_c^i\underline{N^s\Omega(X/S)} \cong \eta^* N^s\Omega(X^{\Delta^i}/X^{\L^i_0}\by_{S^{\L^i_0}}S^{\Delta^i})$. It will follow from Corollary \ref{loopcot} that the pro-object 
$\bL^{X/S}$ is in fact quasi-isomorphic to an object of $cdg\Mod(X)$.

\begin{lemma}\label{h0cot}
If we write $\iota: \pi^0X \to X$, then $\iota^*\bL^{X/S}$ is equivalent to the complex
$$
\iota^*N^s\Omega(X/S) \to   \Omega(\pi^0X/\pi^0S)\ten (\Delta^1/\L^1_0) \to \ldots \to\Omega(\pi^0X/\pi^0S)\ten (\Delta^m/\L^m_0),
$$
in (chain) degrees $[-m, \infty)$. If moreover $f:X \to S$ is smooth, then $\iota^*\bL^{X/S}$ is equivalent to the complex
$$
\Omega(\pi^0X/\pi^0S) \to   \Omega(\pi^0X/\pi^0S)\ten (\Delta^1/\L^1_0) \to \ldots \to\Omega(\pi^0X/\pi^0S)\ten (\Delta^m/\L^m_0),
$$
which is locally projective and concentrated in degrees $[-m,0]$.
\end{lemma}
\begin{proof}
The definition of cosimplicial normalisation shows that
$$
N_c^i\underline{N^s\Omega(X/S)}\cong \eta^* N^s\Omega(X^{\Delta^i}/X^{\L^i_0}\by_{S^{\L^i_0}}S^{\Delta^i}),
$$
and Lemma \ref{ppowers} implies that for $i \ge m$, the maps $\theta_i(f): X^{\Delta^i} \to X^{\L^i_0}\by_{S^{\L^i_0}}S^{\Delta^i}$ are trivial relative $0$-hypergroupoids, hence levelwise weak equivalences, so $N_c^i\underline{N^s\Omega(X/S)}$ is  contractible  for $i>m$. For all $i>0$, the maps $\theta_i(f)$ are smooth  relative $m$-hypergroupoids, hence levelwise smooth, so 
$$
\iota^*N_c^i\underline{N^s\Omega(X/S)}\cong N_c^i\underline{N^s\Omega(\pi^0X/\pi^0S)}=N_c^i\underline{\Omega(\pi^0X/\pi^0S)},
$$
using the fact that $\Omega(\pi^0X/\pi^0S)$ is a module rather than a complex.

Finally, if $X \to S$ is smooth, then it is levelwise smooth, so $\iota^*N^s\Omega(X/S) \simeq \Omega(\pi^0X/\pi^0S)$. This is locally projective in the sense that each $\Omega(\pi^0X_n/\pi^0S_n)$ is projective on $X_n$. Note that in general it is not, however, projective in $cdg\Mod(X)$, since it is not Reedy cofibrant.
\end{proof}

\begin{remark}
To understand the distinction between locally projective and projective objects $P$ in $\Mod(X)$, note that projectivity implies vanishing of the higher $\Ext$ groups $\Ext^i_X(P, \sF)$, while local projectivity merely implies vanishing of the higher $\ext$-presheaves $\ext^i_{\O_X}(P, \sF)$.
\end{remark}

\begin{definition}
Let $\HOM$ be $\Z$-graded derived $\Hom$ for complexes. Explicitly, for levelwise projective $U$,  $\HOM(U,V)_n$ is the space of degree $n$ graded homomorphisms from $U$ to $V$ (ignoring the differentials), and the chain complex structure on $\HOM$ is given by $df= d_V\circ f \pm f\circ d_U$. Note that this is related to simplicial enrichments by  $N^s\HHom(U,V)= \tau_{\ge 0} \HOM(U,V)$, where $N^s$ is simplicial normalisation and $\tau$ is good truncation.   
\end{definition}

\begin{proposition}\label{cotgood0}
For all $M \in cdg_+\Mod(X)$, there is a weak equivalence
$$
\HHom_{\pro(cdg\Mod(X))}(\bL^{X/S}, M) \simeq \DDer(X/S,M).
$$
\end{proposition}
\begin{proof}
 Given a cochain chain complex $L$ bounded below in cochain degrees, observe that $\HOM(\{\Tot L^{\le p}\}_p, M) \cong \Tot \HOM(L,M)$, noting that $\HOM(L,M)$ is a bichain complex. Thus
$$
N^s\HHom(\bL^{X/S}, M)\cong \tau_{\ge 0}\Tot\HOM(N_c^{\bt}\underline{\Omega(X/S)} ,M).
$$

Now, $N_c^n\underline{\Omega(X/S)} = \eta_n^*N^s\Omega(X^{\Delta^n}/X^{\L^n_0}\by_{S^{\L^n_0}}S^{\Delta^n})$. Since $X^{\Delta^n}\to X^{\L^n_0}\by_{S^{\L^n_0}}S^{\Delta^n}$ is a trivial derived Artin $m$-hypergroupoid, we know that $N_c^n\underline{\Omega(X/S)}$ is projective in the category  $cdg_+\Mod(X)$ (since the smoothness conditions imply that the latching maps are monomorphisms with projective cokernel).

For $M\in cdg_+\Mod(X)$, the map $M[-i] \oplus M[1-i] \xra{\id, d} M[-i]$ is a surjection in $cdg_+\Mod(X)$ for $i>0$; this implies that 
$$
\HOM_{1-i}(N_c^n\underline{\Omega(X/S)}, M) \xra{d} \z_{-i}\HOM(N_c^n\underline{\Omega(X/S)}, M)
$$
is surjective, so $\tau_{\ge 0}\HOM(N_c^n\underline{\Omega(X/S)} ,M) \simeq \HOM(N_c^n\underline{\Omega(X/S)} ,M)$ for all $n>0$. We have therefore shown that
$$
N^s\HHom(\bL^{X/S}, M)\simeq \Tot\tau_{\ge 0}\HOM(N_c^{\bt}\underline{\Omega(X/S)} ,M).
$$

Now, the Reedy fibration conditions on $X \to S$ imply that for all $n$, the map 
$$
N_c^n\underline{\Omega(X/S)}/dN_c^{n-1}\underline{\Omega(X/S)} \to N_c^{n+1}\underline{\Omega(X/S)}
$$
is a monomorphism, with projective cokernel in the category $g\Mod(X)$ of  $N^sO(X)$-modules in  graded abelian groups (ignoring the simplicial differential $d^s$).
This implies that the sequence 
$$
\ldots \xra{d_{\Omega}}\HOM(N_c^{i+1}\underline{\Omega(X/S)} ,M)\xra{d_{\Omega}}\HOM(N_c^{i}\underline{\Omega(X/S)} ,M)\xra{d_{\Omega}}\ldots
$$
is exact. Since $\tau_{\ge 0}\HOM_n= \HOM_n$ for $n>0$, and $(\tau_{\ge 0} \HOM)_0 = \Hom$, this implies that 
$$
\Tot\tau_{\ge 0}\HOM(N_c^{\bt}\underline{\Omega(X/S)} ,M)\simeq \Hom(N_c^{\bt}\underline{\Omega(X/S)} ,M) \simeq N^s\DDer(X/S,M),
$$
by Lemma \ref{derchar}.
\end{proof}

\begin{corollary}\label{trivcot}
If $f:Y\to S$ is a trivial derived Artin $m$-hypergroupoid, then $\bL^{Y/S} \simeq 0$.
\end{corollary}
\begin{proof}
Since $f$ takes small extensions to trivial fibrations, $\DDer(Y/S,M)$ is trivially fibrant, so $\DDer(Y/S,M)\simeq 0$. Therefore $\bL^{Y/S} \simeq 0$.
\end{proof}

\begin{corollary}\label{cfhagcot}
$$
\DDer(X/S, M) \simeq \oR\HHom_{(X^{\sharp} \da (c\Aff)^{\sim, \tau} \da S^{\sharp})}(\oSpec (\O_{X} \oplus (N^s)^{-1}M)^{\sharp}, X^{\sharp}).
$$
\end{corollary}
\begin{proof}
Setting $Y:=\oSpec (\O_{X} \oplus (N^s)^{-1}M)$,   Theorem \ref{duskinmor} implies that the space
 $\oR\HHom_{(X^{\sharp} \da (c\Aff)^{\sim, \tau} \da S^{\sharp})}(Y^{\sharp}, X^{\sharp}) $ is  a colimit of spaces
\[
 \HHom_{(X'\da sc\Aff \da S)}(Y',X),       
\]
for a suitable system of trivial derived Artin $m$-hypergroupoids $Y'\to Y$, with $X'= X\by_YY'$.

By Proposition \ref{cotgood0},  this space is given by
$
 \HHom_{cdg\Mod(X)}(\bL^{X/S}, \pi_*\pi^*M),       
$
for $\pi\co X' \to X$, which  by adjunction is just $ \HHom_{cdg\Mod(X')}(\pi^*\bL^{X/S}, \pi^*M)$. Since $dg\Mod$ is a left Quillen hypersheaf by Lemma \ref{dgmodhyp2}, the functor $\pi^* \co cdg\Mod(X)\to cdg\Mod(X')$ gives an equivalence of relative categories, so 
\[
  \HHom_{cdg\Mod(X)}(\bL^{X/S}, \pi_*\pi^*M) \simeq \HHom_{cdg\Mod(X)}(\bL^{X/S}, M),      
\]
which completes the proof.
\end{proof}

\begin{corollary}\label{cotgood}
 $\bL^{X/S}$ gives rise to  a cotangent complex in the sense of \cite{hag2} Definition 1.4.1.7 (or equivalently \cite{lurie} \S 3.2)  by associating to any morphism $f:U \to X^{\sharp}$ with $U \in c\Aff$, the complex $f^* \bL^{X/S} \in \pro(cdg\Mod(U)_{\cart})$.  
In other words,
$$
\HHom_{\pro(dg\Mod(U))}(f^*\bL^{X/S}, N) \simeq \HHom_{(U \da  (c\Aff)^{\sim, \tau} \da S^{\sharp})}(\oSpec (\O_{U} \oplus N_s^{-1}N)^{\sharp}, X^{\sharp}),
$$
functorially in $N \in dg_+\Mod(U)$ and in $f$. 
 \end{corollary}
\begin{proof}
 Passing to a suitable \'etale cover, we may assume that $f$ factors as $f:U \to X$. If we set $M:= f_*N \in cdg_+\Mod(X)$ (which need not be Cartesian and is not to be confused with $f_*^{\cart} N$), then 
$$
\HHom_{(U \da (c\Aff)^{\sim, \tau}\da S^{\sharp})}(\oSpec (\O_{U} \oplus N_s^{-1}N)^{\sharp}, X^{\sharp}) \simeq \DDer(X/S, f_*N)
$$
by Corollary \ref{cfhagcot}, so the result follows from Proposition \ref{cotgood0}.
\end{proof}

\begin{remark}
Beware that the cotangent complex of \cite{lurie} is not the same as the better-known cotangent complex of \cite{lurieDAG4}. The difference is the that the former is based on simplicial rings, while the latter uses symmetric spectra. Thus they correspond locally to Andr\'e--Quillen and topological Andr\'e--Quillen homology, respectively. Roughly speaking, simplicial rings serve to apply homotopy theory to algebraic geometry, while symmetric spectra are used to do the opposite. As an example of the differences, some higher topological   Andr\'e--Quillen homology groups of $\Z[t]$ over $\Z$ are non-zero  --- in other words, $\Z \to \Z[t]$ is not formally smooth as  a morphism of symmetric spectra. 
\end{remark}

\begin{definition}
Given a morphism $\fX \to \fS$ of $m$-geometric derived Artin stacks, use Theorem \ref{relstrict} to form a  derived Artin $m$-hypergroupoid $X \to S$ to a derived Artin $m$-hypergroupoid, with $X^{\sharp}\simeq\fX, S^{\sharp}\simeq \fS$, and define the quasi-coherent complex $\bL^{\fX/\fS}$ on $\fX$ by $a^*\bL^{\fX/\fS} \simeq \bL^{X/S}$, for $a:X \to \fX$, using the equivalence of Corollary \ref{qcohequiv}. The characterisation in Corollary \ref{cotgood} ensures that this is well defined.
\end{definition}

\begin{lemma}\label{Lcart}
$\bL^{X/S}$ is homotopy-Cartesian.
\end{lemma}
\begin{proof}
We need to show that for each map $\pd_i:X_{n+1} \to X_n$, the map
$$
\pd^i: \pd_i^*(\bL^{X/S})^n\to (\bL^{X/S})^{n+1}
$$
is a  pro-quasi-isomorphism of pro-chain complexes of quasi-coherent sheaves on $X_{n+1}$.

It follows from Lemma \ref{omegacalc} that $\bL(X/S)^n=  \{\Tot N_c^{\le p}N^s \eta^*\Omega(X^{\Delta^n}/S^{\Delta^n})\}_p$ 
in $\pro(dg\Mod(X_n))$.  
Thus $\pd^i$ is a monomorphism, and its cokernel is 
$$
C:=  \{\Tot N_c^{\le p}N^s \eta^*\Omega(X^{\Delta^{n+1}}/ X^{\Delta^{n}}\by_{S^{\Delta^{n}}}S^{\Delta^{n+1}})\}_p;
$$
we wish to show that $C$ is contractible. 

Now, observe that $C$ is $\bL(X^{\Delta^{n+1}}/ X^{\Delta^{n}}\by_{S^{\Delta^{n}}}S^{\Delta^{n+1}})^0$ on $X_n= (X^{\Delta^{n+1}})_0$. Since $\pd^i:\Delta^n \to \Delta^{n+1}$ is a trivial cofibration,    $X^{\Delta^{n+1}} \to X^{\Delta^{n}}\by_{S^{\Delta^{n}}}S^{\Delta^{n+1}}$ is a trivial relative $m$-hypergroupoid by Lemma \ref{ppowers}, so 
 Corollary \ref{trivcot} implies that $C \simeq 0$, as required.
\end{proof}

\begin{lemma}\label{0cot}
If $f:X \to S$ is a derived  Artin $0$-hypergroupoid, then $\bL^{X/S} \simeq N^s\Omega(X/S)$.
\end{lemma}
\begin{proof}
It suffices to show that each map $\pd^i_{\Delta}: N^s\Omega(X/S)\ten \Delta^n \to \Omega(X/S)\ten \Delta^{n+1}$ is a weak equivalence in $cdg_+\Mod(X)$. This map is a monomorphism, with cokernel
$$
\eta^*\Omega(X^{\Delta^{n+1}}/ X^{\Delta^n}\by_{S^{\Delta^n}}S^{\Delta^{n+1}}).
$$

By Lemma \ref{ppowers},  $X^{\Delta^{n+1}} \to X^{\Delta^{n}}\by_{S^{\Delta^{n}}}S^{\Delta^{n+1}}$ is a trivial relative Artin $0$-hypergroupoid, which is equivalent to saying that it is a levelwise weak equivalence, so $\Omega(X^{\Delta^{n+1}}/ X^{\Delta^n}\by_{S^{\Delta^n}}S^{\Delta^{n+1}})$ is contractible, as required.
\end{proof}

\begin{corollary}\label{loopcot}
If $X\to S$ is a derived  Artin $m$-hypergroupoid, then 
$$
\bL^{X/S} \simeq N^s\Omega(X/S)\ten (\Delta^m/ \pd\Delta^m)[m].
$$ 
In particular, this implies that $\H_i\bL^{X/S}=0 $ for $i<m$ and that $\bL^{X/S}$ is an essentially constant pro-object.
\end{corollary}
\begin{proof}
We use the characterisation of the cotangent complex from the proof of Corollary \ref{cotgood}, namely that
$$
\HHom_{\pro(cdg\Mod(X))}(\bL^{X/S}, M) \simeq \DDer(X/S,M),
$$
for all $M \in cdg_+\Mod(X)$.

Now, simplicial considerations show that
$$
\DDer(X^{\Delta^{m}}/ X^{\pd\Delta^m}\by_{S^{\pd\Delta^m}}S^{\Delta^{m}} ,\eta_*M) = \ker( \DDer(X/S,M)^{\Delta^m} \to \DDer(X/S,M)^{\pd\Delta^m}),
$$
which is a model for the iterated loop space $\Omega^m\DDer(X/S,M)$, so
$$
\Omega^m\HHom_{\pro(cdg\Mod(X))}(\bL^{X/S}, M) \simeq \HHom_{\pro(cdg\Mod(X))}(\eta^*\bL^{X^{\Delta^{m}}/ X^{\pd\Delta^m}\by_{S^{\pd\Delta^m}}S^{\Delta^{m}} }, M).
$$
 
Since $X^{\Delta^{m}}\to X^{\pd\Delta^m}\by_{S^{\pd\Delta^m}}S^{\Delta^{m}}$ is a derived Artin $0$-hypergroupoid (by Lemma \ref{ppowers}), Lemma \ref{0cot} then implies that
$$
\Omega^m\HHom_{\pro(cdg\Mod(X))}(\bL^{X/S}, M) \simeq \HHom_{cdg\Mod(X)}(N^s\Omega(X/S)\ten (\Delta^m/ \pd\Delta^m), M),
$$
for all $M \in cdg_+\Mod(X)$.

Moreover, 
\begin{eqnarray*}
N^s\Omega^m\HHom_{\pro(cdg\Mod(X))}(\bL^{X/S}, M) &\simeq& \tau_{\ge 0}(N^s\HHom_{\pro(cdg\Mod(X))}(\bL^{X/S}, M)[m])\\
& =& N^s\HHom_{\pro(cdg\Mod(X))}(\bL^{X/S}[-m], M).
\end{eqnarray*}

Thus, for $M \in cdg_+\Mod(X)$, we have
$$
\HHom_{\pro(cdg\Mod(X))}(\bL^{X/S}[-m], M) \simeq \HHom_{cdg\Mod(X)}(N^s\Omega(X/S)\ten (\Delta^m/ \pd\Delta^m), M),
$$
from which we deduce that
$$
\bL^{X/S}[-m] \simeq N^s\Omega(X/S)\ten (\Delta^m/ \pd\Delta^m).
$$
\end{proof}

\begin{definition}\label{cotdefh}
Given a homotopy derived Artin $m$-hypergroupoid $X \to S$, define $\bL^{X/S} \in cdg\Mod(X)_{\cart}$ by taking a Reedy fibrant approximation $u:X \to \hat{X}$ over $S$, and setting $\bL^{X/S}:= u^*N^s\Omega(\hat{X}/S)\ten (\Delta^m/ \pd\Delta^m)[m]$ (so equivalent to  $u^*\bL^{\hat{X}/S}$). In particular, we may apply this when $X \to S$ is an Artin $m$-hypergroupoid (with trivial derived structure), and then exploit the description of $\bL^{X/S}$  in Lemma \ref{h0cot}.
\end{definition}

\subsubsection{Comparison with Olsson}

To compare our definition of  the cotangent complex with Olsson's  (in \cite{olssartin}), we could simply note that  Corollary \ref{cotgood} ensures that his must be the sheafification of ours. However, a more direct comparison is possible. 

If $\fX$ is a quasi-compact Artin $1$-stack, we may take a presentation  $X_0 \to \fX$, for $X_0$ affine. If $\fX$ has affine diagonal, then $\fX$ is $1$-geometric, and $X:=\cosk_0(X_0/\fX)$ will be a simplicial affine scheme.

Given a strongly quasi-compact relative Artin $1$-hypergroupoid $f:X \to Y$, for $Y$ a strongly quasi-compact  Artin $1$-hypergroupoid,  let $\fX:=X^{\sharp}$ and $\fY:=Y^{\sharp}$ be the associated stacks, and $u:X \to \hat{X}$ a Reedy fibrant approximation over $Y$ in the category of simplicial cosimplicial affine schemes. Then \cite{olssartin} 8.2.3 defines the cotangent complex to be the sheafification of the complex 
$$
\bL^{X/Y}_{\fX/\fY}: = \Tot(u^*N^s\Omega(\hat{X}/Y) \to \Omega(X/ Y\by_{\fY}^h\fX))
$$
in $cdg\Mod(X)$.

\begin{proposition}\label{cfolsson}
For $X,Y, \fX,\fY$ as above, $\bL^{X/Y}_{\fX/\fY}$ is the sheafification of the cotangent complex $\bL^{X/Y}$.
\end{proposition}
\begin{proof}
The characterisation in Definition \ref{cotdefh} reduces to
$$
\bL^{X/Y} \simeq \Tot(u^*\Omega(\hat{X}/Y) \to \eta^*\Omega(X^{\Delta^1}/ Y^{\Delta^1}\by_{\pd_0, Y}X)),
$$
so it will suffice  to show that $\Omega(X/ Y\by_{\fY}^h\fX)\cong \eta^*\Omega(X^{\Delta^1}/ Y^{\Delta^1}\by_{\pd_0, Y}X)$.

Now, if we write $\delta: X \to X\by_{\fX}^hX$, then 
$$
\Omega(X/ Y\by_{\fY}^h\fX)= \delta^*\Omega((X\by_{\fX}^hX)/ (X\by_{\fY}^hY)).
$$ 
However, $X^{\Delta^1}\simeq X^{\Delta^1}\by_{\fX^{\Delta^1}}^h\fX$, so (since $X= \cosk_0(X_0/\fX)$)
$$
X^{\Delta^1}\simeq  X^{\pd\Delta^1}\by_{\fX^{\pd\Delta^1}}^h\fX\simeq (X\by X)\by^h_{\fX \by \fX}\fX\simeq X\by_{\fX}^hX,
$$
and similarly for $Y$. This gives the required isomorphism.
\end{proof}

\section{Deformations of derived Artin stacks}\label{defsn}

\subsection{Deforming morphisms of derived stacks}\label{defmor}

\begin{definition}
Given a category $\C$, define the category $c_+\C$ of almost cosimplicial diagrams   in $\C$ to consist of  functors $\Delta_* \to \C$, for $\Delta_*$ as in Definition \ref{delta*}. Thus an almost cosimplicial object $X^*$ consists of objects $X^n \in \C$, with all of the operations $\pd^i, \sigma^i$ of a cosimplicial diagram except $\pd^0$, satisfying the usual relations. 
\end{definition}

\begin{definition}
Taking the dual construction to $\Dec_+^{\op}$, we may define, for any almost cosimplicial object $Y$, an almost cosimplicial object $\Dec^+_{\op}Y$ given by $(\Dec^+_{\op}Y)^n= Y^{n+1}$, with $\pd^i_{\Dec^+_{\op}Y}= \pd^{i+1}_Y$. The operations $\pd^{0}_Y:Y \to \Dec^+_{\op}Y$ define an almost cosimplicial map, which has a retraction $\sigma_Y^{0}$. $\Dec^+_{\op}$ is a comonad on almost cosimplicial complexes, with coproduct  given by $ \pd^{1}_Y: \Dec^+_{\op}Y \to \Dec^+_{\op}\Dec^+_{\op}Y$, the co-unit being $\sigma_Y^{0}$. 
\end{definition}

\begin{proposition}\label{defmor1}
Take a diagram  $Z \xra{g} X \xra{f} S$ of simplicial cosimplicial affine schemes, with $f$ a  Reedy fibration. If $Z \into \tilde{Z}$ is a levelwise closed immersion over $S$ defined by a square-zero ideal $\sI$, then the obstruction to extending $g$  to a  morphism 
$
\tilde{g}: \tilde{Z} \to X
$
over $S$ 
lies in 
$$
\H_{-1}\HOM_{cdg\Mod(X)}(N^s\Omega^{X/S}, g_*N^s\sI).
$$

If the obstruction is zero, then the $\Hom$-space of possible extensions is given by
$$
N^s\HHom_{Z \da sc\Aff \da S}(\tilde{Z}, X)  \simeq \tau_{\ge 0}\HOM_{cdg\Mod(X)}(\bL^{X/S}, N^sg_*\sI).
$$
\end{proposition}
\begin{proof}
Since $f$ is a Reedy fibration, the simplicial matching maps
$$
X_n \to M_nX \by_{M_nS}S_n
$$
are fibrations of cosimplicial affine schemes. For a morphism $V \to W$ of cosimplicial affine schemes to be a fibration implies that the cosimplicial matching maps $V^j \to W^j\by_{M^{j-1}W}M^{j-1}V$ have the right lifting property (RLP) with respect to all closed immersions of affine schemes, so in particular are formally smooth.

Now, $\Delta\by \Delta^{\op}$ is a Reedy category (as in \cite{hovey} \S 5.2),
and we may regard $f$ as a morphism of $\Delta\by \Delta^{\op}$-diagrams in affine schemes. The results above  imply that the  $\Delta\by \Delta^{\op}$ matching maps
$$
X_n^j \to (M_nX^j \by M^{j-1}X_n) \by_{(M_nS^j\by_{M_nM^{j-1}S}M^{j-1}S_n )}S_n^j
$$
have the RLP for all closed immersions.

Likewise, the simplicial latching maps 
$$
L_n\tilde{Z}\cup_{L_nZ}Z_n \to \tilde{Z}_n 
$$
are levelwise closed immersions. However, the cosimplicial latching maps of a levelwise closed immersion  $g:V \to W$ of cosimplicial affine schemes will only be closed immersions if $g$ is a weak equivalence. Thus we instead consider the underlying diagram of  almost cosimplicial affine schemes. For a simplicial ring $A_{\bt}$, the almost cosimplicial latching object $L^i_*\Spec (A_*)$ is given by $\Spec (M_{\L^i_0}A)$, so the almost cosimplicial matching maps of a levelwise closed immersion are closed immersions.

Letting $\bI$ be the Reedy category $\Delta_* \by \Delta^{\op}$, this implies that the latching maps 
$$
L(\bI)_n^j\tilde{Z}\cup_{L(\bI)_n^jZ}Z_n^j\to \tilde{Z}_n^j
$$
are closed immersions, which  allows us inductively to construct a lift
$$
\lambda: \tilde{Z}^*_{\bt} \to X^*_{\bt}
$$
of simplicial almost cosimplicial affine schemes. 

Now, the obstruction to $\lambda$ being a cosimplicial morphism is the map
$$
\pd_0\lambda^{\sharp} - \lambda^{\sharp}\pd_0: O(X)_n \to O(\tilde{Z})_{n-1},
$$
whose image is contained in $\ker( O(\tilde{Z}) \to O(Z))=  \sI$. Thus we have a   map
$$
O(X)_n \to \sI_{n-1};
$$
since $\sI$ is a square-zero ideal, 
this map is an $O(S)$-linear $g^{\sharp}\pd_0$-derivation, so corresponds  to an  $(O(X),g^{\sharp}\pd_{0})$-linear map
$$
\delta: \Omega(X/S)_n \to \sI_{n-1}.
$$
Furthermore, 
$\pd_i\delta= \delta\pd_{i+1}$ for all $i>0$, 
so $\delta$ descends to a map $N^s\delta$ on the normalised complexes, which will be $0$ if and only if $\delta=0$. Moreover, $\delta\sigma_i = \sigma_{i-1}\delta$ for $i \ge 1$, $\delta\sigma_0=0$, and $\delta\pd_0=0$.

Thus for $m \in \Omega(X/S)_n$ and $a \in O(X)_n$,
$$
\delta(m \nabla a)= \delta(\sum \pm \sigma_Im \cdot \sigma_J a) = \sum \pm \delta(\sigma_Im)\cdot g^{\sharp}\pd_0\sigma_J a,
$$
where the sum is over shuffle permutations $(I,J)$ of $[0,n-1]$, $\pm$ is the sign of the permutation, and  $\sigma_I= \sigma_{i_r}\ldots \sigma_{i_2}\sigma_{i_1}$ for $I$ the sequence $i_1 < i_2< \ldots < i_r$.

Now, $\sigma_Im=0$ if $I$ contains $0$, so the only non-zero terms in the sum have $0 \in J$, in which case $\pd_0\sigma_{K \cup \{0\}} a = \sigma_{K-1}a$. We also have $\delta(\sigma_Im) = \sigma_{I-1}m$ in this case, so renumbering $I$ and $K$ gives
$$
\delta(m \nabla a)= \delta(m) \nabla g^{\sharp}a;
$$
 in other words, $N^s\delta$ is an $(N^sO(X), g^{\sharp})$-linear map. 

The conditions on $\pd_i$ ensure that $N^s\delta$ is also a chain map, so we have
$$
N^s \delta \in \z_{-1}\HOM_{N^sO(X)}(N^s\Omega(X/S),N^sg_*\sI).
$$

However, the choice of $\lambda$ was not unique. Another choice  $\lambda'$ gives rise to $\beta:= \lambda^*- (\lambda')^*$. Similar reasoning shows that this is a $g^{\sharp}$-derivation $\beta:\Omega(X/S)\to g_*\sI$, and corresponds to an $N^sO(X)$-linear map on normalised complexes. The operation $\delta'$ is given by $\delta - \pd_0\beta +\beta\pd_0$, so $N^s\delta' = N^s\delta- [d, N^s\beta]$. Thus the obstruction lies in 
$$
\H_{-1}\HOM_{N^sO(X)}(N^s\Omega(X/S),N^sg_*\sI).
$$
 
Finally, $\tilde{Z}\cup_Z\tilde{Z} \cong \oSpec (\O_{\tilde{Z}}\oplus \sI\eps)$, for $\eps^2=0$, so the choice of $\tilde{g}$ gives
$$
\HHom_{Z \da sc\Aff \da S}(\tilde{Z}, X)\cong \HHom_{\tilde{Z} \da sc\Aff \da S}(\tilde{Z}\cup_Z\tilde{Z}, X) \cong \HHom_{cdg\Mod(X)}(N^s\Omega(X/S), N^sg_*\sI),
$$
and Corollary \ref{cfhagcot} completes the proof.
\end{proof}

\begin{lemma}\label{deform2}
If $f:X \to S$  is a  derived Artin $m$-hypergroupoid, and $M \in cdg_+\Mod(X)$, then 
$$
\H_{i}\HOM_{cdg\Mod(X)}(\bL^{X/S}, M)\cong \H_{i}\HOM_{cdg\Mod(X)}(\Omega(X/S), M)
$$
for $i \le 0$.
\end{lemma}
\begin{proof}
As in Lemma \ref{h0cot}, the map $\bL^{X/S} \to N^s\Omega(X/S)$ is surjective, and the kernel is the total complex of 
$$
0 \to N^s\Omega(X^{\Delta^1}/X^{\L^1_0}\by_{S^{\L^1_0}}S^{\Delta^1})\to \ldots \to  N^s\Omega(X^{\Delta^i}/X^{\L^i_0}\by_{S^{\L^i_0}}S^{\Delta^i}) \to \ldots .
$$
Since $X^{\Delta^i}\to X^{\L^i_0}\by_{S^{\L^i_0}}S^{\Delta^i}$ is a trivial relative derived Artin $m$-hypergroupoid, its simplicial matching maps are smooth. By \cite{hag2} Definition 1.2.7.1, this means that the cotangent complexes associated to the matching maps are projective, so the chain complex
$
N^s\Omega(X/S) \ten (\Delta^i/\L^i_0)
$
is a projective object of $cdg_+\Mod(X)$ (but not of $cdg\Mod(X)$). Hence
$$
\H_j\HOM_{cdg\Mod(X)}(N^s\Omega(X/S) \ten (\Delta^i/\L^i_0), M)=0
$$
for all $i>0, j<0$ and $M \in cdg_+\Mod(X)$. 

If we set $\L^0_0:=\emptyset$, then we have a spectral sequence
$$
\H_j\HOM_{cdg\Mod(X)}(N^s\Omega(X/S) \ten (\Delta^i/\L^i_0), M) \abuts \H_{j+i}\HOM_{cdg\Mod(X)}(\bL^{X/S},M).
$$
When $M \in cdg_+\Mod(X)$, this combines with the calculation above to show that
$$
\H_{j}\HOM_{cdg\Mod(X)}(\bL^{X/S},M) \cong \H_j\HOM_{cdg\Mod(X)}(N^s\Omega(X/S), M)
$$
for $j<0$, as required.
\end{proof}

\begin{theorem}\label{defmorstack}%%have rewritten this getting rid of $\oR g_*$.
Take a diagram  $\fZ \xra{g} \fX \xra{f} \fS$ of derived $m$-geometric Artin stacks. If $\fZ \into \tilde{\fZ}$ is a  closed immersion over $\fS$ defined by a square-zero quasi-coherent complex  $\sI$, then the obstruction to extending $g$  to a  morphism 
$
\tilde{g}: \tilde{\fZ} \to \fX
$
over $\fS$ 
lies in 
$$
\Ext^{1}_{\O_{\fZ}}(\oL g^*\bL^{\fX/\fS}, \sI).
$$

If the obstruction is zero, then the $\Hom$-space of possible extensions is given by
$$
\pi_i\HHom_{Z \da sc\Aff \da S}(\tilde{Z}, X)  \simeq \Ext^{-i}_{\O_{\fZ}}(\oL g^*\bL^{\fX/\fS}, \sI).
$$
\end{theorem}
\begin{proof}
Apply Theorem \ref{relstrict} to obtain a derived Artin $m$-hypergroupoid $S$ with $S^{\sharp} \simeq \fS$, and  relative derived Artin $m$-hypergroupoids   $Z \to X \to S$  with $X^{\sharp} \simeq \fX$ and $Z^{\sharp} \simeq \fZ$. It will follow from the proof of Theorem \ref{deformstack} that square-zero deformations of $Z$ correspond to square-zero deformations of $\fZ$, so there exists an extension $Z \to \tilde{Z}$  of simplicial cosimplicial affine schemes with $\tilde{Z}^{\sharp} \simeq \tilde{\fZ}$. Thus the extension is  defined by the square-zero ideal $a^*\sI$, for $a: \tilde{Z} \to \tilde{\fZ}$.

We are now in the scenario of Proposition \ref{defmor1}, and sheafification defines a functor from lifts $\tilde{Z} \to X$ to lifts $\tilde{\fZ} \to \fX$. Applying Lemma \ref{deform2},  we see that these lifts are governed by
$$
H_{*}\HOM_{cdg\Mod(Z)}(g^*\bL^{X/S}, \sI) \cong  H_{*}\HOM_{cdg\Mod(Z)}(g^*\bL^{X/S},\sI )\cong  \Ext^{-*}_{\O_{\fZ}}(\oL g^*\bL^{\fX/\fS}, \sI).
$$

Since this final description is independent of the choice $Z \to X \to S$ of resolutions, we deduce that another choice $Z' \to X' \to S'$ of resolutions would give an equivalent space of lifts. Explicitly, we could take   a third sequence $Z'' \to X'' \to S''$, with trivial relative derived Artin hypergroupoids $\pi_Z:Z'' \to Z$, $Z'' \to Z'$  and similarly for $X'',S''$, compatible with the other morphisms. Since $\pi^*_X\bL^{X/S}\simeq \bL^{X''/S''}$, we get $\pi_Z^*g^*\bL^{X/S}\simeq (g'')^*\bL^{X''/S''}$, so  pulling back along $X'' \to X$ gives an equivalence of lifting spaces, and similarly for $X'' \to X'$. 
\end{proof}

\subsection{Deformations of derived stacks}

\begin{proposition}\label{deform1}
Take a Reedy fibration $f:X \to S$ of simplicial cosimplicial affine schemes. If $S \into \tilde{S}$ is a levelwise closed immersion defined by a square-zero ideal $\sI$,  then the obstruction to lifting $f$  to a levelwise flat morphism
$
 \tilde{f}:\tilde{X} \to \tilde{S},
$
with $\tilde{X}\by_{\tilde{S}}S =X$,
lies in 
$$
\H_{-2}\HOM_{cdg\Mod(X)}(N^s\Omega^{X/S}, f^*N^s\sI).
$$
If the obstruction is zero, then the isomorphism class of liftings is (non-canonically) isomorphic to
$$
\H_{-1}\HOM_{cdg\Mod(X)}(N^s\Omega^{X/S}, f^*N^s\sI).
$$
\end{proposition}
\begin{proof}
We adapt the proof of Proposition \ref{defmor1}.

Letting $\bI$ be the Reedy category $\Delta_* \by \Delta^{\op}$, we have seen that the matching maps
$$
X_n^j \to M(\bI)_n^jX\by_{M(\bI)_n^jS }S_n^j
$$
lift closed immersions (so are pro-smooth), and that the latching maps
$$
L(\bI)_n^jX\to X_n^j
$$
are closed immersions. Using the fact that deformations of pro-smooth morphisms of affine schemes are unobstructed and unique, we may  inductively construct a lift 
$$
\tilde{f}: \tilde{X}^*_{\bt} \to \tilde{S}^*_{\bt}
$$
in the category of simplicial almost cosimplicial affine schemes, with $\tilde{f}$ $\bI$-Reedy fibrant, and note that this lift is unique up to isomorphism.
 It therefore remains only to understand how the operation $\pd^0$ on $X$ can deform.

Using the formal smoothness of $\tilde{f}$, we may lift $\pd^{0}_X :X \to\Dec^+_{\op}X$ to an almost cosimplicial map 
$$
\lambda: \tilde{X} \to\Dec^+_{\op}\tilde{X},
$$
satisfying $\tilde{f}\lambda= \pd^{0}\tilde{f}$ and $\sigma^{0}\lambda=\id$.
This will make $\tilde{X}$ into a cosimplicial scheme if and only if $\lambda^2 = \pd^1\lambda$, since all the other conditions are automatic.

Thus the obstruction to $(\tilde{X}^*, \lambda)$ being a cosimplicial scheme is the  map
$$
(\lambda^{\sharp})^2 -\lambda^{\sharp} \pd_1 : O(\tilde{X})_n \to O(\tilde{X})_{n-2},
$$
whose image is contained in $\ker( O(\tilde{X}) \to O(X))=  f^*\sI$. Since $\sI$ is a square-zero ideal, we know that $f^*\sI$ must also map to $0$. Thus we have a   map
$$
O(X)_n \to (f^*\sI)_{n-2}
$$

This map is an $O(S)$-linear $(\pd_0)^2$-derivation, so corresponds  to an  $(O(X),(\pd_{0})^2)$-linear map
$$
\delta: \Omega(X/S)_n \to (f^*\sI)_{n-2}.
$$
Moreover, 
$\pd_i\delta= \delta\pd_{i+2}$ for all $i>0$, 
so $\delta$ descends to a map $N^s\delta$ on the normalised complexes, which will be $0$ if and only if $\delta=0$. Moreover, $\delta\sigma_i = \sigma_{i-2}\delta$ for $i \ge 2$, and $0$ for $i=0,1$. A calculation similar to that in Proposition \ref{defmor1} shows that  $N^s\delta$ is an $N^sO(X)$-linear map.
 
 We need to show that $\delta ( \lambda^{\sharp}- \pd_1 +\pd_2)= \lambda^{\sharp}\delta$ to ensure that $N^s\delta$ is a chain map. 
$$
\delta (\lambda^{\sharp}-\pd_1+\pd_2)= [(\lambda^{\sharp})^3-(\lambda^{\sharp})^2\pd_1+\lambda^{\sharp}\pd_2] -[\lambda^{\sharp}\pd_1\lambda^{\sharp}]= (\lambda^{\sharp})^3-(\lambda^{\sharp})^2\pd_1 = \lambda^{\sharp}\delta,
$$
as required, so $N^s \delta \in \z_{-2}\HOM_{N^sO(X)}(N^s\Omega(X/S),N^sf^*\sI)$.

However, the choice of $\lambda$ was not unique. Another choice  $\lambda'$ gives rise to $\beta:= \lambda^*- (\lambda')^*$. Similar reasoning shows that this is a $\pd^0$-derivation $\beta:\Omega(X/S)\to f^*\sI$, and corresponds to an $N^sO(X)$-linear map on normalised complexes. The operation $\delta'$ is given by $\delta - \pd_0\beta +\beta(\pd_1-\pd_0)$, so $N^s\delta' = \delta- [d, N^s\beta]$. Thus the obstruction lies in 
$$
\H_{-2}\HOM_{N^sO(X)}(N^s\Omega(X/S),N^s(f^*\sI)).
$$
 
To understand the isomorphism class of the deformations, fix a deformation $(\tilde{X}_{\bt}^*,\lambda)$. Another choice of lift is given by $ (\lambda')^*=\lambda^*-\beta$, and for this to be unobstructed, we need $[d, N^s\beta]=0$, so $N^s\beta \in \z_{-1} \HOM_{N^sO(X)}(N^s\Omega(X/S),N^sf^*\sI)$. Now, an isomorphism of $\tilde{X}_{\bt}^*$ is equivalent to a map $\alpha \in \HOM_{N^sO(X)}(N^s\Omega(X/S),N^sf^*\sI)^0$, and transformation by $\alpha$ sends $N^s\beta$ to $N^s\beta +[d,\alpha]$. Thus the isomorphism class is
$$
\H_{-1}\HOM_{N^sO(X)}(N^s\Omega(X/S),N^sf^*\sI).
$$
\end{proof}

\begin{lemma}\label{deform3}
If the morphism $f$ in Proposition \ref{deform1} is a relative derived Artin $m$-hypergroupoid, the simplicial automorphism group $\underline{\Aut}_{\tilde{S}}(\tilde{X})_{X}$ of a deformation $\tilde{X}$ is given by
$$
N^s\underline{\Aut}_{\tilde{S}}(\tilde{X})_{X} \simeq \tau_{\ge 0}\HOM_{cdg\Mod(X)}(\bL^{X/S}, f^*N^s\sI).
$$
\end{lemma}
\begin{proof}
This follows immediately from Corollary \ref{cotgood}.
\end{proof}

\begin{theorem}\label{deformstack}
Take a morphism $f:\fX \to \fS$ of derived $m$-geometric Artin stacks. If $\fS \into \tilde{\fS}$ is a $0$-representable closed immersion defined by a square-zero quasi-coherent complex $\sI$,   then the obstruction to lifting $f$ to a morphism
$
\tilde{f}: \tilde{\fX} \to\tilde{\fS}
$
with $\tilde{\fX}\by^h_{\tilde{\fS}}\fS \simeq \fX$
lies in
$$
\EExt^2_{\O_{\fX}}(\bL^{\fX/\fS}, f^*\sI).
$$
If this obstruction is $0$, then the equivalence class of deformations is
$$
\Ext^1_{\O_{\fX}}(\bL^{\fX/\fS}, f^*\sI).
$$
and the simplicial automorphism group $\underline{\Aut}_{\tilde{S}}(\tilde{X})_{\fX}$ of any deformation $\tilde{\fX}$ is given by
$$
\pi_i\underline{\Aut}_{\tilde{\fS}}(\tilde{\fX})_{\fX} \cong \Ext^{-i}_{\O_{\fX}}(\bL^{\fX/\fS}, f^*\sI).
$$
\end{theorem}
\begin{proof}
Apply Theorem \ref{relstrict} to obtain a derived Artin $m$-hypergroupoid $\tilde{S}$ with $\tilde{S}^{\sharp} \simeq \tilde{\fS}$, and pull back $\sI$ to give a levelwise closed immersion $S \into \tilde{S}$. Similarly, there exists a relative derived Artin $m$-hypergroupoid $X \to S$, with $X^{\sharp}\simeq \fX$.  Proposition \ref{deform1}, Lemma \ref{deform2} and Lemma \ref{deform3} now imply that deformations of $X$ are governed by 
$$
\Ext^{*}_{\O_{\fX}}(\bL^{\fX/\fS}, f^*\sI),
$$
and  sheafification defines a functor from the $\infty$-groupoid of deformations of $X$ to the $\infty$-groupoid of deformations of $\fX$, so we just need to show that this is an equivalence. 

Given a deformation $\tilde{\fX}$ of $\fX$, there exists  a relative derived Artin $m$-hypergroupoid $\tilde{X'} \to \tilde{S}$, with $(\tilde{X'})^{\sharp} \simeq \tilde{\fX}$. Setting  $X':=\tilde{X'}\by_{\tilde{S}}S$ gives  $(X')^{\sharp} \simeq \fX$. By Theorem \ref{duskinmor}, there exists a derived Artin $m$-hypergroupoid $X''$ equipped with morphisms $X'' \to X'$, $X'' \to X$ which are  trivial relative derived Artin $m$-hypergroupoids.

It therefore suffices to show that the simplicial groupoid $\Gamma$ of deformations of a  trivial relative derived Artin $m$-hypergroupoid $T \to Y$  is contractible. By Corollary \ref{trivcot}, $\bL^{T/Y} \simeq 0$, so Proposition \ref{deform1}, Lemma \ref{deform2} and Lemma \ref{deform3} imply that $\Gamma$ has one equivalence class of objects, with contractible automorphism group.
\end{proof}

\begin{remark}
Theorem \ref{deformstack} has an alternative proof, as a Corollary of Theorem \ref{defmorstack}, which we now sketch.

Given a quasi-coherent complex $\sF$ on $\fS$, let $\fS\oplus \sF:= \oSpec (\O_{\fS} \oplus (N^s)^{-1}\sF\eps)$, with $\eps^2=0$. Then there exists a morphism  $t: \fS\oplus\sI[-1] \to \fS$ such that $\tilde{\fS}$ is the homotopy cofibre product $\tilde{\fS}:= \fS\cup_{0, \fS\oplus\sI[-1], t}^h\fS$. The category of deformations $\tilde{\fX}$ over $\tilde{\fS}$ is therefore equivalent to the simplicial category of equivalences $\alpha: t^*\fX \to 0^*\fX$ over $ \fS\oplus\sI[-1]$, fixing the fibre $\fX$ over $\fS$. This is the same as the simplicial category of morphisms $t^*\fX \to \fX$ over $\fS$, fixing $\fX$. Since $\fX \to t^*\fX$ is a square-zero extension, Theorem \ref{defmorstack} describes these morphisms in terms of the cotangent complex. 
\end{remark}

\begin{remark}
If  $\fX$ and $\fS$ are  (non-derived) geometric Artin stacks, then flatness of $\tilde{\fX} \to \tilde{\fS}$ is sufficient to ensure that $\tilde{\fX}\by^h_{\tilde{\fS}}\fS \simeq \tilde{\fX}\by_{\tilde{\fS}}\fS$. In this case, Theorems \ref{defmorstack} and \ref{deformstack} have alternative proofs, by considering the almost simplicial schemes underlying the associated simplicial schemes. It is then necessary to study deformations of $X_0$ and of  $\pd_0^X$, rather than of $\pd^0_X$.
This is the approach taken in \cite{higher} \S \ref{higher-stacksn}.
\end{remark}

\begin{corollary}\label{thickenart}
Given a derived $m$-geometric Artin stack $\fX$, and an Artin $m$-hypergroupoid $Z$ with $Z^{\sharp} \simeq \pi^0\fX$, there exists a derived  Artin $m$-hypergroupoid $X$ with $X^{\sharp} \simeq \fX$ and $\pi^0X = Z$.
\end{corollary}
\begin{proof}
It suffices to find a homotopy derived Artin hypergroupoid $X$  with these properties, since we may then take a Reedy fibrant approximation. Writing $\fX$ as the colimit
$\pi^0\fX \to \pi^{\le 1}\fX \to \pi^{\le 2} \fX \to \ldots$ of homotopy square-zero extensions as in Remark \ref{postnikov}, we will inductively construct $a:X(i)\to \fX$ with $X(i)^{\sharp} \simeq \fX(i):=\pi^{\le i}\fX$ and $\pi^0X(i)= Z$.

Assume that we have constructed $X(i-1)$. We will now construct the schemes $X(i)_n$ inductively, so assume that we have done this compatibly for all $j<n$, giving latching and matching objects
$
L_nX(i), M_nX(i) $. We seek a diagram 
$$
L_nX(i)\xra{\alpha_i} X(i)_n \xra{\beta_i} M_nX(i)\by_{\fX(i)^{\pd \Delta^n}}^h\fX(i)
$$ 
lifting 
$$
L_nX(i-1)\xra{\alpha_{i-1}} X(i-1)_n \xra{\beta_{i-1}} M_nX(i-1)\by_{ \fX(i-1)^{\pd \Delta^n}}^h\fX(i-1).
$$ 
Since the embedding $\fX(i-1) \to \fX(i)$ is defined by the square-zero sheaf $\pi_i\O_{\fX}[-i]$, deformations of $\beta_{i-1}$
governed by the complex
\begin{eqnarray*}
C_{\bt}&:=&\HOM_{dg\Mod(X(i-1)_n)}(\bL^{X(i-1)_n / M_nX(i-1)\by^h_{\fX(i-1)^{\pd \Delta^n}}\fX(i-1)}, a_n^*\pi_i\O_{\fX}[-i])\\
&\cong&\HOM_{dg\Mod(Z_n)}(\bL^{Z_n /M_nZ\by^h_{\pi^0\fX^{\pd \Delta^n}}\pi^0\fX}, a^*_n\pi_i\O_{\fX}[-i]). 
\end{eqnarray*}
Since $a:Z \to \pi^0\fX$ is a resolution, the map  $Z_n \to M_nZ\by^h_{\pi^0\fX^{\pd \Delta^n}}\pi^0\fX$ is a smooth $m$-representable morphism, so Lemma \ref{h0cot} implies that the  cotangent complex $\bL^{Z_n /M_nZ\by^h_{\pi^0\fX^{\pd \Delta^n}}\pi^0\fX}_{\bt}$ is equivalent to a complex of locally projective modules concentrated in degrees $[-m,0]$. Since $Z_n$ is affine, locally projective modules on $Z_n$  are projective, so $\H_j(C_{\bt})=0$ for all $j<i$, and in particular for $j=0,-1,-2$. %%% If challenged on this, note that we can choose a suitable resolution as $\Dec^{n+1}_+Z$, which has a section, so $\oR \pd_{0*}$ is equiv to $\sigma_0^*$. Still possible issues.

Since $i>0$, Theorem \ref{deformstack} then implies that $\beta_i:X(i)_n \to M_nX(i)\by^h_{\fX(i)^{\pd \Delta^n}}\fX(i)$  exists, and is unique up to unique isomorphism in the homotopy category, so we may choose a Reedy fibrant representative. In particular, $\beta_i$  is smooth, and since $L_nX(i-1) \to L_nX(i)$ is a homotopy square-zero extension, Theorem \ref{defmorstack} implies that deformations of $\alpha_{i-1}$ are governed by 
the complex
\begin{eqnarray*}
D_{\bt}&:=&\HOM_{dg\Mod(L_nX(i-1))}(\oL\alpha^*\bL^{X(i-1)_n / M_nX(i-1)\by^h_{\fX(i-1)^{\pd \Delta^n}}\fX(i-1)}, \pi_i\O_{L_nX}[-i])\\
&\cong&\HOM_{dg\Mod(L_nZ)}(\oL\alpha^*\bL^{Z_n /M_nZ\by^h_{\pi^0\fX^{\pd \Delta^n}}\pi^0\fX}, \pi_i\O_{L_nX}[-i]).
\end{eqnarray*}
As for $C_{\bt}$, we deduce that $\H_j(D_{\bt})=0$ for all $j<i$. In particular, this holds for $j=0,-1$, so $\alpha_i$ exists, and is unique up to weak equivalence.

This completes the inductive step for constructing $X(i)$, and we set $X= \varinjlim X(i)$.
\end{proof}

\bibliographystyle{alphanum}
%\addcontentsline{toc}{section}{Bibliography}
\bibliography{references}
\end{document}